\documentclass[11pt]{amsart}
\usepackage{amssymb,amsthm}
\usepackage{amsmath}
\usepackage{mathrsfs}
\usepackage{xspace}
\RequirePackage[dvipsnames,usenames]{xcolor}
\usepackage[all
]{xy}
\usepackage{tikz}
\usepackage{hyperref}
\usepackage{comment}
\usepackage
[left=1.02in,top=1.0in,right=1.02in,bottom=1.0in]
{geometry}
\usepackage{mabliautoref}

\hypersetup{
bookmarks,
bookmarksdepth=3,
bookmarksopen,
bookmarksnumbered,
pdfstartview=FitH,
colorlinks,backref,hyperindex,
linkcolor=Sepia,
anchorcolor=BurntOrange,
citecolor=MidnightBlue,
citecolor=OliveGreen,
filecolor=BlueViolet,
menucolor=Yellow,
urlcolor=OliveGreen
}

\theoremstyle{remark}

\DeclareMathAlphabet{\mathchanc}{OT1}{pzc}%
                                 {m}{it}

\newcommand{\bN}{\mathbb{N}}

\newcommand{\bP}{\mathbb{P}}
\newcommand{\bQ}{\mathbb{Q}}

\newcommand{\bZ}{\mathbb{Z}}

\newcommand{\bfP}{\mathbf{P}}

\newcommand{\scr}{\mathscr}
\newcommand{\sA}{\scr{A}}
\newcommand{\sB}{\scr{B}}

\newcommand{\sD}{\scr{D}}
\newcommand{\sE}{\scr{E}}
\newcommand{\sF}{\scr{F}}
\newcommand{\sG}{\scr{G}}
\newcommand{\sH}{\scr{H}}
\newcommand{\sI}{\scr{I}}
\newcommand{\sJ}{\scr{J}}

\newcommand{\sL}{\scr{L}}
\newcommand{\sM}{\scr{M}}
\newcommand{\sN}{\scr{N}}
\newcommand{\sO}{\scr{O}}

\newcommand{\sT}{\scr{T}}
\newcommand{\sU}{\scr{U}}

\newcommand{\sX}{\scr{X}}

\newcommand{\sZ}{\scr{Z}}

\newcommand{\ol}[1]{\overline{#1}}

\newcommand{\of}{\overline{f}}

\newcommand{\oA}{\overline{A}}
\newcommand{\oB}{\overline{B}}

\newcommand{\oD}{\overline{D}}
\newcommand{\oE}{\overline{E}}

\newcommand{\oH}{\overline{H}}

\newcommand{\oT}{\overline{T}}

\newcommand{\oX}{\overline{X}}
\newcommand{\oY}{\overline{Y}}
\newcommand{\oZ}{\overline{Z}}

\DeclareMathOperator{\vol}{{vol}}

\DeclareMathOperator{\lct}{{lct}}

\DeclareMathOperator{\can}{{can}}

\DeclareMathOperator{\Sym}{{Sym}}

\DeclareMathOperator{\Aut}{Aut}

\DeclareMathOperator{\codim}{codim}

\DeclareMathOperator{\depth}{{depth}}

\DeclareMathOperator{\Gr}{{Gr}}

\DeclareMathOperator{\Hom}{Hom}

\newcommand{\sfIsom}[0]{{\sf Isom}}

\DeclareMathOperator{\id}{{id}}

\DeclareMathOperator{\im}{{im}}

\DeclareMathOperator{\Isom}{Isom}

\DeclareMathOperator{\Mor}{{Mor}}
\DeclareMathOperator{\mult}{mult}

\DeclareMathOperator{\Proj}{{Proj}}

\DeclareMathOperator{\red}{red}

\DeclareMathOperator{\rk}{{rk}}

\DeclareMathOperator{\Spec}{{Spec}}
\DeclareMathOperator{\coeff}{{coeff}}

\DeclareMathOperator{\Supp}{{Supp}}

\DeclareMathOperator{\sym}{{Sym}}

\DeclareMathOperator{\var}{{Var}}

\newcommand{\factor}[2]{\left. \raise 2pt\hbox{\ensuremath{#1}} \right/
        \hskip -2pt\raise -2pt\hbox{\ensuremath{#2}}}

\makeatletter
\renewcommand\subsubsection{
  \renewcommand{\sfdefault}{pag}
  \@startsection{subsubsection}%
  {2}{0pt}{.4\baselineskip}{0\baselineskip}{\raggedright
    \itshape
  }}
\renewcommand\subsection{
  \renewcommand{\sfdefault}{pag}
  \@startsection{subsection}%
  {2}{0pt}{.8\baselineskip}{0.2\baselineskip}{\raggedright
    \sffamily\itshape\small\bfseries 
  }}
\renewcommand\section{
  \renewcommand{\sfdefault}{phv}
  \@startsection{section} %
  {1}{0pt}{\baselineskip}{.8\baselineskip}{\centering
    \sffamily
    \scshape
    \bfseries
  }}
\makeatother

\renewcommand{\theenumi}{\arabic{enumi}}

\usepackage{scalerel}
\usepackage{tikz,needspace,paralist}
\usetikzlibrary{decorations.pathmorphing}
\usetikzlibrary{shapes.geometric} 
\DeclareSymbolFont{largesymbolsA}{U}{jkpexa}{m}{n}
\SetSymbolFont{largesymbolsA}{bold}{U}{jkpexa}{bx}{n}
\DeclareMathSymbol{\varprod}{\mathop}{largesymbolsA}{16}
\usepackage[shortlabels]{enumitem}
\newcommand{\kdot}{{{\,\begin{picture}(1,1)(-1,-2)\circle*{2}\end{picture}\
    }}} 
\newcommand{\into}{\hookrightarrow}

\newcommand{\wt}{\widetilde}

\newcommand{\otau}{\overline{\tau}}
\newcommand{\ophi}{\overline{\phi}}
\newcommand\mtimes[3]{{\varprod_{#1}^{#2}}_{\raise 1ex
    \hbox{\scriptsize #3}}} 
\numberwithin{equation}{theorem}

\DeclareMathOperator{\GL}{GL}
\DeclareMathOperator{\Mat}{Mat}

\author{S\'andor J Kov\'acs}%
\address{\hskip-1em University of Washington, Department of Mathematics, Box 354350
  Seattle, WA 98195-4350, USA}%
\email{skovacs@uw.edu}
\author{Zsolt Patakfalvi}%
\address{\hskip-1em Department of Mathematics, Princeton University, Fine Hall,
  Washington Road, NJ 08544-1000, USA}%
\email{pzs@math.princeton.edu}

\title[Projectivity of moduli and subadditvity of log-Kodaira dimension]{Projectivity
  of the moduli space of stable log-varieties and subadditvity of log-Kodaira
  dimension}

\begin{document}

\begin{abstract}
  We prove a strengthening of Koll\'ar's Ampleness Lemma and use it to prove that any
  proper coarse moduli space of stable log-varieties of general type is
  projective. We also prove subadditivity of log-Kodaira dimension for fiber spaces
  whose general fiber is of log general type.
\end{abstract}

\maketitle

\setcounter{tocdepth}{1}
\tableofcontents

\section{Introduction}
\label{sec:introduction}

Since Mumford's seminal work on the subject, $\sM_g$, the moduli space of smooth
projective curves of genus $g \geq 2$, has occupied a central place in algebraic
geometry and the study of $\sM_g$ has yielded numerous applications. An important
aspect of the applicability of the theory is that these moduli spaces are naturally
contained as open sets in $\ol\sM_g$ the moduli space of stable curves of genus $g$
and the fact that this later space admits a projective coarse moduli scheme.

Even more applications stem from the generalization of this moduli space,
$\sM_{g,n}$, the moduli space of $n$-pointed smooth projective curves of genus $g$
and its projective compactification, $\ol\sM_{g,n}$, the moduli space of $n$-pointed
stable curves of genus $g$.

It is no surprise that after the success of the moduli theory of curves huge efforts
were devoted to develop a similar theory for higher dimensional varieties. However,
the methods used in the curve case, most notably GIT, proved inadequate for the
higher dimensional case. Gieseker \cite{MR0498596} proved that the moduli space of
smooth projective surfaces of general type is quasi-projective, but the proof did not
provide a modular projective compactification. In fact, Wang and Xu showed recently
that such GIT compactification using asymptotic Chow stability is impossible
\cite{Wang_Xu_Nonexistence_of_asymptotic_GIT_compactification}.  The right definition
of stable surfaces only emerged after the development of the minimal model program
allowed bypassing the GIT approach
\cite{Kollar_Shepher_Barron_Threefolds_and_deformations} and the existence and
projectivity of the moduli space of stable surfaces and higher dimensional varieties
have only been proved very recently as the combined result of the effort of several
people over several years
\cite{Kollar_Shepher_Barron_Threefolds_and_deformations,Kollar_Projectivity_of_complete_moduli,Alexeev_Boundedness_and_K_2_for_log_surfaces,Viehweg_Quasi_projective_moduli,Hassett_Kovacs_Reflexive_pull_backs,Abramovich_Hassett_Stable_varieties_with_a_twist,Kollar_Hulls_and_Husks,Kollar_Moduli_of_varieties_of_general_type,Kollar_Singularities_of_the_minimal_model_program,Fujino_Semi_positivity_theorems_for_moduli_problems,Hacon_McKernan_Xu_Boundedness_of_moduli_of_varieties_of_general_type,Kollar_Second_moduli_book}.

Naturally, one would also like a higher dimensional analogue of $n$-pointed curves
and extend the existing results to that case
\cite{Alexeev_Moduli_spaces_M_g_n_W_for_surfaces}. The obvious analogue of an
$n$-pointed smooth projective curve is a smooth projective log-variety, that is, a
pair $(X,D)$ consisting of a smooth projective variety $X$ and a simple normal
crossing divisor $D\subseteq X$. For reasons originating in the minimal model theory
of higher dimensional varieties, one would also like to allow some mild singularities
of $X$ and $D$ and fractional coefficients in $D$, but we will defer the discussion
of the precise definition to a later point in the paper (see
\autoref{def:stable_log_variety}).  In any case, one should mention that the
introduction of fractional coefficients for higher dimensional pairs led Hassett to
go back to the case of $n$-pointed curves and study a weighted version in
\cite{Hassett_Moduli_spaces_of_weighted_pointed_stable_curves}. These moduli spaces
are more numerous and have greater flexibility than the traditional ones. In fact,
they admit natural birational transformations and demonstrate the workings of the
minimal model program in concrete highly non-trivial examples. Furthermore, the log
canonical models of these moduli spaces of weighted stable curves may be considered
to approximate the canonical model of $\ol\sM_{g,n}$
\cite{HB_HD_LMM_TFDC,HB_HD_LMM_TFF}.

It turns out that the theory of \emph{moduli of stable log-varieties}, also known as
\emph{moduli of semi-log canonical models} or \emph{KSBA stable pairs}, which may be
regarded as the higher dimensional analogues of Hassett's moduli spaces above, is
still very much in the making. It is clear what a stable log-variety should be: the
correct class (for surfaces) was identified in
\cite{Kollar_Shepher_Barron_Threefolds_and_deformations} and further developed in
\cite{Alexeev_Moduli_spaces_M_g_n_W_for_surfaces}. This notion, is easy to generalize
to arbitrary dimension cf.\ \cite{Kollar_Moduli_of_varieties_of_general_type}.  On
the other hand, at the time of the writing of this article it is not entirely obvious
what the right definition of the corresponding moduli functor is over non reduced
bases. For a discussion of this issue we refer to
\cite[\S6]{Kollar_Moduli_of_varieties_of_general_type}.  A major difficulty is that
in higher dimensions when the coefficients of $D$ are not all greater than $1/2$ a
deformation of a log-variety cannot be simplified to studying deformations of the
ambient variety $X$ and then deformations of the divisor $D$.  An example of this
phenomenon, due to Hassett, is presented in \autoref{ex:Hassett}, where a family
$(X,D) \to \bP^1$ of stable log varieties is given such that $D \to \bP^1$ does not
form a flat family of pure codimension one subvarieties. In fact, the flat limit
$D_0$ acquires an embedded point, or equivalently, the scheme theoretic restriction
of $D$ onto a fiber is not equal to the divisorial restriction. Therefore, in the
moduli functor of stable log-varieties one should allow both deformations that
acquire and also ones that do not acquire embedded points on the boundary
divisors. This is easy to phrase over nice (e.g., normal) bases see
\autoref{def:a_functor} for details. However, at this point it is not completely
clarified how it should be presented in more intricate cases, such as for instance
over a non-reduced base. Loosely speaking the infinitesimal structure of the moduli
space is not determined yet (see \autoref{rem:m_ambiguity} for a discussion on this),
although there are also issues about the implementation of labels or markings on the
components of the boundary divisor (cf. \autoref{rem:labeling}).


By the above reasons, several functors have been suggested, but none of them yet
emerged as the obvious ``best''. However, our results apply to any moduli functor for
which the objects are the stable log-varieties (see \autoref{def:a_functor} for the
precise condition on the functors). In particular, our results apply to any moduli
space that is sometimes called a \emph{KSBA compactification} of the moduli space of
log-canonical models.

Our main result is the following.  Throughout the article we are working over an
algebraically closed base field $k$ of characteristic zero.
\begin{theorem}[=\autoref{cor:projective}]
\label{thm:main}
  Any algebraic space that is the coarse moduli space of a moduli functor of stable
  log-varieties with fixed volume, dimension and coefficient set (as defined in
  \autoref{def:a_functor}) is a projective variety over $k$.
\end{theorem}

For auxiliary use, in \autoref{sec:particular_functor} we also present one particular
functor as above, based on a functor suggested by Koll\'ar
\cite[\S6]{Kollar_Moduli_of_varieties_of_general_type}. In particular, the above
result is not vacuous.

%
%
As mentioned Mumford's GIT method used in the case of moduli of stable curves does
not work in higher dimensions and so we study the question of projectivity in a
different manner. The properness of any algebraic space as in \autoref{thm:main} is
shown in \cite{Kollar_Second_moduli_book}.  For the precise statement see
\autoref{prop:proper}. Hence, to prove projectivity over $k$ one only has to exhibit
an ample line bundle on any such algebraic space.  Variants of this approach have
been already used in
\cite{Knudsen_The_projectivity_of_the_moduli_space_of_stable_curves_III,
  Kollar_Projectivity_of_complete_moduli,
  Hassett_Moduli_spaces_of_weighted_pointed_stable_curves}.  Generalizing Koll\'ar's
method to our setting \cite{Kollar_Projectivity_of_complete_moduli}, we use the
polarizing line bundle $\det f_* \sO_X(r (K_{X/Y} + D))$, where $f:(X,D)\to Y$ is a
stable family and $r>0$ is a divisible enough integer. Following Koll\'ar's idea and
using the Nakai-Moishezon criterion it is enough to prove that this line bundle is
big for a maximal variation family over a normal base. However, Koll\'ar's Ampleness
Lemma \cite[3.9,3.13]{Kollar_Projectivity_of_complete_moduli} is unfortunately not
strong enough for our purposes and hence we prove a stronger version in
\autoref{thm:generalized_ampleness_lemma}. There, we also manage to drop an
inconvenient condition on the stabilizers from
\cite[3.9,3.13]{Kollar_Projectivity_of_complete_moduli}, which is not necessary for
the current application, but we hope will be useful in the future.  Applying
\autoref{thm:generalized_ampleness_lemma} and some other arguments outlined in
\autoref{sec:outline} we prove that the above line bundle is big in
\autoref{thm:big_higher_dim_base}.

A side benefit of this approach is that proving a positivity property of $K_{X/Y} +
D$ opens the doorq to other applications. For example, a related problem in the
classification theory of algebraic varieties is the subadditivity of log-Kodaira
dimension.  We prove this assuming the general fiber is of log general type in
\autoref{thm:subadditivity}.  This generalizes to the logarithmic case the celebrated
results on the subadditivity of Kodaira dimension
\cite{Kawamata_Characterization_of_abelian_varieties,Kawamata_Minimal_models_and_the_Kodaira_dimension_of_agebraic_fiber_spaces,Viehweg_Weak_positivity,Viehweg_Weak_positivity_and_the_additivity_of_the_Kodaira_dimension_II_The_local_Torelli_map,Kollar_Subadditivity_of_the_Kodaira_dimension},
also known as Iitaka's conjecture $C_{n,m}$ and its strengthening by Viehweg
$C^+_{n,m}$. For \autoref{thm:subadditivity} we refer to \autoref{sec:subadditivity},
here we only state two corollaries that need less preparation.

\begin{theorem} (= \autoref{thm:to_subadditivity_for_q_proj_varieties} and
  \autoref{cor:subadditivty_q_proj})
  \label{thm:main_subadd}
  \begin{enumerate}
  \item If $f : (X, D) \to (Y,E) $ is a surjective map of log-smooth projective pairs
    with coefficients at most $1$, such that $D \geq f^*E$ and $K_{X_\eta} +
    D_{\eta}$ is big, where $\eta$ is the generic point of $Y$, then
    \begin{equation*}
      \kappa(K_X + D) \geq \kappa\left( K_{X_{\eta}} + D_{\eta} \right) + \kappa (K_Y + E).
    \end{equation*}
  \item Let $f : X \to Y $ be a dominant map of (not necessarily proper) algebraic
    varieties such that the generic fiber has maximal Kodaira dimension. Then
    \begin{equation*}
      \kappa(X) \geq \kappa\left(X_\eta \right)  + \kappa(Y).
    \end{equation*}
  \end{enumerate}
\end{theorem}

In the logarithmic case Fujino obtained results similar to the above in the case of
maximal Kodaira dimensional base \cite[Thm
1.7]{Fujino_Notes_on_the_weak_positivity_theorems} and relative one dimensional
families
\cite{Fujino_Subadditivity_of_the_logarithmic_Kodaira_dimension_for_morphisms_of_relative_dimension_one_revisited}. Another
related result of Fujino is subadditivity of the \emph{numerical} log-Kodaira
dimension \cite{Fujino_On_subadditivity_of_the_logarithmic_Kodaira_dimension}. A
version of the latter, under some additional assumptions, was also proved by Nakayama
\cite[V.4.1]{NN_ZDA}.  The numerical log-Kodaira dimension is expected to be equal to
the usual log-Kodaira dimension by the Abundance Conjecture. However, the latter is
usually considered the most difficult open problem in birational geometry currently.
Our proof does not use either the Abundance Conjecture or the notion of numerical
log-Kodaira dimension.

Further note that our proof of \autoref{thm:main_subadd} is primarily algebraic. That
is, we obtain our positivity results, from which \autoref{thm:main_subadd} is
deduced, algebraically, starting from the semi-positivity results of Fujino
\cite{Fujino_Semi_positivity_theorems_for_moduli_problems,Fujino_Notes_on_the_weak_positivity_theorems}. Hence,
our approach has a good chance to be portable to positive characteristic when the
appropriate semi-positivity results (and other ingredients such as the mmp) become
available in that setting. See
\cite{Patakfalvi_Semi_positivity_in_positive_characteristics} for the currently
available semi-positivity results in positive characteristic, and
\cite{Chen_Zhang_The_subadditivity_of_the_Kodaira-dimension_for_fibrations_of_relative_dimension_one_in_positive_characteristics,Patakfalvi_On_subadditivity_of_Kodaira_dimension_in_positive_characteristic}
for results on subadditivty of Kodaira-dimension.

\autoref{thm:main_subadd} is based on the following theorem stating that the sheaves
$f_* \sO_X(r (K_{X/Y} + \Delta))$ have more positivity properties than just that
their determinants are ample. This is a generalization of
\cite{Kollar_Subadditivity_of_the_Kodaira_dimension} and \cite[Thm
3.1]{Esnault_Viehweg_Effective_bounds_for_semipositive_sheaves_and_for_the_height_of_points_on_curves_over_complex_function_fields}
to the logarithmic case.

\begin{theorem} (=\autoref{thm:pushforward})
\label{thm:main_big}
 If $f : (X,D) \to Y$ is a family of stable log-varieties of maximal variation over
  a normal, projective variety $Y$ with klt general fiber, then $f_*
 \sO_X(r(K_{X/Y}+D))$ is big for every divisible enough integer $r>0$.
\end{theorem}

Note that \autoref{thm:main_big} fails without the klt assumption.  Also,
\autoref{thm:main_big} allows for numerous applications, such as, the already
mentioned \autoref{thm:main_subadd}, as well as upcoming work in progress on a
log-version of
\cite{Abramovich_A_high_fibered_power_of_a_family_of_varieties_of_general_type_dominates_a_variety_of_general_type}
in \cite{AKPT} and on the ampleness of the CM line bundle on the moduli space of
stable varieties in
\cite{Patakfalvi_Xu_Projectivity_of_CM_line_bundle_on_moduli_space_of_canonically_polarized_varieties}. We
also state \autoref{thm:main_big} and our other positivity results over almost
projective bases in \autoref{sec:almost_results}, that is, over bases that are big
open sets in projective varieties. We hope this will be helpful for some
applications.

%
%
%
%
%
%
%
%
%
%
%
%
%

\subsection{Outline of the proof}
\label{sec:outline}

As mentioned above, using the Nakai-Moishezon criterion for ampleness,
\autoref{thm:main} reduces to the following statement (=
\autoref{prop:big_higher_dim_base}): given a family of stable log-varieties $f :
(X,D) \to Y$ with maximal variation over a smooth, projective variety, $\det f_*
\sO_X(q(K_{X/Y} + D))$ is big for every divisible enough integer $q>0$. This follows
relatively easily from the bigness of $K_{X/Y} + D$. To be precise it also follows
from the bigness of the log canonical divisor $K_{X^{(r)}/Y} +D_{X^{(r)}}$ of some
large enough fiber power for some integer $r >0$ (see \autoref{notation:product} and
the proof of \autoref{prop:big_higher_dim_base}). In fact, one cannot expect to do
better for higher dimensional bases, see \autoref{rem:fiber_power_necessary} for
details. Here we review the proof of the bigness of these relative canonical
divisors, going from the simpler cases to the harder ones.

\subsubsection{The case of $\dim Y=1$ and $\dim X=2$.}
\label{sec:outline_1_1}\ 
In this situation, roughly speaking, we have a family of weighted stable curves as
defined by Hassett \cite{Hassett_Moduli_spaces_of_weighted_pointed_stable_curves}.
The only difference is that in our notion of a family of stable varieties there is no
marking (that is, the points are not ordered). This means that the marked points are
allowed to form not only sections but multisections as well. However, over a finite
cover of $Y$ these multisections become unions of sections, and hence we may indeed
assume that we have a family of weighted stable curves. Denote by $s_i: Y \to X$ $(1,
\dots, m)$ the sections given by the marking and let $D_i$ be the images of these
sections. Hassett proved projectivity \cite[Thm 2.1, Prop
3.9]{Hassett_Moduli_spaces_of_weighted_pointed_stable_curves} by showing that the
following line bundle is ample:
\begin{equation}
  \label{eq:line_bundle_Hassett}
  \det f_* \sO_X(r(K_{X/Y} + D)) \otimes \left( \bigotimes_{i=1}^m s_i^*
    \sO_X(r(K_{X/Y} + D)) \right) .
\end{equation}
Unfortunately, this approach does not work for higher dimensional fibers, because
according to the example of \autoref{ex:Hassett}, the sheaves corresponding to $s_i^*
\sO_X(r(K_{X/Y} + D))$ which is the same as $\left(f|_{D_i}\right)_*
\sO_{D_i}(r(K_{X/Y} + D)|_{D_i})$ are not functorial in higher dimensions. In fact,
the function $y \mapsto h^0\left((D_i)_y, \sO_{D_i}(r(K_{X/Y} + D)|_{D_i}) \right)$
jumps down in the limit in the case of example of \autoref{ex:Hassett}, which means
that there is no possibility to collect the corresponding space of sections on the
fibers into a pushforward sheaf. Note that here it is important that $(D_i)_y$ means
the divisorial restriction of $D_i$ onto $X_y$. Indeed, with the scheme theoretic
restriction there would be no jumping down, since $D_i$ is flat as a scheme over
$Y$. However, the scheme theoretic restriction of $D_i$ onto $X_y$ contains an
embedded point and therefore the space of sections on the divisorial restriction is
one less dimensional than on the scheme theoretic restriction.

So, the idea is to try to prove the ampleness of $\det f_* \sO_X(r(K_{X/Y} + D))$ in
the setup of the previous paragraph, hoping that that argument would generalize to
higher dimensions. Assume that $\det f_* \sO_X(r(K_{X/Y} + D))$ is not ample. Then by
the ampleness of \autoref{eq:line_bundle_Hassett}, for some $1 \leq i \leq m$, $s_i^*
\sO_X(r(K_{X/Y} + D))$ must be ample. Therefore, for this value of $i$, $D_i \cdot
(K_{X/Y} +D )>0$. Furthermore, by decreasing the coefficients slightly, the family is
still a family of weighted stable curves. Hence $K_{X/Y} + D - \varepsilon D_i$ is
nef for every $0 \leq \varepsilon \ll 1$ (see \autoref{lem:pushforward_nef}, although
this has been known by other methods for curves). Putting these two facts together
yields that
\begin{equation*}
  (K_{X/Y} + D)^2 = \underbrace{(K_{X/Y} + D) \cdot (K_{X/Y} + D - \varepsilon D_i
    )}_{\geq 0 \textrm{, because } K_{X/Y} + D \textrm{ and } K_{X/Y} + D - \varepsilon D_i \textrm{ are nef}}  +
  \underbrace{(K_{X/Y} + D) \cdot \varepsilon D_i}_{>0} >0  .
\end{equation*}
This proves the bigness of $K_{X/Y} + D$, and the argument indeed generalizes to
higher dimensions as explained below.

\subsubsection{The case of $\dim Y=1$ and arbitrary $\dim X$.}\ 
Let $f : (X,D) \to Y$ be an arbitrary family of non-isotrivial stable log-varieties
over a smooth projective curve.  Let $D_i$ ($i=1,\dots,m$) be the union of the
divisors (with reduced structure) of the same coefficient (cf.\
\autoref{def:D_c}). The argument in the previous case suggests that the key is to
obtain an inequality of the form
\begin{equation}
  \label{eq:goal_base_dim_one}
  \left( (K_{X/Y} + D)|_{D_i}\right)^{\dim D_i} >0.
\end{equation}
Note that it is considerably harder to reach the same conclusion from this
inequality, than in the previous case, because the $D_i$ are not necessarily
$\bQ$-factorial and then $(X, D - \varepsilon D_i)$ might not be a stable family. To
remedy this issue we pass to a $\bQ$-factorial dlt-blowup. For details see
\autoref{lem:long}.

Let us now turn to how one might obtain \autoref{eq:goal_base_dim_one}.  First, we
prove using our generalization (see \autoref{thm:generalized_ampleness_lemma}) of the
Ampleness Lemma a higher dimensional analogue of \autoref{eq:line_bundle_Hassett} in
\autoref{prop:big_downstairs}, namely, that the following line bundle is ample:
\begin{equation}
  \label{eq:line_bundle_us}
  \det f_* \sO_X(r(K_{X/Y} + D)) \otimes \left(\bigotimes_{i=1}^m \det \left(f|_{D_i}
    \right)_* \sO_{D_i}(r(K_{X/Y} + D)|_{D_i}) \right). 
\end{equation}
The main difference compared to \autoref{eq:line_bundle_Hassett} is that $f|_{D_i}$
is no longer an isomorphism between $D_i$ and $Y$ as it was in the previous case
where the $D_i$ were sections. In fact, $D_i \to Y$ has positive dimensional fibers
and hence $\sE_i:=\left(f|_{D_i} \right)_* \sO_{D_i}(r(K_{X/Y} + D)|_{D_i})$ is a
vector bundle of higher rank. As before, if $\det f_* \sO_X(r(K_{X/Y} + D))$ is not
ample, then for some $i$, $\det \sE_i$ has to be. However, since $\sE_i$ is higher
rank now, it is not as easy to obtain intersection theoretic information as earlier.

As a result one has to utilize a classic trick of Viehweg which leads to working with
fibered powers. Viehweg's trick is using the fact that there is an inclusion
\begin{equation}
  \label{eq:embeddint_E_i}
  \xymatrix{%
    \det \sE_i\ \ar@{^(->}[r] & \displaystyle\bigotimes_{j=1}^{d} \sE_i, 
  }
\end{equation}
where $d:= \rk \sE_i$, and where the latter sheaf can be identified with a
pushforward from the fiber product space $D_i^{(d)} \to Y$ (see
\autoref{notation:product}).  This way one obtains that
\begin{equation*}
  \left(\left.\left(K_{X^{(d)}/Y} + D_{X^{(d)}}\right) \right|_{D_i^{(d)}}
  \right)^{\dim     D_i^{(d)}} >0 ,
\end{equation*}
from which it is an easy 
computation to prove \autoref{eq:goal_base_dim_one}

\subsubsection{The case of both $\dim Y$ and $\dim X$ arbitrary.}\ 
We only mention briefly what goes wrong here compared to the previous case, and what
the solution is. The argument is very similar to the previous case until we show that
\autoref{eq:line_bundle_us} is big. However, it is no longer true that if $\det f_*
\sO_X(r(K_{X/Y} + D))$ is not big, then one of the $\det \sE_i$ is big. So, the
solution is to treat all the sheaves at once via an embedding as in
\autoref{eq:embeddint_E_i} of the whole sheaf from \autoref{eq:line_bundle_us} into a
tensor-product sheaf that can be identified with a pushforward from an appropriate
fiber product (see \autoref{eq:det_to_tensor}).  The downside of this approach is
that one then has to work on $X^{(l)}$ for some big $l$, but we still obtain an
equation of the type \autoref{eq:goal_base_dim_one}, although with $D_i$ replaced
with a somewhat cumbersome subvariety of fiber product type.

After that an enhanced version of the previous arguments yields that $K_{X^{(l)}/Y} +
D_{X^{(l)}}$ is big on at least one component, which is enough for our purposes. In
fact, in this case
we cannot expect that $K_{X/Y} +D$ would be big on any particular component, cf.\
\autoref{rem:fiber_power_necessary}. However, the bigness of $K_{X^{(l)}/Y} +
D^{(l)}$ on a component already implies the bigness of $\det f_* \sO_X(r(K_{X/Y} +
D))$ (see \autoref{prop:big_higher_dim_base}). This argument is worked out in
\autoref{sec:det}.

\subsubsection{Subadditivity of log-Kodaira dimension}\ 
First we prove \autoref{thm:main_big} in \autoref{sec:pushforward} using ideas
originating in the works of Viehweg. This implies that although in \autoref{sec:det}
we were not able to prove the bigness of $K_{X/Y} + D$ (only of $K_{X^{(l)}/Y} +
D^{(l)}$), it actually does hold for stable families of maximal variation with klt
general fibers (cf.\ \autoref{cor:big_total_space}). Then with a comparison process
(see the proof of \autoref{prop:subadditivity}) of an arbitrary log-fiber space $f' :
(X',D') \to Y'$ and of the image in moduli of the log-canonical model of its generic
fiber, we are able to obtain enough positivity of $K_{X'/Y'} + D'$ to deduce
subadditivity of log-Kodaira dimension if the log-canonical divisor of the general
fiber is big.

\subsection{An important example}
\label{ex:Hassett}
The following example is due to Hassett (cf.\
\cite[Example~42]{Kollar_Moduli_of_varieties_of_general_type}), and has been
referenced at a couple of places in the introduction.
  \begin{center}
    \includegraphics[scale=.5]{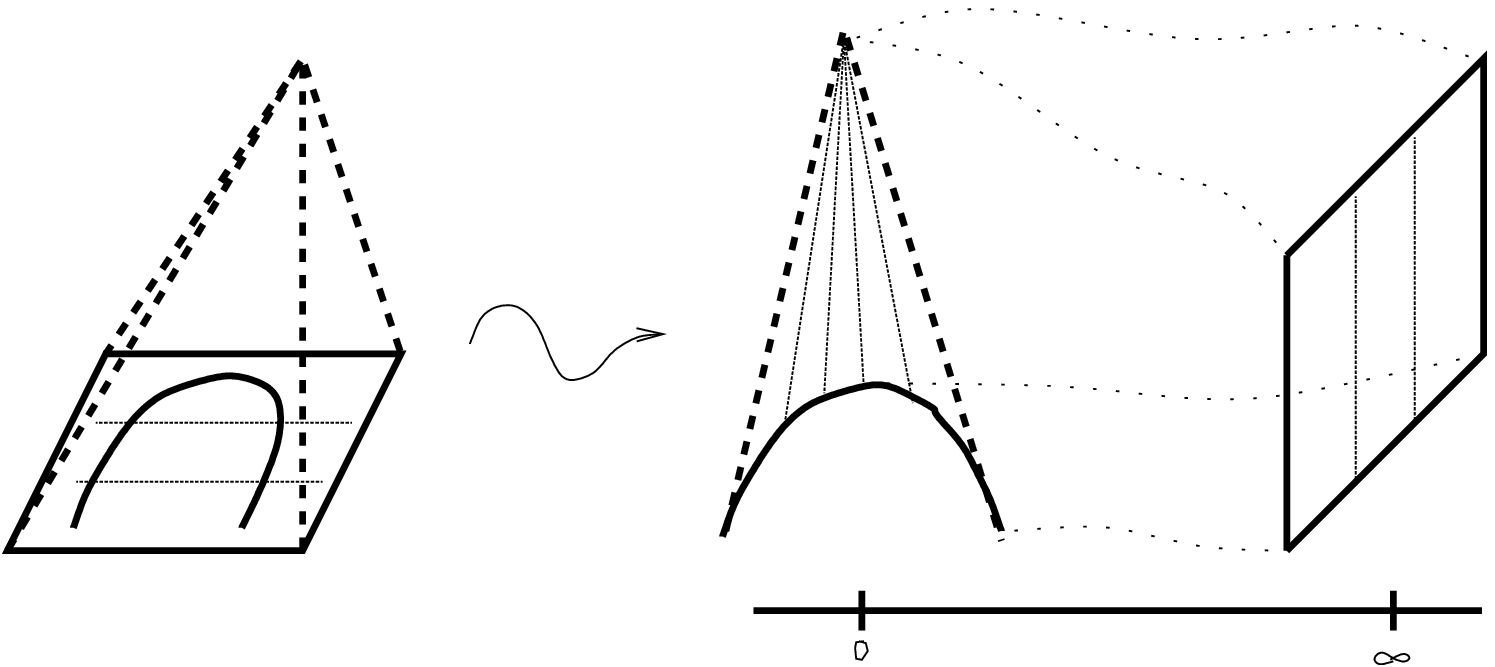} 
  \end{center}
  Let $\sX$ be the cone over $\bP^1 \times \bP^1$ with polarization $\sO_{\bP^1
    \times \bP^1}(2,1)$ and let $\sD$ be the conic divisor $\frac{1}{2} p_2^* P +
  \frac{1}{2} p_2^* Q$, where $p_2 : \bP^1 \times \bP^1 \to \bP^1$ is the projection
  to the second factor, and $P$ and $Q$ are general points. Let $H_0$ be a cone over
  a hyperplane section $C$ of $\bP^1 \times \bP^1$ with the given polarization, and
  $H_\infty$ a general hyperplane section of $\sX$ (which is isomorphic to $\bP^1
  \times \bP^1$). Note that since $\deg \sO_{\bP^1 \times \bP^1}(2,1)|_C = 4$, $H_0$
  is a cone over a rational normal curve of degree $4$. Let $f : \sH \to \bP^1$ be
  the pencil of $H_0$ and $H_\infty$. It is naturally a subscheme of the blowup
  $\sX'$ of $\sX$ along $H_0 \cap H_\infty$. Furthermore, the pullback of $\sD$ to
  $\sX'$ induces a divisor $\sD'$ on $\sH$, such that
  \begin{enumerate}
  \item its reduced fiber over $0$ is a cone over the intersection of $\frac{1}{2}
    p_2^* P + \frac{1}{2} p_2^* Q$ with $C$, that is, over $4$ distinct points on
    $\bP^1$ with coefficients $\frac{1}{2}$, and
  \item its fiber over $\infty$ is two members of one of the rulings of $\bP^1 \times
    \bP^1$ with coefficients $\frac{1}{2}$. In the limit both of these lines
    degenerate to a singular conic, and they are glued together at their singular
    points.
  \end{enumerate}
  In case the reader is wondering how this is relevant to stable log-varieties of
  general type, we note that this is actually a local model of a degeneration of
  stable log-varieties, but one can globalize it by taking a cyclic cover branched
  over a large enough degree general hyperplane section of $\sX$.  For us only the
  local behaviour matters, so we will stick to the above setup. Note that since
  $\chi(\sO_{\sD_\infty'})=2$, the above described reduced structure cannot agree
  with the scheme theoretic restriction $\sD_{0,\mathrm{sch}}'$ of $\sD'$ over $0$,
  since then $\chi(\sO_{\sD_{0,\mathrm{sch}}'})=1$ would hold. Therefore
  $\sD_{0,\mathrm{sch}}'$ is non-reduced at the cone point. Furthermore, note that
  the log canonical divisor of $(\sX, \sD)$ is the cone over a divisor corresponding
  to $\sO_{\bP^1 \times \bP^1}\left(-2 +2 , -2 +1 + \frac{1}{2} + \frac{1}{2} \right)
  \simeq \sO_{\bP^1 \times \bP^1}$. In particular, this log canonical class is
  $\bQ$-Cartier, and hence $(\sH, \sD')$ does yield a local model of a degeneration
  of stable log-varieties.

\subsection{Organization}

We introduce the basic notions on general and on almost proper varieties in
\autoref{sec:basic} and \autoref{sec:almost}. In \autoref{sec:bigness_lemma} we state
our version of the Ampleness Lemma. In \autoref{sec:moduli} we define moduli functors
of stable log-varieties and we also give an example of a concrete moduli functor for
auxiliary use. \autoref{sec:det} contains the proof of \autoref{thm:main} as well as
of the necessary positivity of $\det f_* \sO_X(r (K_{X/Y} +D
))$. \autoref{sec:pushforward} is devoted to the proof of
\autoref{thm:main_big}. \autoref{sec:subadditivity} contains the statements and the
proofs of the subadditivity statements including \autoref{thm:main_subadd}. Finally,
in \autoref{sec:almost_results} we shortly deduce almost projective base versions of
the previously proven positivity statements.

\subsubsection*{\sffamily\itshape\small\bfseries Acknowledgement}\ The authors are
thankful to J\'anos Koll\'ar for many insightful conversations on the topic; to Maksym
Fedorchuk for the detailed answers on their questions about the curve case; to
James McKernan and Chenyang Xu for the information on the results in the article
\cite{Hacon_McKernan_Xu_Boundedness_of_moduli_of_varieties_of_general_type}.


\section{Basic tools and definitions}
\label{sec:basic}

We will be working over an algebraically closed base field $k$
characteristic zero in the entire article. In this section we give those definitions
and auxiliary statements that are used in multiple sections of the article. Most
importantly we define stable log-varieties and their families here.

\begin{definition}
  A \emph{variety} will mean a reduced but possibly reducible 
  separated scheme of finite type over $k$.  A \emph{vector bundle} $W$ on a variety
  $Z$ in this article will mean a locally free sheaf. Its dual is denoted by $W^*$.
\end{definition}

\begin{remark}
  It will always be assumed that the support of a divisor does not contain any
  irreducible component of the conductor subscheme. Obviously this is only relevant
  on non-normal schemes. The theory of Weil, Cartier, and $\bQ$-Cartier divisors work
  essentially the same on \emph{demi-normal} schemes, i.e., on schemes that satisfy
  Serre's condition $S_2$ and are semi-normal and Gorenstein in codimension $1$. 
  For more details on demi-normal schemes and their properties, including the
  definition and basic properties of divisors on demi-normal schemes 
  see \cite[\S5.1]{Kollar_Singularities_of_the_minimal_model_program}.
\end{remark}

\begin{definition}
  \label{def:big_open_set}
  Let $Z$ be a scheme. A \emph{big open subset} $U$ of $Z$ is an open subset
  $U\subseteq Z$ such that $\depth_{Z\setminus U} \sO_Z\geq 2$.  If $Z$ is $S_2$,
  e.g., if it is normal, then this is equivalent to the condition that
  $\codim_Z(Z\setminus U)\geq 2$.
\end{definition}

\begin{definition}
  The dual of a coherent sheaf $\sF$ on a scheme $Z$ will be denoted by $\sF^*$ and
  the sheaf $\sF^{**}$ is called the \emph{reflexive hull} of $\sF$.  If the natural
  map $\sF\to \sF^{**}$ is an isomorphism, then $\sF$ is called \emph{reflexive}.
  For the basic properties of reflexive sheaves see
  \cite[\S1]{Hartshorne_Stable_reflexive_sheaves}.
  
  Let $Z$ be an $S_2$ scheme and $\sF$ a coherent sheaf on $Z$. Then
  the \emph{reflexive powers} of $\sF$ are the reflexive hulls of
  tensor powers of $\sF$ and are denoted the following way:
  $$
  \sF^{[m]}:=  \left( \sF^{\otimes m} \right)^{**}
  $$
  Obviously, $\sF$ is reflexive if and only if $\sF\simeq \sF^{[1]}$.  Let $\sG$ be
  coherent sheaf on $Z$. Then the \emph{reflexive product} of $\sF$ and $\sG$ (resp.\
  reflexive symmetric power of $\sF$) is the reflexive hull of their tensor product
  (resp.\ of the symmetric power of $\sF$) and is denoted the following way:
  $$
  \sF[\otimes]\sG:= \left( \sF\otimes \sG \right)^{**} \qquad \Sym^{[a]}(\sF):=
  \left( \Sym^a( \sF) \right)^{**}
  $$
\end{definition}

\begin{notation}
  \label{not:base-change}
  Let $f:X\to Y$ and $Z\to Y$ be morphisms of schemes. Then the base change to $Z$
  will be denoted by
  $$
  f_{Z}: X_{Z} \to Z,
  $$
  where $X_{Z}:= X\times _Y Z$ and $f_{Z}:=f\times_Y \id_{Z}$. If $Z=\{y\}$ for
  a point $y\in Y$, then we will use $X_y$ and $f_y$ to denote $X_{\{y\}}$ and
  $f_{\{y\}}$.
\end{notation}

\begin{lemma}
  \label{lem:pushforward_tensor_product_isomorphism}
  Let $f : X \to Y$ and $g : Z \to Y$ be surjective morphisms such that $Y$ is normal
  and let $\sL$ and $\sN$ be line bundles on $X$ and $Z$ respectively.  Assume that
  there is a big open set of $Y$ over which $X$ and $Z$ are flat and $f_* \sL$ and
  $g_* \sN$ are locally free. Then
  \begin{equation*}
    \left( {\left({f_Z}
        \right)}_* ( p_X^* \sL \otimes p_Z^* \sN)  \right)^{**} \simeq
    f_* \sL \left[       \otimes \right] g_* \sN .
  \end{equation*}
  Furthermore, if $X$ and $Z$ are flat and $f_* \sL$ and $g_* \sN$ are locally free
  over the entire $Y$, then the above isomorphism is true without taking reflexive
  hulls. 
\end{lemma}

\begin{proof}
  Since the statement is about reflexive sheaves, we may freely pass to big open
  sets. In particular, we may assume that $f$ and $g$ are flat and $f_* \sL$ and $g_*
  \sN$ are locally free. Then
  \begin{equation*}    
    \begin{aligned}
    \xymatrix@R=.75em@C=0em{%
      \big(f_Z\big)_*
      ( p_X^* \sL \otimes p_Z^* \sN) 
      & 
      \simeq & g_* \left(  \left(p_{Z}\right)_{*} p_X^* \sL \otimes
        \sN \right)
      &      \simeq & g_* \left(  g^* f_* \sL  \otimes \sN \right)
      &      \simeq  & f_* \sL \otimes g_* \sN.
      \\
      \hskip-7em \ar@<3.65em>[u] \textrm{\tiny $g\circ p_Z$}
      & \hskip-7em \ar[u] \textrm{\tiny projection formula for $p_Z$} \hskip-7em & 
      & \hskip-7em \ar[u] \textrm{\tiny flat base-change} \hskip-7em &
      & \hskip-7em \ar[u] \textrm{\tiny projection formula for $g$} \hskip-7em &
    }
    \end{aligned}
    \qedhere 
  \end{equation*}
\end{proof}

\begin{notation}
  \label{notation:restriction_divisor}
  Let $f : X \to Y$ be a flat equidimensional morphism of demi-normal schemes, and
  $Z\to Y$ a morphism between normal varieties. Then for a $\bQ$-divisor $D$ on
  $X$ that avoids the generic and codimension $1$ singular points of the fibers of
  $f$, we will denote by $D_{Z}$ the \emph{divisorial pull-back} of $D$ to $X_{Z}$,
  which is defined as follows: As $D$ avoids the singular codimension $1$ points of
  the fibers, there is a big open set $U\subseteq X$ such that $D|_U$ is
  $\bQ$-Cartier.  Clearly, $U_{Z}$ is also a big open set in $X_{Z}$ and we define
  $D_{Z}$ to be the unique divisor on $X_{Z}$ whose restriciton to $U_{Z}$ is
  $\left(D|_U\right)_{Z}$. 
\end{notation}

\begin{remark}
  Note that this construction agrees with the usual pullback if $D$ itself is
  $\bQ$-Cartier, because the two divisors agree on $U_{Z}$.

  Also note that $D_{Z}$ is not necessarily the (scheme theoretic) base change of $D$
  as a subscheme of $X$. In particular, for a point $y\in Y$, $D_y$ is not
  necessarily equal to the scheme theoretic fiber of $D$ over $y$. The latter may
  contain smaller dimensional embedded components, but we restrict our attention to
  the divisorial part of this scheme theoretic fiber. This issue has already come up
  multiple times in \autoref{sec:introduction}, in particular in the example of
  \autoref{ex:Hassett}.

  Finally, note that if $q(K_{X/Y}+D)$ is Cartier, then using this definition the
  line bundle $\sO_X(q(K_{X/Y}+D))$ is compatible with base-change, that is, for a
  morphism $Z\to Y$,
  \begin{equation*}
    \left(\sO_X(q(K_{X/Y}+D))\right)_Z \simeq \sO_Z(q(K_{X_Z/Z}+D_Z)).
  \end{equation*}
  To see this, recall that this holds over $U_Z$ by definition and both sheaves are
  reflexive on $Z$. (See \autoref{def:relative_canonical} for the precise definition
  of $K_{X/Y}$.)
\end{remark}

\begin{definition}
  \label{def:stable_log_variety}
  A \emph{pair} $(Z,\Gamma)$ consist of an equidimensional demi-normal variety $Z$
  and an effective $\bQ$-divisor $\Gamma\subset Z$.
  A \emph{stable log-variety} $(Z,\Gamma)$ is a pair such that
  \begin{enumerate}
  \item $Z$ is proper,
  \item $(Z,\Gamma)$ has \emph{slc singularities}, and
  \item the $\bQ$-Cartier $\bQ$-divisor $K_Z+\Gamma$ is
    \emph{ample}.
  \end{enumerate}
  For the definition of slc singularities the reader is referred to
  \cite[5.10]{Kollar_Singularities_of_the_minimal_model_program}
\end{definition}

\begin{definition}
  \label{def:relative_canonical}
  If $f : X \to Y$ is either
  \begin{enumerate}
  \item a flat projective family of equidimensional demi-normal varieties, or
  \item a surjective morphism between normal projective varieties,
  \end{enumerate}
  then $\omega_{X/Y}$ is defined to be $f^! \sO_Y$. In particular, if $Y$ is
  Gorenstein (e.g., $Y$ is smooth), then $\omega_{X/Y} \simeq \omega_X \otimes
  f^*\omega_Y^{-1}$. In any case, $\omega_{X/Y}$ is a reflexive sheaf (c.f.,
  \cite[Lemma 4.9]{Patakfalvi_Schwede_Depth_of_F_singularities}) of rank
  $1$. Furthermore, if either in the first case $Y$ is also normal or in the second
  case $Y$ is smooth, then $\omega_{X/Y}$ is trivial at the codimension one points,
  and hence it corresponds to a Weil divisor that avoids the singular codimension one
  points \cite[5.6]{Kollar_Singularities_of_the_minimal_model_program}. This divisor
  can be obtained by fixing a big open set $U \subseteq X$ over which $\omega_{X/Y}$
  is a line bundle, and hence over which it corresponds to a Cartier divisor, and
  then extending this Cartier divisor to the unique Weil-divisor extension on
  $X$. Note that in the first case $U$ can be chosen to be the relative Gorenstein
  locus of $f$, and in the second case the regular locus of $X$. Furthermore, in the
  first case, we have $K_{X/Y}|_{V} \sim K_{X_V/V}$ for any $V \to Y$ base-change
  from a normal variety (here restriction is taken in the sense of
  \autoref{notation:restriction_divisor}).
\end{definition}

\begin{definition}
  \label{def:stable_lof_family}
  A \emph{family of stable log-varieties}, $f : (X, D) \to Y$ over a normal variety
  consists of a pair $(X,D)$ and a flat proper surjective morphism $f:X\to Y$ such
  that
  \begin{enumerate}
  \item $D$ avoids the generic and codimension $1$ singular points of every fiber,
  \item $K_{X/Y}+D$ is $\bQ$-Cartier, and
  \item $(X_y,D_y)$ is a connected stable log-variety for all $y\in Y$.
  \end{enumerate}
\end{definition}

\begin{notation}
  \label{notation:product}
  For a morphism $f : X \to Y$ of schemes and $m\in\bN_+$, define
  \begin{equation*}
    X^{(m)}_Y :=\ \mtimes{1}mY X \ =
    \underbrace{X \times_Y X \times_Y \dots \times_Y
      X}_{\textrm{$m$ times}} ,
  \end{equation*}
  and let $f^{(m)}_Y : X^{(m)}_Y \to Y$ be the induced natural
  map. For a sheaf of $\sO_X$-modules $\sF$ define
  \begin{equation*}
    \sF^{(m)}_Y := \bigotimes_{i=1}^m p_i^* \sF,
  \end{equation*}
  where $p_i$ is the $i$-th projection $X^{(m)}_Y \to X$. Similarly, if $f$ is flat,
  equidimensional with demi-normal fibers, then for a divisor $\Gamma$ on $X$ define
  \begin{equation*}
    \Gamma_{X^{(m)}_Y} := \sum_{i=1}^m p_i^* \Gamma,
  \end{equation*}
  a divisor on $X^{(m)}_Y$.

  Finally, for a subscheme $Z\subseteq X$, $Z^{(m)}_Y$ is naturally a subscheme of
  $X^{(m)}_Y$. Notice however that if $m>1$ and $Z$ has positive codimension in $X$,
  then $Z^{(m)}_Y$ is never a divisor in $X^{(m)}_Y$. In particular, if $Y$ is
  normal, $f$ is flat, equidimensional and has demi-normal fibers, and $\Gamma$ is an
  effective divisor that does not contain any generic or singular codimension $1$
  points of the fibers of $f$, then
  \begin{equation}
    \label{eq:reduced_intersection}
    \left( \Gamma^{(m)}_Y \right)_{\red} = \left( \bigcap _{i=1}^m p_i^* \Gamma
    \right)_{\red}.  
  \end{equation}

  Notice the difference between $\Gamma_{X^{(m)}_Y}$ and $\Gamma^{(m)}_{Y}$. The
  former corresponds to taking the $(m)^\text{th}$ box-power of a divisor as a sheaf,
  while the latter to taking fiber power as a subscheme. In particular,
  $$
  \sO_{X^{(m)}_Y}( \Gamma_{X^{(m)}_Y} ) \simeq \left( \sO_X(\Gamma) \right) ^{(m)}_Y,
  $$
  while $\Gamma^{(m)}_Y$ is not even a divisor if $m>1$.
 
  In most cases, we omit $Y$ from the notation. I.e., we use $X^{(m)}$,
  $\Gamma_{X^{(m)}}$, $\Gamma^{(m)}$, $f^{(m)}$ and $\sF^{(m)}$ instead of
  $X^{(m)}_Y$, $\Gamma_{X^{(m)}_Y}$, $\Gamma^{(m)}_Y$, $f^{(m)}_Y$ and $\sF^{(m)}_Y$,
  respectively.
\end{notation}

\section{Almost proper varieties and big line bundles}
\label{sec:almost}

\begin{definition}
  An \emph{almost proper} variety is a variety $Y$ that admits an embedding as a big
  open set into a proper variety $Y\hookrightarrow \oY$. If $Y$ is almost proper,
  then a \emph{proper closure} will mean a proper variety with such an embedding. The
  proper closure is not unique, but also, obviously, an almost proper variety is not
  necessarily a big open set for an arbitrary embedding into a proper (or other)
  variety. An almost proper variety $Y$ is called \emph{almost projective} when it
  has a proper closure $\oY$ which is projective. Such a proper closure will be
  called a \emph{projective closure}.
\end{definition}

\begin{lemma}\label{lem:upper-bound-on-sections}
  Let $Y$ be an almost projective variety of dimension $n$ and $B$ a Cartier divisor
  on $Y$.  Then there exists a constant $c>0$ such that for all $m>0$
  $$ 
  h^0(Y, \sO_Y(mB)) \leq c\cdot m^n 
  $$
\end{lemma}

\begin{proof}
  Let $\iota: Y\hookrightarrow \oY$ be a projective closure of $Y$ 
  and set $\sB_m=\iota_*\sO_Y(mB)$. Let $\sH$ be a very ample invertible sheaf on
  $\oY$ such that $H^0(\oY, \sH\otimes (\sB_1)^{*})\neq 0$ where $(\sB_1)^{*}$ is the
  dual of $\sB_1$. It follows that there exists an embedding $\sO_Y(B)\into \sH|_Y$
  and hence for all $m>0$ another embedding $\sO_Y(mB)\into \sH^m|_Y$. Pushing this
  forward to $\oY$ one obtains that $\sB_m \subseteq \iota_*\sH^m|_Y\simeq \sH^m$.
  Note that the last isomorphism follows by the condition of $Y$ being almost
  projective/proper, that is, because $\depth_{\oY\setminus Y}\sO_{\oY}\geq 2$.
  Finally this implies that
  $$ 
  h^0(Y, \sO_Y(mB)) = h^0(\oY, \sB_m) \leq h^0(\oY, \sH^m) \sim c\cdot m^n,
  $$
  where the last inequality follows from
  \cite[I.7.5]{Hartshorne_Algebraic_geometry}.
\end{proof}

\begin{definition}
  Let $Y$ be an almost proper variety of dimension $n$. A Cartier divisor $B$ on $Y$
  is called \emph{big} if $h^0(Y,\sO_Y(mB))> c\cdot m^n$ for some $c>0$ constant and
  $m\gg 1$ integer. A line bundle $\sL$ is called \emph{big} if the associated
  Cartier divisor is big.
\end{definition}

\begin{lemma}
  \label{lem:big-is-the-same}
  Let $Y$ be an almost proper variety of dimension $n$ and $\iota: Y\hookrightarrow
  \oY$ a projective closure of $Y$.  Let $\oB$ be a Cartier divisor on $\oY$ and
  denote its restriction to $Y$ by $B=\oB|_Y$. Then $B$ is big if and only if $\oB$
  is big.
\end{lemma}

\begin{proof}
  Clear from the definition and the fact that $\iota_*\sO_Y(mB)\simeq
  \sO_{\oY}(m\oB)$ for every $m\in \bZ$.
\end{proof}

\begin{remark}
  Note that it is generally not assumed that $B$ extends to $\oY$ as a Cartier
  divisor.
\end{remark}

\begin{lemma}
  \label{lem:equivalent-props-of-big}
  Let $Y$ be an almost projective variety of dimension $n$ and $B$ a Cartier divisor
  on $Y$.  Then the following are equivalent:
  \begin{enumerate}
  \item $mB\sim A+E$ where $A$ is ample and $E$ is effective for some $m>0$,
  \item the rational map $\phi_{|mB|}$
    associated to the linear system $|mB|$ is birational for some $m>0$,
  \item the projective closure of the image $\phi_{|mB|}$ has
    dimension $n$ for some $m>0$,  and
  \item $B$ is big.
  \end{enumerate}
\end{lemma}

\begin{proof}
  The proof included in
  \cite[2.60]{Kollar_Mori_Birational_geometry_of_algebraic_varieties} works almost
  verbatim. We include it for the benefit of the reader since we are applying it in a
  somewhat unusual setup.
  
  Clearly, the implications $(1)\Rightarrow (2)\Rightarrow (3)$ are obvious.
  To prove $(3) \Rightarrow (4)$, let $T=\overline{\phi_{|mB|}(Y)}\subseteq \bP^N$. By
  assumption $\dim T=n$, so by \cite[I.7.5]{Hartshorne_Algebraic_geometry} the
  Hilbert polynomial of $T$ is
  $
  h^0(T, \sO_T(l)) = (\deg T/n!)\cdot l^n + (\text{lower order terms}).
  $
  By definition of the associated rational map $\phi_{|mB|}$ induces an injection
  $H^0(T, \sO_T(l)) \subseteq H^0(Y, \sO_Y(lmB))$, which proves $(3) \Rightarrow
  (4)$.

  To prove $(4) \Rightarrow (1)$, let $B$ be a Cartier divisor on $Y$ and let $\iota:
  Y\hookrightarrow \oY$ be a projective closure of $Y$.  Further let $\oA$ be a
  general member of a very ample linear system on $\oY$.  Then $A:=\oA\cap Y$ is an
  almost projective variety by \cite[5.2]{MR0460317}.  It follows by
  Lemma~\ref{lem:upper-bound-on-sections} that
    $h^0(A, \sO_A(mB|_A)) \leq c\cdot m^{n-1}$, 
    which, combined with the exact sequence
  $$
  0 \to H^0(Y, \sO_Y(mB-A)) \to H^0(Y, \sO_Y(mB)) \to H^0(A, \sO_A(mB|_A)),
  $$
  shows that if $B$ is big, then $H^0(Y, \sO_Y(mB-A))\neq 0$ for $m\gg 0$ which
  implies $(1)$ as desired.
\end{proof}

The notion of weak-positivity used in this article is somewhat weaker than that of
\cite{Viehweg_Quasi_projective_moduli}. The main difference is that we do not require
being global generated on a fixed open set for every $b>0$ in the next
definition. This is a minor technical issue and proofs of the basic properties works
just as for the definitions of \cite{Viehweg_Quasi_projective_moduli}, after
disregarding the fixed open set. The reason why this weaker form is enough for us is
that we use it only as a tool to prove bigness, where there is no difference between
our definition and that of \cite{Viehweg_Quasi_projective_moduli}.

\begin{definition}
  \label{def:wp_big}
  Let $X$ be a normal, almost projective variety and $\sH$ an ample line bundle on $X$.
  \begin{enumerate}
  \item A coherent sheaf $\sF$ on $X$ is weakly-positive, if for every integer $a>0$
    there is an integer $b>0$, such that $\Sym^{[ab]} ( \sF ) \otimes \sH^b$ is
    generically globally generated. Note that this does not depend on the choice of
    $\sH$ \cite[Lem 2.14.a]{Viehweg_Quasi_projective_moduli}.
  \item A coherent sheaf $\sF$ on $X$ is big if there is an integer $a>0$ such that
    $\Sym^{[a]}(\sF) \otimes \sH^{-1}$ is generically globally generated. This
    definition also does not depend on the choice of $\sH$ by a similar argument as for
    the previous point. Further, this definition is compatible with the above
    definition of bigness for divisors and the correspondence between divisors and
    rank one reflexive sheaves.
  \end{enumerate}
\end{definition}

\begin{lemma}
  \label{lem:wp_big_properties}
  Let $X$ be a normal, almost projective variety, $\sF$ a weakly-positive and $\sG$ a
  big coherent sheaf. Then
  \begin{enumerate}
  \item \label{itm:tensors} $\bigoplus \sF$, $\Sym^{[a]}(\sF)$, $\left[ \bigotimes
    \right] \sF$, $\det \sF$ are weakly-positive,
  \item \label{itm:gen_surj} generically surjective images of $\sF$ are
    weakly-positive, and those 
    of $\sG$ are big,
  \item \label{itm:wp_ample} if $\sA$ is an ample line bundle, then $\sF \otimes \sA$
    is big, and
  \item
    \label{itm:wp_big} if $\sG$ is of rank $1$, then $\sF [\otimes] \sG$ is big.
  \end{enumerate}
\end{lemma}

\begin{proof}
  Let us fix an ample line bundle $\sH$.  \autoref{itm:tensors} follows verbatim from
  \cite[2.16(b) and 2.20]{Viehweg_Quasi_projective_moduli}, and
  \autoref{itm:gen_surj} follows immediately from the definition. Indeed, given
  generically surjective morphisms $\sF \to \sF'$ and $\sG \to \sG'$, there are
  generically surjective morphisms $\Sym^{[ab]} ( \sF ) \otimes \sH^b \to \Sym^{[ab]}
  ( \sF' ) \otimes \sH^b$ and $\Sym^{[a]}(\sG) \otimes \sH^{-1} \to \Sym^{[a]}(\sG')
  \otimes \sH^{-1}$ proving the required generic global generation.

  To prove \autoref{itm:wp_ample}, take an $a>0$, such that $\sA^a \otimes \sH^{-1}$
  is effective and $\sA^{c}$ is very ample for $c>a$. Then for a $b>a$ such that
  $\Sym^{[3b]} ( \sF ) \otimes \sA^{b}$ is globally generated, the embedding
  \begin{equation*}
    \Sym^{[3b]} ( \sF ) \otimes \sA^{b} \hookrightarrow  \Sym^{[3b]} ( \sF ) \otimes
    \sA^{3b -a} \simeq  \Sym^{[3b]} ( \sF \otimes \sA ) \otimes \sA^{-a} \hookrightarrow
    \Sym^{[3b]} ( \sF \otimes \sA ) \otimes \sH^{-1} 
  \end{equation*}
  is generically surjective which implies the statement.

  To prove \autoref{itm:wp_big} take an $a$, such that $\sH^{-1} \otimes \sG^{[a]} $
  is generically globally generated. This corresponds to a generically surjective
  embedding $\sH \to \sG^{[a]}$. According to \autoref{itm:tensors} and
  \autoref{itm:wp_ample}, $\left(\Sym^{[a]}(\sF) \otimes \sH\right)$ is big. Hence,
  by \autoref{itm:gen_surj}, $\Sym^{[a]}(\sF) \otimes \sG^{[a]} \simeq \Sym^{[a]}(\sF
  [ \otimes] \sG)$ is also big. Therefore, for some $b>0$, $\Sym^{[b]}(
  \Sym^{[a]}(\sF [ \otimes] \sG)) \otimes \sH^{-1}$ is generically globally generated
  and then the surjection $\Sym^{[b]}( \Sym^{[a]}(\sF [ \otimes] \sG)) \to
  \Sym^{[ab]}(\sF [ \otimes] \sG))$ concludes the proof.
\end{proof}

\section{Ampleness Lemma}
\label{sec:bigness_lemma}

\begin{theorem}
  \label{thm:generalized_ampleness_lemma}
  \label{prop:ampleness}
  Let $W$ be a weakly-positive vector bundle of rank $w$ on a normal almost
  projective variety $Y$ with a reductive structure group $G \subseteq \GL(k,w)$ the
  closure of the image of which in the projectivization $\bP(\Mat(k,w))$ of the space
  of $w \times w$ matrices is normal and let $Q_i$ be vector bundles of rank $q_i$ on
  $Y$ admitting generically surjective homomorphisms $\alpha_i : W \rightarrow Q_i$
  for $i=0,\dots,n$. Let $Y(k) \to \varprod_{i=0}^n \Gr(w,q_i)(k)/G(k)$ be the
  induced classifying map of sets. Assume that this map has finite fibers on a dense
  open set of $Y$.  Then
  $ \bigotimes_{i=0}^n  \det Q_i$
  is big.
\end{theorem}

\begin{remark}
  \label{rem:identification}
  One way to define the above classifying map is to choose a basis on every fiber of
  $W$ over every closed point up to the action of $G(k)$. For this it is enough to
  fix a basis on one fiber of $W$ over a closed point, and transport it around using
  the $G$-structure. In fact, a little less is enough. Given a basis, multiplying
  every basis vector by an element of $k^\times$ does not change the corresponding
  rank $q$ quotient space, and hence the classifying map, so we only need to fix a
  basis up to scaling by an element of $k^\times$. To make it easier to talk about
  these in the sequel we will call a basis which is determined up to scaling by an
  element of $k^\times$ a \emph{homogenous basis}. 
\end{remark}

\begin{remark}
  \label{rem:normal}
  The normality assumption in \autoref{thm:generalized_ampleness_lemma} is satisfied
  if $W= \Sym^d V$ with $v:= \rk V$ and $G:=\GL(k,v)$ acting via the representation
  $\Sym^d$. Indeed, in this case the closure of the image of $G$ in $\bP(\Mat(k,w))$
  agrees with the image of the embedding $\Sym^d : \bP(\Mat(k,v)) \to
  \bP(\Mat(k,w))$. In particular, it is isomorphic to $\bP(\Mat(k,v))$, which is
  smooth.

  For more results regarding when this normality assumption is satisfied in more
  general situations see \cite{MR1992080,MR2090668,MR2976317} and other references in
  those papers.
\end{remark}

\begin{remark}
  \autoref{thm:generalized_ampleness_lemma} is a direct generalization of the core
  statement \cite[3.13]{Kollar_Projectivity_of_complete_moduli} of Koll\'ar's
  Ampleness Lemma \cite[3.9]{Kollar_Projectivity_of_complete_moduli}.
  This statement is more general in several ways:
  \begin{itemize}[leftmargin=1.5em]
  \item The finiteness assumption on the classifying map is weaker (no assumption on
    the stabilizers).
  \item The ambient variety $Y$ is only assumed to be almost
    projective instead of projective.
  \end{itemize}
  Our proof is based on Koll\'ar's original idea with some modifications to allow for
  weakening the finiteness assumptions. 

  Note that if $Y$ is projective and $W$ is nef on $Y$, then it is also weakly
  positive \cite[Prop.~2.9.e]{Viehweg_Quasi_projective_moduli}.
\end{remark}

We will start by making a number of reduction steps to simplify the statement.  The
goal of this reduction is to show that it is enough to prove the following theorem
which contains the essential statement.

\begin{theorem}
  \label{thm:impr-bign-lemma}
  Let $W$ be a weakly-positive vector bundle of rank $w$ on a normal almost
  projective variety $Y$ with a reductive structure group $G \subseteq \GL(k,w)$ the
  closure of the image of which in the projectivization $\bP(\Mat(k,w))$ of the space
  of $w \times w$ matrices is normal and let $\alpha: W\twoheadrightarrow Q$ be a
  surjective morphism onto a vector bundle of rank $q$.  Let $Y(k) \to
  \Gr(w,q)(k)/G(k)$ be the induced classifying map. If this map has finite fibers on
  a dense open set of $Y$, then the line bundle
  $\det Q$
  is big.
\end{theorem}

\begin{lemma}
  \autoref{thm:impr-bign-lemma} implies \autoref{thm:generalized_ampleness_lemma}.
\end{lemma}

\begin{proof}
  \textbf{Step 1.}  \emph{We may assume that the $\alpha_i$ are surjective.}  Let
  $Q_i^-=\im\alpha_i\subseteq Q_i$. Then there exists a big open subset $\iota:
  U\into Y$ such that ${Q_i^-}|_{U}$ is locally free of rank $q_i$. If
  $\bigotimes_{i=1}^n \det ({Q_i^-}|_{U})$ is big, then so is $\left[
    \bigotimes_{i=1}^n \right] \det Q_i^-=\iota_* \left( \bigotimes_{i=1}^n \det
    ({Q_i^-}|_{U}) \right)$ and hence so is $\bigotimes_{i=1}^n \det Q_i$. Therefore
  we may replace $Y$ with $U$ and $Q_i$ with ${Q_i^-}|_{U}$.

  \textbf{Step 2.} \emph{It is enough to prove the statement for one quotient
    bundle.} Indeed, let $W'=\bigoplus_{i=0}^n W$ with the diagonal $G$-action,
  $Q'=\bigoplus_{i=0}^n Q_i$, and $\alpha:= \bigoplus_{i=0}^n \alpha_i : W'\to Q'$
  the induced morphism. If all the $\alpha_i$ are surjective, then so is $\alpha$.
  
  Furthermore, there is a natural injective $G$-invariant morphism
  $$
  \xymatrix@R=-.25em {%
    \displaystyle\varprod_{i=0}^n Gr(w,q_i)\ \ar@{^(->}[r] &
    \displaystyle Gr\bigg(rw, \sum_{i=0}^n
      q_i\bigg)\\
    (L_1,\dots,L_r) \ar@{|->}[r] & \displaystyle\bigoplus_{i=0}^n
    L_i. }
  $$
  Since the $G$-action on $\varprod_{i=0}^n Gr(w,q_i)$ is the restiction
  of the $G$-action on $Gr\left(rw, \sum_{i=0}^n q_i\right)$ via this
  embedding it follows that the induced map on the quotients remain
  injective:
  $$
  \xymatrix {%
    \factor{\displaystyle\varprod_{i=0}^n Gr(w,q_i)}G \ \ar@{^(->}[r] &
    \factor{\displaystyle Gr\bigg(rw, \sum_{i=0}^n q_i\bigg)}{G.} }
  $$
  It follows that the classifying map of $\alpha':W'\to Q'$ also has
  finite fibers and then the statement follows because
  $\det Q \simeq \bigotimes_{i=0}^n \det Q_i$. 
\end{proof}

\begin{lemma}
  \label{lem:wp_sub}
  If $V \subseteq W $ is a $G$-invariant sub-vector bundle of the $G$-vector bundle
  $W$ on a normal almost projective variety $X$, and $W$ is weakly positive, then so
  is $V$.
\end{lemma}

\begin{proof}
  $V$ corresponds to a subrepresentation of $G$, and by the characteristic zero and
  reductivity assumptions it follows that $V$ is a direct summand of $W$, so $V$ is
  also weakly positive.
\end{proof}

\begin{remark}
  The above lemma, which is used in the last paragraph of the proof, is the only
  place where the characteristic zero assumption is used in the proof of
  \autoref{thm:generalized_ampleness_lemma}. In particular, the statement holds in
  positive characteristic for a given $W$ if the $G$-subbundles of $W$ are
  weakly-positive whenever $W$ is.  According to \cite[Prop
  3.5]{Kollar_Projectivity_of_complete_moduli} this holds for example if $Y$ is
  projective and $W$ is nef satisfying the assumption $(\Delta)$ of \cite[Prop
  3.6]{Kollar_Projectivity_of_complete_moduli}. The latter is satisfied for example
  if $W= \Sym^d(W')$ for a nef vector bundle $W'$ of rank $w'$ and $G= \GL(k,w')$.
\end{remark}


\addtocounter{theorem}{-3}
\begin{proof}[Proof of \autoref{thm:impr-bign-lemma}]
  We start with the same setup as in
  \cite[3.13]{Kollar_Projectivity_of_complete_moduli}. Let
  $\pi:\bP=\bP(\oplus_{i=1}^w W^*)\to Y$, which can be viewed as the space of
  matrices with columns in $W$, and consider the universal basis map
  $$
  \varsigma: \bigoplus_{j=1}^w \sO_{\bP}(-1) \to \pi^* W,
  $$
  formally given via the identification $H^0(\bP,\sO_{\bP}(1) \otimes \pi^* W )
  \simeq H^0(Y, \bigoplus_{j=1}^w W^* \otimes W )$ by the identity sections of the
  different summands of the form $W^* \otimes W$. Informally, the closed points of
  $\bP$ over $y \in Y$ can be thought of as $w$-tuples $(x_1, \dots, x_w) \in W_y$
  and hence a dense open subset of $\bP_y$ corresponds to the choice of a basis of
  $W_y$ up to scaling by an element of $k^\times$, i.e., to a homogenous
  basis. Similarly, the map $\varsigma$ gives $w$ local sections of $\pi^* W$ which
  over $(x_1, \dots, x_w)$ take the values $x_1, \dots,x_w$, up to scaling by an
  element of $k^{\times}$ where this scaling corresponds to the transition functions
  of $\sO_{\bP}(-1)$.

  As explained in \autoref{rem:identification}, to define the classifying map we need
  to fix a homogenous basis of a fiber over a fixed closed point. Let us fix such a
  point $y_0\in Y$ and a homogenous basis on $W_{y_0}$ and keep these fixed
  throughout the proof. This choice yields an identification of $\bP_{y_0}$ with
  $\bP(\Mat(k,w))$.  Notice that the dense open set of $\bP_{y_0}$ corresponding to
  the different choices of a homogenous basis of $W_{y_0}$ is identified with the
  image of $\GL(k,w)$ in $\bP(\Mat(k,w))$ and the point in $\bP_{y_0}$ representing
  the fixed homogenous basis above is identified with the image of the identity
  matrix in $\bP(\Mat(k,w))$.  Now we want to restrict to a $G$ orbit inside all the
  choices of homogenous bases. Let $\wt{G}$ denote the closure of the image of $G
  \subseteq \GL(k,w)$ in $\bP(\Mat(k,w))$. 
  Via the identification of $\bP_{y_0}$ and $\bP(\Mat(k,w))$, $\wt G$ corresponds to
  a $G$-invariant closed subscheme of $\bP_{y_0}$, which carried around by the
  $G$-action defines a $G$-invariant closed subscheme $\mathbf{P} \subseteq
  \bP$. Note that since $\wt{G}$ is assumed to be normal, so is $\mathbf{P}$ by
  \cite[II 6.5.4]{EGAIV}.  To simplify
  notation let us denote the restriction $\pi|_{\mathbf{P}}$ also by
  $\pi$. Restricting the universal basis map to $\mathbf{P}$ and twisting by
  $\sO_{\mathbf{P}}(1)$ gives
  $$
  \beta:=\varsigma|_{\bfP}\otimes \id_{\sO_{\bfP}(1)}: \bigoplus_{j=1}^w
  \sO_{\mathbf{P}} \to \pi^* W \otimes \sO_{\mathbf{P}}(1).
  $$
  Let $\Upsilon\subset \mathbf{P}$ be the divisor where this map is not surjective,
  i.e., those points that correspond to non-invertible matrices via the above
  identification of $\bP_{y_0}$ and $\bP(\Mat(k,w))$. By construction, $\beta$ gives
  a trivialization of $\pi^* W \otimes \sO_{\mathbf{P}}(1)$ over
  $\mathbf{P}\setminus\Upsilon$. It is important to note the following fact about
  this trivialization: let $p \in \mathbf{P}_{y_0}$ be the closed point that via the
  above identification of $\bP_{y_0}$ and $\bP(\Mat(k,w))$ corresponds to the image
  of the identity matrix in $\bP(\Mat(k,w))$. Then the trivialization of $\pi^* W
  \otimes \sO_{\mathbf{P}}(1)$ given by $\beta$ gives a basis on $(\pi^*W_{y_0})_p$
  which is compatible with our fixed homogenous basis on $W_{y_0}$. Furthermore, for
  any $p'\in (\bfP\setminus \Upsilon)_{y_0}$ the basis on $(\pi^*W_{y_0})_{p'}$ given
  by $\beta$ corresponds to the fixed homogenous basis of $W_{y_0}$ twisted by the
  matrix (which is only given up to scaling by an element of $k^\times$)
  corresponding to the point $p'\in\bfP_{y_0}$ via the identification of $\bP_{y_0}$
  and $\bP(\Mat(k,w))$.  Note that as $G$ is reductive, it is closed in $\GL(k,w)$
  and hence $G(k)$ is transitive on $(\bfP\setminus \Upsilon)_{y_0}$. It follows that
  then the choices of homogenous bases of $W_{y_0}$ given by $\beta$ on
  $(\pi^*W_{y_0})_{p'}$ for $p'\in(\mathbf{P} \setminus \Upsilon )_{y_0}$ form a
  $G(k)$-orbit, 
  and this orbit may be identified with $(\bfP\setminus \Upsilon)_{y_0}$.
  Transporting this identification around $Y$ using the $G$-action we obtain:
  For every $y \in Y(k)$, 
  \begin{equation}
    \label{eq:compatibility_trivializations}
    \parbox{38.5em}{%
      $(\mathbf{P} \setminus \Upsilon )_{y}$ may be identified with
      the $G(k)$-orbit of homogenous bases of $W_{y}$ containing the homogenous basis
      obtained from the fixed homogenous basis of $W_{y_0}$ via the $G$-structure.
    } 
  \end{equation}
  Next consider the composition of $\wt\alpha=\pi^*\alpha\otimes
  \id_{\sO_{\mathbf{P}}(1)}$ and $\beta$:
  \begin{equation*}
    \xymatrix{%
      \gamma : \bigoplus_{j=1}^w \sO_{\mathbf{P}} \ar[r]^-\beta & 
      \pi^* W  \otimes \sO_{\mathbf{P}}(1) 
      \ar[r]^-{\wt\alpha} &  
      \pi^* Q  \otimes \sO_{\mathbf{P}}(1) 
    }
  \end{equation*}
  which is surjective on $\mathbf{P} \setminus \Upsilon$.  Taking
  $q^\text{th}$ wedge products yields
  \begin{equation*}
    \xymatrix{%
      \gamma^q : \bigoplus_{j=1}^{\binom wq} \sO_{\mathbf{P}}  \ar[r]^-{\beta^q}
      &  
      \pi^* (\wedge^q W)   \otimes \sO_{\mathbf{P}}(q) 
      \ar[r]^-{\wt\alpha^q} 
      &  
      \pi^* \det Q  \otimes \sO_{\mathbf{P}}(q) 
    }
  \end{equation*}
  which is still surjective outside $\Upsilon$ and hence gives a morphism
  \begin{equation*}
    \nu : \mathbf{P} \setminus\Upsilon \rightarrow \underbrace{\Gr(w, q) 
      \subseteq \mathbf{P} \left( \bigwedge^{q}   (k^{\oplus w})\right)
      =:\mathbf{P}_{Gr}}_{\textrm{Pl\"ucker embedding}}, 
  \end{equation*}
  \vskip-1em
  such that
  \begin{itemize}
  \item according to \autoref{eq:compatibility_trivializations}, on the $k$-points
    $\nu$ is a lift of the classifying map $Y \to \Gr/G$, where $\Gr:=\Gr(w,q)$ is
    the Grassmannian of rank $q$ quotients of a rank $w$ vectorspace, and
  \item $\nu^* \sO_{\Gr}(1)
    \simeq \left(\pi^*\det Q \otimes \sO_{\mathbf{P}}(q)\right) |_{\mathbf{P}
      \setminus\Upsilon}$, where $\sO_{\Gr}(1)$ is the restriction of
    $\sO_{\mathbf{P}_{Gr}}(1)$ via the Pl\"ucker embedding.
  \end{itemize}
  We will also view $\nu$ as a rational map $\nu: \mathbf{P} \dashrightarrow
  \Gr$.  Let $\sigma: \wt{\mathbf{P}} \to \mathbf{P}$ be the blow up of $(\im
  \gamma^q) \otimes \left( \pi^*\det Q \otimes \sO_{\mathbf{P}}(q)
  \right)^{-1} \subseteq \sO_{\mathbf{P}}$ and set $\wt\pi:=\pi \circ
  \sigma$. It follows that $\wt\nu=\nu\circ \sigma: \wt{\mathbf{P}} \to \Gr$ is
  well-defined everywhere on $\wt{\mathbf{P}}$ and there exists an effective
  Cartier divisor $E$ on $\wt{\mathbf{P}}$ such that
  \begin{equation}
    \label{eq:divisors}
    \sigma^*( \pi^* \det Q \otimes \sO_{\mathbf{P}}(q) ) \simeq \wt\nu^*
    \sO_{\Gr}(1) \otimes \sO_{\wt{\mathbf{P}}}(E). 
  \end{equation}
  Let $Y^\circ\subseteq Y$ be the dense open set where the classifying
  map has finite fibers and let $\mathbf{P}^\circ:=\wt\pi^{-1}(Y^\circ)
  \setminus \sigma^{-1}(\Upsilon)\subset \wt{\mathbf{P}}$. Observe that
  $\mathbf{P}^\circ \simeq\pi^{-1}(Y^0)\setminus \Upsilon$ via $\sigma$.
  
  Next let $T$ be the image of the product map $(\wt\pi \times
  \wt\nu): \wt{\mathbf{P}} \to Y \times \Gr$:
  $$
  T:= 
  {\im \left[\,(\wt\pi \times
      \wt\nu) 
      \,\right]} \subseteq Y \times \Gr,
  $$
  and let $\tau : T \to \Gr$ and $\phi : T \to Y$ be the
  projection. Furthermore, let $\vartheta: \wt{\mathbf{P}} \to T$ denote the
  induced morphism.  We summarize our notation in the following
  diagram. Note that although $Y$ is only almost proper, every scheme
  in the diagram (except $\Gr$ which is proper over $k$) is proper
  over $Y$.
  $$
  \xymatrix{%
    \wt{\mathbf{P}} \ar@/_5ex/[dd]_{\wt\pi} \ar[d]_\sigma \ar[rr]^-{\wt\nu}
    \ar[drr]|!{[dr];[rr]}\hole^{\vartheta}
    & &  \Gr \\
    \mathbf{P} \ar[d]_\pi & \ \mathbf{P}^\circ \ar[ru]^{\nu}
    \ar[r]_(.45){{}^{\vartheta|_{\mathbf{P}^\circ}}} \ar@{_(->}[l]
    \ar@{_(->}[lu] & %
    **[r] T \subseteq {Y\times \Gr}%
    \ar[lld]^{\phi} \ar[u]_\tau \\
    Y }
   $$

  \begin{subclaim}
    The map $\tau|_{\vartheta(\mathbf{P}^\circ)}$ has finite fibers.
  \end{subclaim}
  
  \begin{proof}
    Since $k$ is assumed to be algebraically closed, it is enough to show that
    for every $k$-point $x$ of $\Gr$ there are finitely many
    $k$-points of $\vartheta(\mathbf{P}^\circ)$ mapping onto $x$. Let $(y,x)$
    be such a $k$-point, where $y \in Y(k)$. Choose then $z \in
    \mathbf{P}^\circ(k)$ such that $\vartheta(z)=(y,x)$. Then $\pi(z)=y$ and
    $\nu(z)=x$.  Furthermore, if $\psi$ is the classifying map and
    $\xi$ is the quotient map $\Gr(k) \to \Gr(k)/G(k)$, then
    \begin{equation*}
      \psi(y)= \psi(\pi(z))= \xi(\nu(z)) = \xi(x). 
    \end{equation*}
    Therefore, $y \in \psi^{-1}(\xi(x))$. However, by the finiteness of the
    classifying map there are only finitely many such $y$.
  \end{proof}
  By construction $\vartheta(\mathbf{P}^\circ)$ is dense in $T$ and it is
  constructible by Chevalley's Theorem. Then the dimension of the generic fiber of
  $\tau$ equals the dimension of the generic fiber of
  $\tau|_{\vartheta(\mathbf{P}^\circ)}$ and hence $\tau$ is generically finite.

  Next consider a projective closure $Y\into \oY$ of $Y$ and let
  $\oT\subseteq \oY\times \Gr$ denote the closure of $T$ in $\oY\times
  \Gr$. Let $\ophi:\oT\to\oY$ and $\otau:\oT\to\Gr$ denote the
  projections. Clearly, $\ophi|_T=\phi$, $\otau|_T=\tau$, and $\otau$
  is also generically finite.
  Let $\oH$ be an ample Cartier divisor on $\oY$. Since $\otau^*
  \sO_{\Gr}(1)$ is big, there is an $m$, such that $\otau^*
  \sO_{\Gr}(m)\otimes \phi^* \sO_{\oY}(- \oH)$ has a non-zero
  section. Let $H=\oH|_{\oY}$ and restrict this section to $T$.  It
  follows that the line bundle
  \begin{equation}
    \label{eq:3}
    \vartheta^*\left(\tau^* \sO_{\Gr}(m)\otimes \phi^*\sO_Y(- H)\right) \simeq
    \wt\nu^* \sO_{\Gr}(m)\otimes \wt\pi^* \sO_Y(- H)
  \end{equation}
  also has a non-zero section, and then by \autoref{eq:divisors} and
  (\ref{eq:3}) there is also a non-zero section of
  \begin{equation*}
    \sigma^* \left(  \pi^*(\det Q)^m \otimes \sO_{\mathbf{P}}  
      \left( m q \right) \right)\otimes \wt\pi^* \sO_Y(- H) \simeq 
    \sigma^* \left(  \pi^*(\det Q)^m \otimes \pi^* \sO_Y(- H) \otimes \sO_{\mathbf{P}}  
      \left( m q \right) \right) . 
  \end{equation*}
  Pushing this section forward by $\sigma$ and using the projection formula we obtain
  a section of
  \begin{equation*}
    \left(  \pi^*(\det Q)^m \otimes \pi^* \sO_Y(- H) \otimes \sO_{\mathbf{P}}  
      \left( m q \right) \right)\otimes \sigma_* \sO_{\wt{\mathbf{P}}}
    \simeq
    \underbrace{  \pi^*(\det Q)^m \otimes \pi^* \sO_Y(- H) \otimes \sO_{\mathbf{P}}  
      \left( m q \right) }_{\textrm{$\sigma$ is birational and $\mathbf{P}$ is normal}}. 
  \end{equation*}
Pushing this section down via $\pi$ and rearranging the sheaves on the two sides of
  the arrow we obtain a non-zero morphism
  \begin{equation}
    \label{eq:2}
    (\pi_*\sO_{\mathbf{P}}(mq))^* \otimes \sO_Y(H) \to (\det Q)^m.
  \end{equation}
  Now observe, that by construction
  $$
  \left( \pi_*\sO_{\bP}(mq) \right)^* \simeq \left( \sym^{mq}\left(\bigoplus_{i=1}^w
      W^*\right) \right)^* \simeq \sym^{mq} \left( \bigoplus_{i=1}^w W \right)
  $$ 
  is weakly-positive and $\left( \pi_*\sO_{\mathbf{P}}(mq) \right)^*$ is a
  $G$-invariant subbundle of $\left( \pi_*\sO_{\bP}(mq) \right)^*$ for $m \gg 0$. In
  particular, by \autoref{lem:wp_sub}, $\left( \pi_*\sO_{\mathbf{P}}(mq) \right)^*$
  is weakly positive as well. Then by (\ref{eq:2}) and
  \autoref{lem:wp_big_properties} it follows that $\det Q$ is big.
\end{proof} 
\addtocounter{theorem}{3}


\section{Moduli spaces of stable log-varieties}
\label{sec:moduli}

\begin{definition}
  \label{def:close_under_addition}
  A set $I \subseteq [0,1]$ of coefficients is said to be \emph{closed under
    addition}, if for every integer $s>0$ and every $x_1,\dots, x_s \in I$ such that
  $\sum_{i=1}^s x_i \leq 1$ it holds that $\sum_{i=1}^s x_i \in I$.
\end{definition}

\begin{definition}
  \label{def:a_functor}
  Fix $0 < v \in \bQ$, $0 < n \in \bZ$ and a finite set of coefficients $I \subseteq
  [0,1]$ closed under addition. A functor $\sM : \mathfrak{Sch}_k \to
  \mathfrak{Sets}$ (or to groupoids) is \emph{a moduli (pseudo-)functor of stable
    log-varieties of dimension $n$, volume $v$ and coefficient set $I$}, if for each
  normal $Y$,
  \begin{equation}
    \label{eq:a_functor}
    \sM (Y) = 
    \left\{ \raisebox{2em}{\xymatrix{ (X,D) \ar[d]^f \\ Y }} \left| \parbox{26em}{
          \begin{enumerate}
          \item $f$ is a flat morphism,
          \item $D$ is a Weil-divisor on $X$ avoiding the generic and the codimension
            1 singular points of $X_y$ for all $y \in Y$,
          \item for each $y \in Y$, $(X_y, D_y)$ is a stable log-variety of dimension
            $n$, such that the coefficients of $D_y$ are in $I$, and $(K_{X_y} +
            D_y)^n = v$, and 
          \item $K_{X/Y} +D$ is $\bQ$-Cartier.
          \end{enumerate}} \right.  \right\},
  \end{equation}
  and the line bundle $Y \mapsto \det f_* \sO_X(r(K_{X/Y} + D))$ associated to every
  family as above extends to a functorial line bundle on the entire (pseudo-)functor
  for every divisible enough integer $r>0$.

  Also note that if $\sM$ is regarded as a functor in groupoids, then in
  \autoref{eq:a_functor} instead of equality only equivalence of categories should be
  required.
\end{definition}

\begin{remark}\label{rem:a_functor}
  \begin{enumerate}    
  \item The condition ``$D$ is a Weil-divisor on $X$ avoiding the generic and the
    codimension 1 singular points of $X_y$ for all $y \in Y$'' guarantees that $D_y$
    can be defined sensibly. Indeed, according to this condition, there is a big open
    set of $X_y$, over which $D$ is $\bQ$-Cartier.
  \item The condition ''$K_{X/Y} +D$ is $\bQ$-Cartier`` is superfluous according to a
    recent, yet unpublished result of Koll\'ar, which states that for a flat family
    with stable stable fibers if $y \mapsto (K_{X_y} +D_y)^n$ is constant, then
    $K_{X/Y} +D $ is automatically $\bQ$-Cartier.
  \item $I$ has to be closed under addition, to guarantee properness. Indeed,
    divisors with coefficients $c_1, c_2, \dots, c_s$, respectively, can come
    together in the limit to form a divisor with a coefficient $\sum_{i=1}^s c_i$.
  \item According to \cite[Thm
    1.1]{Hacon_McKernan_Xu_Boundedness_of_moduli_of_varieties_of_general_type}, after
    fixing $n$, $v$ and a DCC set $I \subseteq [0,1]$, there exist automatically
    \begin{enumerate}[(a)]
    \item a finite set $I_0 \subseteq I$ containing all the possible coefficients of
      stable log-varieties of dimension $n$, volume $v$ and coefficient set $I$, and
    \item a uniform $m$ such that $m(K_X +D)$ is Cartier for all stable log-varieties
      $(X,D)$ of dimension $n$, volume $v$ and coefficient set $I$.
    \end{enumerate}
    In particular, $m$ may also be fixed in the above definition if it is chosen to
    be divisible enough after fixing the other three numerical invariants.
  \end{enumerate}
\end{remark}

\begin{proposition}
  \label{prop:proper}
  Let $n>0$ be an integer, $v>0$ a rational number and $I \subseteq [0,1]$ a finite
  coefficient set closed under addition. Then any moduli (pseudo-)functor of stable
  log-varieties of dimension $n$, volume $v$ and coefficient set $I$ is proper. That
  is, if it admits a coarse moduli space which is an algebraic space, then that
  coarse moduli space is proper over $k$.  If in addition the pseudo-functor itself
  is a DM-stack , then it is a proper DM-stack over $k$ (from which the existence of
  the coarse moduli space as above follows
  \cite{Keel-Mori97,Conrad_The_Keel_Mori_theorem_via_stacks}).
\end{proposition}

\begin{proof}
  This is shown in \cite[Thm 12.11]{Kollar_Second_moduli_book}.
\end{proof}

\begin{proposition}
  \label{prop:finite_automorphism}
  If $(X,D)$ is a stable log-variety then $\Aut(X,D)$ is finite.
\end{proposition}

\begin{proof}
  Let $\pi : \oX \to X$ be the normalization of $X$ and $\oD$ is defined via
  \begin{equation*}
    K_{\oX} + \oD = \pi^* ( K_X + D ) 
  \end{equation*}
  where $K_X$ and $K_{\oX}$ are chosen compatibly such that $K_X$ avoids the singular
  codimension one points of $X$.  Note that $\oD\geq 0$ by
  \cite[(5.75)]{Kollar_Singularities_of_the_minimal_model_program}. Any automorphism
  of $(X,D)$ extends to an automorphism of $\left(\oX,\oD \right)$, hence we may
  assume that $(X,D)$ is normal. Furthermore, since $X$ has finitely many irreducible
  components, the automorphisms fixing each component form a finite index
  subgroup. Therefore, we may also assume that $X$ is irreducible. Let $U \subseteq
  X$ be the regular locus of $X \setminus \Supp D$. Note that $U$ is
  $\Aut(X,D)$-invariant, hence there is an embedding $\Aut(X,D) \hookrightarrow
  \Aut(U)$. In particular, it is enough to show that $\Aut (U)$ is finite. Next let
  $g :(Y, E) \to (X,D)$ be a log-crepant resolution that is an isomorphism over $U$
  and for which $g^{-1} (X \setminus U)$ is a normal-crossing divisor. Let $F$ be the
  reduced divisor with support equal to $g^{-1} (X \setminus U)$. Then $(Y,E)$ is
  log-canonical, and $E \leq F$. Therefore, $g^* (K_X + D) = K_Y + E \leq K_Y + F$
  and hence $(Y, F)$ is of log general type. However, $U = Y \setminus \Supp F$, and
  hence $U$ itself is of general type. Then by \cite[Thm
  11.12]{Iitaka_An_introduction_to_birational_geometry_of_algebraic_varieties} a
  group (which is called $\mathrm{SBir}(U)$ there) containing $\Aut(U)$ is
  finite. 
\end{proof}

\subsection{A particular functor of stable log-varieties}
\label{sec:particular_functor}

In what follows we describe a particular functor of stable log-varieties introduced
by J\'anos Koll\'ar \cite[(3) of page
21]{Kollar_Moduli_of_varieties_of_general_type}. The main reason we do so is to be
able to give \autoref{def:variation} and prove
\autoref{cor:extending_stable_log_families} and \autoref{cor:finite_cover}. These are
used in the following sections.

In fact, our method will be somewhat non-standard: we define a pseudo-functor
$\sM_{n,m,h}$ which is larger than needed in \autoref{def:the_functor}. We show that
$\sM_{n,m,h}$ is a DM-stack (\autoref{prop:DM_stack}) and if $m$ is divisible enough
(after fixing $n$ and $v$), the locus of stable log-varieties of dimension $n$,
volume $v$ and coefficient set $I$ is proper and closed in $\sM_{n,m,h}$. Hence the
reduced closed substack on this locus is a functor of stable log-varieties as in
\autoref{def:a_functor}. We emphasize that our construction is not a functor that we
propose to use in the long run. For example, we are not describing the values it
takes on Artinian non-reduced schemes. However, it does allow us to make
\autoref{def:variation} and prove \autoref{cor:extending_stable_log_families} and
\autoref{cor:finite_cover}, which is our goal here. Finding a reasonably good
functor(s) is an extremely important, central question which is postponed for future
endeavors.

The issue in general about functors of stable log-varieties is that, as
\autoref{def:a_functor} suggests, it is not clear what their values should be over
non-reduced schemes. The main problem is to understand the nature and behavior of $D$
in those situations. Koll\'ar's solution to this is that instead of trying to figure
out how $D$ should be defined over non-reduced schemes, let us replace $D$ as part of
the data with some other data equivalent to \autoref{eq:a_functor} that has an
obvious extension to non-reduced schemes. This ``other'' data is as follows: instead
of remembering $D$, fix an integer $m>0$ such that $m(K_X + D)$ is Cartier, and
remember instead of $D$ the map $\omega_X^{\otimes m} \to \sO_X(m(K_X +
D))=:\sL$. There are two things we note before proceeding to the precise definition.
\begin{enumerate}
\item A global choice of $m$ as above is possible according to
  \autoref{rem:a_functor}.
\item Fixing $\left(X, \phi: \omega_X^{\otimes m } \to \sL \right)$ is slightly more
  than just fixing $(X,D)$, since composing $\phi$ with an automorphism $\xi$ of $\sL$
  is formally different, but yields the same $D$. In particular, we have to remember
  that different pairs $(X, \phi)$ that only differ by an automorphism $\xi$ of $\sL$
  should be identified eventually.
\end{enumerate}
We define our auxiliary functor $\sM_{n,m,h}$ according to the above considerations.

\begin{definition} 
  \label{def:the_functor}
  Fix an integer $n>0$, a polynomial $h: \bZ \to \bZ$ and an integer $m>0$ divisible
  enough (after fixing $n$ and $h$). We define the auxiliary pseudo-functor
  $\sM_{n,m,h}$ as
  \begin{equation}
    \label{eq:the_functor}
    \stretchleftright[800]{.}
    {\sM_{n,m,h} (Y) = 
      \left\{ \left( \raisebox{2em}{\xymatrix{ X \ar[d] \\ Y }}, \phi:
          \omega_{X/Y}^{\otimes m} \to \sL \right) \left|\hskip-1em
          \parbox{25em}{%
            \begin{enumerate}
            \item
              $f$ is a flat morphism of pure relative dimension $n$,
            \item 
              $\sL$ is a relatively very ample line bundle on $X$ such that $R^i
              f_* (\sL^r)=0$ for every $r>0$, and 
            \item 
              for all $y \in Y$: {
              \begin{enumerate}
              \item 
                $\phi$ is an isomorphism at the generic points and at the
                codimension 1 singular points of $X_y$, and hence it determines a
                divisor $D_y$, such that $\sL_y \simeq \sO_y(m(K_{X_y} + D_y))$,                 
              \item 
                $(X_y, D_y)$ is slc, and
              \item
                $h(r)=\chi(X_y, \sL_y^r)$ for every integer $r>0$.
              \end{enumerate}}
          \end{enumerate}} \right.  \right\}}{/}\equiv,
  \end{equation}
  where 
  \begin{enumerate}[(a)]
  \item as indicated earlier, if $Y$ is normal, $\phi$ corresponds to an actual
    divisor $D$ such that $\sO_X(m(K_{X/Y} + D)) \simeq \sL$. Explicitly, $D$ is the
    closure of $\frac{E}{m}$, where $E$ is the divisor determined by $\phi$ on the
    relatively Gorenstein locus $U$.
  \item The arrows in $\sM_{n,m,h}$ between
    \begin{equation*}
      \left(X \to S, \phi: \omega_{X/S}^{\otimes m} \to \sL \right) \in \sM_{n,m,h}(S), 
    \end{equation*}
    and
    \begin{equation*}
      \left(X' \to T, \phi': \omega_{X'/T}^{\otimes m} \to \sL' \right) \in \sM_{n,m,h}(T), 
    \end{equation*}
    over a fixed $T \to S$ are of the form $(\alpha: X' \to X, \xi : \alpha^* \sL \to
    \sL')$, such that the square
    \begin{equation*}
      \xymatrix@R=1.75em{
        X' \ar[r]^\alpha \ar[d] & X \ar[d] \\
        T \ar[r] & S
      }
    \end{equation*}
    is Cartesian, and $\xi$ is an isomorphism such that the following diagram is
    commutative.
    \begin{equation}
      \label{eq:arrow_cfg}
      \begin{aligned}
        \xymatrix@C=6em@R=1.125em{
          \alpha^* \omega_{X/S}^{\otimes m} \ar@/^2em/[rr]^{\alpha^* \phi} \ar[r] &
          \left( \alpha^* \omega_{X/S}^{\otimes m} \right)^{**} \ar[dd]^{\simeq}_{
            \left.\parbox{11.5em}{\tiny unique extension of the canonical isomorphism
                on the relative Gorenstein locus given by
                \cite[3.6.1]{Conrad_Grothendieck_duality_and_base_change}}\right\}
            \to} \ar[r] & \alpha^*
          \sL   \ar[dd]^{\xi}  \\ \\
          \omega_{X'/T}^{\otimes m} \ar@/_2em/[rr]_{\phi'} \ar[r]
          & \omega_{X'/T}^{[m]} \ar[r] & \sL'  \\
        }.
      \end{aligned}
    \end{equation}
    In other words, $\phi'$ corresponds to $\xi \circ \alpha^* \phi$ via the natural
    identification
    \begin{equation*}
      \Hom \left( \alpha^* \omega_{X/S}^{\otimes m}, \sL'  \right) = \Hom \left(
        \omega_{X'/T}^{\otimes m}, \sL'  \right).   
    \end{equation*}

  \item An arrow as above is an isomorphism if $T \to S$ is the identity and $\alpha$
    is an isomorphisms.

  \item \label{itm:pullback_construction} We fix the following pullback
    construction. It features subtleties similar to that of \autoref{eq:arrow_cfg}
    stemming from the fact that only the hull $\omega_{X/Y}^{[m]}$ of
    $\omega_{X/Y}^{\otimes m}$ is compatible with base-change. So, let us consider
    $(X, \phi):=\left(X \to S, \phi: \omega_{X/S}^{\otimes m} \to \sL \right) \in
    \sM_{n,m,h}(S)$ and a $k$-morphism $T \to S$. Then $(X, \phi)_T:= \left(X_T \to
      T, \phi_{[T]} : \omega_{X_T/T}^{\otimes m} \to \sL_T \right)$, where $\phi_{[T]}$
    is defined via the following commutative diagram.
    \begin{equation*}
      \xymatrix@C=6em{
        \omega_{X_T/T}^{\otimes m} \ar[d] \ar[dr]^-{\phi_{[T]}} \\
        \omega_{X_T/T}^{[m]} \ar[r] & \sL_T \\
        \hskip-.75em\left(\omega_{X/Y}^{\otimes m}\right)_T\hskip-.75em 
        \ar[ur]_-{\phi_T} \ar[u]
      }
    \end{equation*}
    In other words, via the natural identification
    $\Hom\left(\left(\omega_{X/Y}^{\otimes m}\right)_T , \sL_T \right)
    =\Hom\left(\omega_{X_T/T}^{\otimes m} , \sL_T \right)$, $\phi_T$ corresponds to
    $\phi_{[T]}$.
 
  \end{enumerate}
\end{definition}

We leave the proof of the following statement to the reader. We only note that the
main reason it holds is that the presence of the very ample line bundle $\sL$ makes
descent work.

\begin{proposition}
  When viewed as a pseudo-functor (or equivalently as a category fibered in
  groupoids) $\sM_{n,m,h}$ is an \'etale (or even fppf) stack.
\end{proposition}


\begin{proposition}
  \label{prop:representable_and_unramified}
  Let $\left(f : X \to Y, \phi: \omega_{X/Y}^{\otimes m} \to \sL \right)$ and
  $\left(f' : X' \to Y, \phi': \omega_{X'/Y}^{\otimes m} \to \sL' \right)$ be two
  objects in $\sM_{n, m,h} (Y)$. Then the isomorphism functor of these two families
  $\sfIsom_Y((X,\phi),(X', \phi'))$ is representable by a quasi-projective scheme
  over $Y$, which is denoted by $\Isom_Y((X,\phi),(X', \phi'))$.  Furthermore, this
  isomorphism scheme, $\Isom_Y((X,\phi),(X', \phi'))$, is unramified over $Y$.
\end{proposition}

\begin{remark}
  Recall that, by definition, $\sfIsom_Y((X,\phi),(X', \phi'))(T)$ is the set of
  $T$-isomorphisms between $(X,\phi)_T$ and $(X', \phi')_T$ for any scheme $T$ over
  $k$.
\end{remark}

\begin{proof}
  First, we show the representability part of the statement.  Let $I:=
  \Isom_Y^*(X,X') \to Y$ be the connected components of the Isom scheme
  $\Isom_Y(X,X')$ parametrizing isomorphisms $\gamma: X_T \to X_T'$ such that
  $\gamma^* \sL_T' \equiv_T \sL_T$ \cite[Exercise
  1.10.2]{Kollar_Rational_curves_on_algebraic_varieties}. It comes equipped with a
  universal isomorphism $\alpha: X_I \to X_I'$. Now, let
  $J:=\underline{\Isom}_I(\alpha^* \sL_I', \sL_I)$ be the open part of
  $\underline{\Hom}_I(\alpha^* \sL_I', \sL_I)$ \cite[33]{Kollar_Hulls_and_Husks}
  parametrizing isomorphisms. This space also comes equipped with a universal
  isomorphism $\xi: \alpha_J^* \sL_J' \to \sL_J$. This space $J$, with the universal
  family $\alpha_J : X_J \to X_J'$ and $\xi: \alpha_J^* \sL_J' \to \sL_J$ is a fine
  moduli space for the functor
  \begin{equation*}
    T \mapsto \{ ( \beta, \zeta) | \beta : X_T \to X_T' \textrm{ and } \zeta :
    \beta^* \sL_T' \to \sL_T \textrm{ are isomorphisms} \}. 
  \end{equation*}
  This is almost the functor $\sfIsom_Y((X,\phi),(X', \phi'))$, except in the latter
  there is an extra condition that the following diagram commutes:
  \begin{equation}
    \label{eq:commutative_condition}
    \begin{aligned}
      \xymatrix@C=4em{ \beta^* \omega_{X_T'/T}^{\otimes m} \ar[d]^\simeq
        \ar[r]^-{\beta^* \phi_{[T]}'} & \beta^* \sL_T' \ar[d]^{\zeta} \\
        \omega_{X_T/T}^{\otimes m} \ar[r]^-{\phi_{[T]}} & \sL_T }
    \end{aligned}    
  \end{equation}
  Note that here we do not have to take hulls. Indeed, $\beta^*
  \omega_{X_T'/T}^{\otimes m}$ itself is isomorphic to $\omega_{X_T/T}^{\otimes m}$
  via the $m$-th tensor power of the unique extension of the canonical map of
  \cite[Thm 3.6.1]{Conrad_Grothendieck_duality_and_base_change} from the relative
  Gorenstein locus, since $\beta$ is an isomorphism and hence $\beta^*
  \omega_{X_T'/T}$ is reflexive.

  Hence we are left to show that the condition of the commutativity of
  \autoref{eq:commutative_condition} is a closed condition. That is, there is a
  closed subscheme $S \subseteq J$, such that the condition of
  \autoref{eq:commutative_condition} holds if and only if the induced map $T \to J$
  factors through $S$.

  Set $\psi:= \phi_{[J]}$ and let $\psi'$ be the composition
  \begin{equation*}
    \xymatrix@C=4em{
      \omega_{X_J/J}^{\otimes m} \simeq \alpha^* \omega_{X_J'/J}^{\otimes m}
      \ar[r]^-{\alpha^* \phi_{[J]}'} & \alpha^* \sL_J' \ar[r]^-{\xi}  & \sL_J 
    }.
  \end{equation*}
  Consider $M:=\underline{\Hom} \left(\omega_{X_J/J}^{\otimes m}, \sL_J \right) $
  \cite[33]{Kollar_Hulls_and_Husks}. The homomorphisms $\psi$ and $\psi'$ give two
  sections $s, s': J \to M$.  Let $S:= s'^{-1} (s(J))$.

  In the remainder of the proof we show the above claimed universal property of $S$.
  Take a scheme $T$ over $k$ and a pair of isomorphisms $(\beta, \zeta)$, where
  $\beta$ is a morphism $ X_T \to X_T'$ and $\xi$ is a homomorphism $\beta^* \sL_T' \to
  \sL_T$. Let $\mu : T \to J$ be the moduli map, that is via this map $\beta= \alpha_T$
  and $\zeta= \xi_T$. We have to show that the commutativity of
  \autoref{eq:commutative_condition} holds if and only if $\mu$ factors through the
  closed subscheme $S \subseteq J$.

  First, by the natural identification
  $\underline{\Hom} \left(\omega_{X_T/T}^{\otimes m}, \sL_T \right) =
  \underline{\Hom} \left( \left(\omega_{X_J/J}^{\otimes m} \right)_T, \sL_T \right)$
  the commutativity of \autoref{eq:commutative_condition} is equivalent to $\psi_T =
  \psi_T'$.  Second, by functoriality of $\Mor$, the latter condition is equivalent
  to $s_T=s_T'$ (as sections of $M_T \to T$).  However, the latter is equivalent to
  the factorization of $T \to J$ through $S$, which shows that indeed
  $\Isom_Y((X,\phi),(X', \phi')):=S$ represents the functor $\sfIsom_Y((X,\phi),(X',
  \phi'))$.

  For the addendum, note that $\Isom_Y((X,\phi),(X', \phi'))$ is a group scheme over
  $Y$. Since $\mathrm{char} k=0$, the characteristics of all the geometric points is
  $0$ and hence all the geometric fibers are smooth. This implies that
  $\Isom_Y((X,\phi),(X', \phi'))$ is unramified over $Y$ \cite[Tag
  02G8]{stacks-project}, since its geometric fibers are finite by
  \autoref{prop:finite_automorphism}.
\end{proof}

\begin{lemma}
  \label{lem:slc_open}
  Let $\left(f : X \to Y, \omega_{X/Y}^{\otimes m} \to \sL \right)$ satisfy
  conditions \emph{(1), (2), (3i)} and \emph{(3iii)} in \autoref{eq:the_functor},
  i.e., do not assume that $(X_y,D_y)$ is slc. Further assume that $X_y$ is
  demi-normal for all $y \in Y$ and $Y$ is essentially of finite type over $k$. Then
  the subset $Y^\circ:=\{y\in Y\mid (X_y,D_y) \text{ is slc}\}\subseteq Y$ is open.
\end{lemma}

\begin{proof}
  Let $\tau : Y' \to Y$ be a resolution. As $\tau$ is proper, we may replace the
  original family with the pullback to $Y'$ and so we may assume that $Y$ is
  smooth. Next we show that \emph{the slc locus $\{ y \in Y | (X_y, D_y) \textrm{ is
      slc} \}$ is constructible}. For that it is enough to show that there is a
  non-empty open set $U$ of $Y$ such that either $(X_y, D_y)$ is slc for all $y \in
  U$ or $(X_y, D_y)$ is not slc for all $y \in U$
  and conclude by noetherian induction. To prove the existence of such a $U$, we may
  assume that $Y$ is irreducible. Let $\rho : X' \to X$ be a semi-smooth
  log-resolution and $U \subseteq Y$ an open set for which
  \begin{itemize}
  \item $\rho^{-1} f^{-1} U  \to U$ is flat,
  \item $X_y' \to X_y$ is a semi-smooth log-resolution for all $y \in U$, and
  \item for any exceptional divisor $E$ of $\rho$ that does not dominate $Y$ (i.e.,
    which is $f$-vertical) $f(\rho(E))\cap U=\emptyset$.
  \end{itemize}
  It follows that for $y \in U$, the discrepancies of $(X_y, D_y)$ are 
  independent of $y$. Hence, either every such $(X_y,D_y)$ is slc or all of them are
  not slc.
 
  Next, we prove that \emph{the locus $\{ y \in Y | (X_y, D_y) \textrm{ is slc} \}$
    is closed under generalization}, which will conclude our proof by \cite[Exc
  I.3.18.c]{Hartshorne_Algebraic_geometry}. For, this we should prove that if $Y$ is
  a DVR, essentially of finite type over $k$, and $(X_\xi, D_\xi)$ is slc for the
  closed point $\xi \in Y$, then so is $(X_\eta, D_\eta)$ for the generic point $\eta
  \in Y$. However, this follows immeditaley by inversion of adjunction for slc
  varieties \cite[Cor 2.11]{Patakfalvi_Fibered_stable_varieties}, since that implies
  that $(X, D + X_\xi)$ is slc and then by localizing at $\eta$ we obtain that
  $(X_{\eta}, D_{\eta})$ is slc.
\end{proof}

\begin{proposition}
  \label{prop:DM_stack}
  $\sM_{n,m,h}$ is a DM-stack of finite type over $k$.
\end{proposition}

\begin{proof}
  For simplicity let us denote $\sM_{n,m,h}$ by $\sM$.  According to
  \cite[4.21]{Deligne_Mumford_The_irreducibility_of_the_space_of_curves_of_given_genus}
  we have to show that $\sM$ has representable and unramified diagonal, and there is
  a smooth surjection onto $\sM$ from a scheme of finite type over $k$. For any stack
  $\sX$ and a morphism from a scheme $T \to \sX \times_k \sX$ corresponding to $s,t
  \in \sX(T)$, the fiber product $\sX \times_{\sX \times_k \sX} T$ can be identified
  with $\Isom_T(s,t)$. Hence the first condition follows from
  \autoref{prop:representable_and_unramified}. For the second condition we are to
  construct a cover $S$ of $\sM$ by a scheme such that $S \to \sM$ is formally
  smooth. The rest of the proof is devoted to this.



  Set $N:= p(1) -1$. Then, $\mathfrak{Hilb}_{\bP^{N}}^h$ contains every $\left(X,
    \phi : \omega_X^{\otimes m} \to \sL \right) \in \sM(k)$, where $X$ is embedded
  into $\bP^N_k$ using $H^0(X,\sL)$. Let $\sH_1 := \mathfrak{Hilb}_{\bP^{N}}^h$ be
  the open subscheme corresponding to $X \subseteq \bP^N$, such that $H^i(X,
  \sO_X(r))=0$ for all integers $i>0$ and $r>0$. According to
  \cite[III.12.2.1]{EGAIV}, there is an open subscheme $\sH_2 \subseteq \sH_1$
  parametrizing the reduced equidimensional and $S_ 2$ varieties. Since small
  deformations of nodes are either nodes or regular points, we see that there is an
  open subscheme $\sH_3 \subseteq \sH_2$ parametrizing the demi-normal varieties
  (where reducedness and equidimensionality is included in demi-normality). Let
  $\sU_3$ be the universal family over $\sH_3$. According to
  \cite[33]{Kollar_Hulls_and_Husks} there is a fine moduli scheme
  $M_4:=\underline{\Hom}_{\sH_3} \left(\omega_{\sU_3/\sH_3}^{\otimes m},
    \sO_{\sU_3}(1) \right)$. Define $\sU_4$ and $\sO_{\sU_4}(1)$ to be the pullback
  of $\sU_3$ and of $\sO_{\sU_3}(1)$ over $M_4$. Then there is a universal
  homomorphism $\gamma: \omega_{\sU_4/M_4}^{\otimes m} \to \sO_{\sU_4}(1)$.
  Let $M_5 \subseteq M_4$ be the open locus where $\gamma$ is an isomorphism 
  at every generic point and singular codimension one point of each fiber is
  open. Let $\sU_5$ and $\sO_{\sU_5}(1)$ the restrictions of $\sU_4$ and
  $\sO_{\sU_4}(1)$ over $M_5$.  According to \autoref{lem:slc_open}, there is an even
  smaller open locus $M_6 \subseteq M_5$ defined by
  \begin{equation*}
    M_6:= \left\{ t \in M_5 \left|  \omega_{(\sU_5)_t}^{\otimes m} \to
        \sO_{(\sU_5)_t}(1) \textrm{ corresponds to an slc pair } \right.  \right\}. 
  \end{equation*}
  Then define $S:=M_6$ and $g : U \to S$ and $\phi: \omega_{U/S}^{\otimes m} \to
  \sO_U(1)$ to be respectively the restrictions of $\sU_5 \to M_5$ and of $\gamma$
  over $M_6$.  From \autoref{def:the_functor} and by cohomology and base-change it
  follows that for each $\left(h: X_T \to T, \phi' : \omega_{X_T/T}^{\otimes m} \to
    \sL_T \right) \in \sM(T)$ such that $T$ is Noetherian,
  \begin{enumerate}
  \item the sheaf $h_* \sL_T$ is locally free, and
  \item giving a map $\nu : T \to S$ and an isomorphism $(\alpha, \xi)$ between $(h:
    X_T \to T, \phi' : \omega_{X_T/T}^{\otimes m} \to \sL_T)$
    and $\left(U_T \to T, \phi_{[T]} : \omega_{U_T/T}^{\otimes m} \to \sO_{U_T}(1)
    \right)$ is equivalent to fixing a set of free generators $s_0, \dots, s_n \in
    h_* \sL_T$.
  \end{enumerate}
  Indeed, for the second statement, fixing such a generator set is equivalent to
  giving a closed embedding $\iota: X_T \to \bP^N_T$ with Hilbert polynomial $h$
  together with an isomorphism $\zeta : \sL_T \to \iota^*
  \sO_{\bP^N_T}(1)$. Furthermore, the latter is equivalent to a map
  $\nu_{\mathrm{pre}} : T \to \sH_3$ together with isomorphisms $\alpha: X_T \to
  (\sU_3)_T$ and $\xi : \sL_T \to \alpha^* \sO_{(\sU_3)_T}(1)$.  Then the composition
  \begin{equation*}
    \xymatrix{
      \alpha^* \omega_{(\sU_3)_T/T}^{\otimes m} \ar[r]^-{\simeq} &
      \omega_{X_T/T}^{\otimes m} \ar[r]^-{\phi'} & \sL_T \ar[r]^-{\xi} & \iota^*
      \sO_{(\sU_3)_T}(1)  
    }
  \end{equation*}
  yields a lifting of $\nu_{\mathrm{pre}}$ to a morphism $\nu : T \to S$, such that
  $(\alpha, \xi^{-1})$ is an isomorphism between $(X_T, \phi')$ and $\left(U_T,
    \phi_{[T]} \right)$.

  Now, we show that the map $S \to \sM$ induced by the universal family over $S$ is
  smooth. It is of finite type by construction, so we have to show that it is
  formally smooth. Let $\delta : (A', \mathfrak m') \twoheadrightarrow (A,\mathfrak
  m)$ be a surjection of Artinian local rings over $k$ such that $\mathfrak m (\ker
  \delta) =0$.  Set $T:= \Spec A$ and $T':=\Spec A'$. According to
  \cite[IV.17.14.2]{EGAIV}, we need to show that if there is a 2-commutative diagram
  of solid arrows as follows, then one can find a dashed arrow keeping the diagram
  2-commutative.
  \begin{equation*}
    \xymatrix{
      S \ar[d] & \ar[l] T \ar[d] \\
      \sM & \ar[l] T' \ar@{-->}[lu]
    }
  \end{equation*}
  In other words, given a family $\left( h : X_{T'} \to T', \phi' :
    \omega_{X_{T'}/T'}^{\otimes m} \to \sL \right) \in \sM(T')$, with an isomorphism
  $(\beta, \zeta)$ between $\left(X_T, \phi_T' \right)$ and $(U_T, \phi_T)$. We are
  supposed to prove that $(\beta, \zeta)$ extends over $T'$. However, as explained
  above, $(\beta, \zeta)$ corresponds to free generators of $\left(h_T\right)* \sL_T$,
  which can be lifted over $T'$ since $T \to T'$ is an infinitesimal extension of
  Artinian local schemes.
\end{proof}

\begin{lemma}
  Let $\left(f : X \to Y, \omega_{X/Y}^{\otimes m} \to \sL \right) \in \sM_{n,m,h}(T)$
  for some $T$ essentially of finite type over $k$ and $I \subseteq [0,1]$ a finite
  coefficient set closed under addition. Then the locus
  \begin{equation}
    \label{eq:coefficients_in_I}
    \{t \in T | (X_t, D_t) \textrm{ has coefficients in } I \}
  \end{equation}
  is closed (here $D_t$ is the divisor corresponding to $\phi_t$).  Furthermore, if
  $m$ is divisible enough (after fixing $n$, $v$ and $I$), then the above locus is
  proper over $k$.
\end{lemma}

\begin{proof}
  For the first statement, according to \cite[Exc
  II.3.18.c]{Hartshorne_Algebraic_geometry} we are suppsoed to prove that the above
  locus is constructible and closed under specialization. Both of these follow from
  the fact that if $T$ is normal, and $D_T$ is the divisor corresponding to $\phi_T$,
  then there is a dense open set $U \subseteq T$ such that the coefficients of $D_T$
  and of $D_t$ agree for all $t \in U$. For the ``closed under specialization'' part
  one should also add that if $T$ is a DVR with generic point $\eta$ and special
  point $\varepsilon$, then the coefficient set of $D_\eta$ agrees with the
  coefficient set of $D$, and the coefficients of $D_\varepsilon$ are sums formed
  from coefficients of $D$. Since $I$ is closed under addition, if $D_\eta$ has
  coefficients in $I$, so does $D_\varepsilon$.

  The properness statement follows from \cite[Thm 12.11]{Kollar_Second_moduli_book}
  and \cite[Thm
  1.1]{Hacon_McKernan_Xu_Boundedness_of_moduli_of_varieties_of_general_type}.
\end{proof}

\begin{notation}
  \label{notation:cutting_out_coefficients}
  Fix an integer $n>0$, a rational number $v>0$ and a finite coefficient set $I
  \subseteq [0,1]$ closed under addition. After this choose an $m$ that is divisible
  enough. For stable log-varieties $(X, D)$ over $k$ for which $\dim X =n$,
  $(K_X+D)^n = v$ and the coefficient set is in $I$, there are finitely many
  possibilities for the Hilbert polynomial $h(r)= \chi(X,rm(K_X+D))$ by \cite[Thm
  1.1]{Hacon_McKernan_Xu_Boundedness_of_moduli_of_varieties_of_general_type}. Let
  $h_1, \dots, h_s$ be these values. For each integer $1 \leq i \leq s$, let $\sM_i$
  denote the reduced structure on the locus \autoref{eq:coefficients_in_I} of
  $\sM_{n,m,h_i}$ and let $\sM_{n,v,I}:= \amalg_{i=1}^s \sM_{i}$ (where $\amalg$
  denotes disjoint union).
\end{notation}

\begin{proposition}
  $\sM_{n,v,I}$ is a pseudo-functor for stable log-varieties of dimension $n$, volume
  $v$ and coefficient set $I$.
\end{proposition}

\begin{proof}
  Given a normal variety $T$, $\sM_{n,v,I}(T)= \amalg_{i=1}^s \sM_i(T)$.  Since in
  \autoref{notation:cutting_out_coefficients}, $\sM_i$ were defined by taking reduced
  structures, for reduced schemes $T$, there are no infinitesmial conditions on
  $\sM_i(T)$. That is it is equivalent to the sub-groupoid of $\sM_{n,m,h_i}(T)$
  consisting of $\left( X \to T, \phi : \omega_{X/T}^{\otimes m} \to \sL \right)$,
  such that the coefficients of $(X_t, D_t)$ is in $I$. Then it follows by
  construction that the disjoint union of these is equivalent to the groupoid given
  in \autoref{eq:a_functor} and that the line bundle $\det f_* \sL^j$ associated to
  $\left( X \to T, \phi : \omega_{X/T}^{\otimes m} \to \sL \right) \in
  \sM_{n,v,I}(T)$ yields a polarization for every integer $j>0$.
\end{proof}

\begin{remark}
  \label{rem:m_ambiguity}
  $\sM_{n,v,I}$ a-priori depends on the choice of $m$, which will not matter for our
  applications. However, one can show by exhibiting isomorphic groupoid
  representations that in fact the normalization of any DM-stack $\sM$ which is a
  pseudo-functor of stable log-varieties of dimenion $n$, volume $v$ and coefficient
  set $I$ is isomorphic to the normalization of $\sM_{n,v,I}$.
\end{remark}

\begin{definition}
  \label{def:variation} 
  Given a family $f : (X,D) \to Y$ of stable log-varieties over an irreducible normal
  variety, such that the dimension $\dim X_y=n$ and the volume $(K_{X_y} + D_y)^n$ of
  the fibers are fixed. Let $I$ be the set of all possible sums, at most $1$, formed
  from the coefficients of $D$. Then, there is an associated moduli map $\mu: Y \to
  \sM_{n,v,I}$. The \emph{variation $\var f$ of $f$} is defined as the dimension of
  the image of $\mu$.

  Note that this does not depend on the choice of $m$ or $I$ (see
  \autoref{rem:m_ambiguity}), since it is $\dim Y - d$, where $d$ is the general
  dimension of the isomorphism equivalence classes of the fibers $(X_y,D_y)$. This
  general dimension exists, because it can also be expressed as the general fiber
  dimension of $\Isom_Y((X, \phi),(X,\phi))$, where $(X,\phi) \in \sM_{n,m,h_i}(Y)$
  corresponds to $(X,D)$.
  
  Further note that it follows from the above discussion that using any
  pseudo-functor of stable log-varieties of dimension $n$, volume $v$ and coefficient
  set $I$ instead of $\sM_{n,v,I}$ leads to the same definition of variation.
\end{definition}


\begin{remark}
  \label{rem:variation}
  \autoref{cor:pullback_from_max_var} gives another alternative definition of
  variation: it is the smallest number $d$ such that there exists a diagram as in
  \autoref{cor:pullback_from_max_var} with $d = \dim Y'$.
\end{remark}

\begin{corollary}
  \label{cor:extending_stable_log_families}
  Given $f : (X,D) \to Y$ a family of stable log-varieties over a normal variety $Y$,
  and a compactification $\oY \supseteq Y$, there is a generically finite proper
  morphism $\tau : \oY' \to \oY$ from a normal variety, and a family $f :
  \left(\oX,\oD \right) \to \oY'$ of stable log-varieties, such that $\left(\oX_{Y'},
    \oD_{Y'}\right) \simeq (X_{Y'}, D_{Y'})$, where $Y':= \tau^{-1} Y$.
\end{corollary}

\begin{proof}
  Let $n$ be the dimension and $v$ the volume of the fibers of $f$. Let $I \subseteq
  [0,1]$ be a finite coefficient set closed under addition that contains the
  coefficients of $D$.  Denote for simplicity $\sM_{n,v,I}$ by $\sM$. According to
  \cite[Thm 16.6]{Laumon_Moret_Bailly_Champs_algebrique}, there is a finite,
  generically \'etale surjective map $S \to \sM$, and $f : (X,D) \to Y$ induces
  another one $Y \to \sM$. Let $Y'$ be a component of the normalization of $Y
  \times_{\sM} S$ dominating $Y$. Note that since $\sM$ is a DM-stack, $Y$ is a
  scheme and $Y' \to Y$ is finite and surjective. Hence, we may compactify $Y'$ to
  obtain a normal projective variety $\oY'$, such that the maps $Y' \to S$ and $Y'
  \to Y$ extend to morphisms $\oY' \to S$ and $\oY' \to \oY$ (note that both $S$ and
  $\oY'$ are proper over $k$). Hence, we have a $2$-commutative diagram
  \begin{equation*}
    \xymatrix{
      Y' \ar@{^(->}[r] \ar[d] & \oY' \ar[d]^{\tau} \ar[r] & S \ar[d] \\
      Y \ar@{^(->}[r] \ar@/_2pc/[rr]^{\ } & \oY &  \sM
    },
  \end{equation*}
  which shows that the induced family on $\oY'$ has the property as required, that
  is, by pulling back to $Y'$ it becomes isomorphic to the pullback of $(X,D)$ to
  $Y'$.
\end{proof}

\begin{corollary}
  \label{cor:finite_cover}
  If $\sM$ is a moduli (pseudo-)functor of stable log-varieties of dimension $n$,
  volume $v$ and coefficient set $I$ admitting a coarse moduli space ${\sf M}$ which
  is an algebraic space, then there is a finite cover $S \to {\sf M}$ from a normal
  scheme $S$ induced by a family ${\sf f}\in\sM(S)$.
\end{corollary}

\begin{proof}
  Since for every moduli (pseudo-)functor $\sM$ of stable log-varieties of dimension
  $n$, volume $v$ and coefficient set $I$, $\sM(k)$ is the same (as a set or as a
  groupoid), and furthermore ${\sf M}$ is proper over $k$ according to
  \autoref{prop:proper}, it is enough to show that there is a proper $k$-scheme $S$,
  such that $S$ supports a family ${\sf f}\in \sM(S)$ for which
  \begin{enumerate}
  \item the isomorphism equivalence classes of the fibers of ${\sf f}$ are finite, and
  \item every isomorphism class in $\sM(k)$ appears as a fiber of ${\sf f}$.
  \end{enumerate}
  However, the existence of this follows by \cite[Thm
  16.6]{Laumon_Moret_Bailly_Champs_algebrique} and \autoref{prop:DM_stack}.
\end{proof}

\begin{corollary}
  \label{cor:pullback_from_max_var}
  Given a family $ f : (X, D) \to Y$ of stable log-varieties over a normal variety,
  there is diagram
 \begin{equation*}
    \xymatrix{
      (X',D') \ar[d]^{f'} & (X'',D'') \ar[l] \ar[r] \ar[d] & (X,D) \ar[d]^f \\
      Y' & \ar[l] Y'' \ar[r] & Y } 
  \end{equation*}
 with Cartesian squares, such that 
 \begin{enumerate}
 \item $Y'$ and $Y''$ are normal,
 \item  $\var f = \dim Y'$,
 \item $Y'' \to Y$ is finite, surjective, and
 \item $f' : (X',D') \to Y'$ is a family of stable log-varieties for which the
   induced moduli map is finite. In particular, the fiber isomorphism classes of $f'
   : (X',D') \to Y'$ are finite.
\end{enumerate}
  
\end{corollary}

\begin{proof}
  Set $n := \dim X_y$ and $v:=(K_{X_y} + D_y)^n$. Let $I$ be the set of all possible
  sums, at most $1$, formed from the coefficients of $D$. Then there is an induced
  moduli map $\nu : Y \to \sM_{n,v,I}$. Let $S \to \sM_{n,v,I}$ be the finite cover
  given by \autoref{cor:finite_cover}. The map $Y \times_{\sM_{n,v,I}} S \to Y$ is
  finite and surjective.  Define $Y''$ to be the normalization of an irreducible
  component of $Y \times_{\sM_{n,v,I}} S$ that dominates $Y$ and define $Y'$ to be
  the normalization of the image of $Y''$ in $S$. That is, we obtain a
  $2$-commutative diagram
  \begin{equation*}
    \xymatrix{
      Y'' \ar[r] \ar[d] & Y' \ar[d] \\
      Y \ar[r]^{\nu} & \sM_{n,v,I}
    }.
  \end{equation*}
  This yields families over $Y'$ and $Y''$ as required by the statement.
\end{proof}



\section{Determinants of pushforwards}
\label{sec:det}


\noindent
The main results of this section are the following theorem and its corollary. For the
definition of stable families see \autoref{def:stable_lof_family} and for the
definition of variation see \autoref{def:variation} and \autoref{rem:variation}. We
also use \autoref{notation:product} in the next statement.

\begin{theorem}
  \label{thm:big_higher_dim_base}\label{prop:big_upstairs}
  If $f : (X, D) \to Y$ is a family of stable log-varieties of maximal variation over
  a smooth projective variety, then
  \begin{enumerate}
  \item\label{ketto}%
    there exists an $r>0$ such that $K_{X^{(r)}/Y} + D_{X^{(r)}}$ is big on at least
    one component of $X^{(r)}$, or equivalently $$\left(K_{X^{(r)}/Y} + D_{X^{(r)}}
    \right)^{\dim X^{(r)}} >0,$$ and
  \item\label{egy}%
    for every divisible enough $q>0$, $\det f_* \sO_X(q (K_{X/Y} + \Delta))$ is big.
  \end{enumerate}
\end{theorem}

\begin{remark}
  \label{rem:fiber_power_necessary}
  The $r$-th fiber power in point \autoref{ketto} of
  \autoref{thm:big_higher_dim_base} cannot be dropped. This is because there exist
  families $f : X \to Y$ of maximal variation that are not varying maximally on any
  of the components of $X$. Note the following about such a family:
  \begin{enumerate}
  \item $K_{X/Y}$ cannot be big on any component $X_i$ of $X$. Indeed, since the
    variation of $f|_{X_i}$ is not maximal, after passing to a generically finite
    cover of $X_i$, $K_{X/Y}|_{X_i}$ is a pull back from a lower dimensional variety.
  \item On the other hand, $X^{(r)} \to Y$ will have a component of maximal variation
    for $r\gg 0$. In particular, $K_{X^{(r)}/Y}$ does have a chance to be big on at
    least one component.
  \end{enumerate}
  To construct a family as above, start with two non-isotrivial smooth families $g_i
  : Z_i \to C_i$ ($i=1,2$) of curves of different genera, both at least two \cite[Sec
  V.14]{Barth_Peters_Van_de_Ven_Compact_complex_surfaces}. Take a multisection on
  each of these. By taking a base-change via the multisections, we may assume that in
  fact each $g_i$ is endowed with a section $s_i : C_i \to Z_i$. Now define $f_1:=
  g_1 \times \id_{C_2} : X_1:= Z_1 \times C_2 \to Y:= C_1 \times C_2$ and $f_2:=
  \id_{C_1} \times g_2 : X_2:= C_1 \times Z_2 \to Y$. The section $s_i$ of $g_i$
  induce sections of $f_i$ as well. Let $D_i$ be the images of these. Then, according
  to \cite[Thm 5.13]{Kollar_Singularities_of_the_minimal_model_program}, $(X_1, D_1)$
  and $(X_2, D_2)$ glues along $D_1$ and $D_2$ to form a stable family $f : X \to Y$
  as desired. Also notice that in this example $f^{(2)}:X^{(2)}\to Y$ has a component
  of maximal variation.
\end{remark}

\begin{corollary}
  \label{cor:projective}
  Any algebraic space that is the coarse moduli space of a functor of stable
  log-varieties with fixed volume, dimension and coefficient set (as in
  \autoref{def:a_functor}) is a projective variety over $k$.
\end{corollary}

The rest of the section contains the proofs of \autoref{thm:big_higher_dim_base} and
\autoref{cor:projective}. The first major step is \autoref{prop:big_downstairs},
which needs a significant amount of notation to be introduced.

\begin{definition}\label{def:D_c}
  For a $\bQ$-Weil divisor $D$ on a demi-normal variety and for a $c \in \bQ$ we
  define the $c$-coefficient part of $D$ to be the reduced effective divisor
  \begin{equation*}
    D_c:=\sum_{\coeff_E D =c} E ,
  \end{equation*}
  where the sum runs over all prime divisors. Clearly
  $$
  D = \sum_{c\in\bQ} cD_c.
  $$
  Notice that $D_c$ is invariant under any automorphism of the pair $(X,D)$, that is,
  under any automorphism of $X$ that leaves $D$ invariant. In fact, an automorphism
  of $X$ is an automorphism of the pair $(X,D)$ if and only if it leaves $D_c$
  invariant for every $c\in\bQ$.
\end{definition}

\begin{definition}
  \label{def:D_c-comp-with-base-change}
  Let $f : (X, D) \to Y$ be a family of stable log-varieties.  We will say that the
  \emph{coefficients of $D$ are compatible with base-change} if for each $c \in \bQ$
  and $y \in Y$,
  \begin{equation*}
    D_c |_{X_y} = (D_y)_c.
  \end{equation*}
  Note that this condition is automatically satisfied if all the coefficients are
  greater than $\frac{1}{2}$.
\end{definition}

\begin{notation}
  \label{not:main-setup}

  Let $f:(X,D)\to Y$ be a family of stable log-varieties over a smooth projective
  variety. For a fixed $m\in \bZ$ that is divisible by the Cartier index of $K_{X/Y}+
  D$, and an arbitrary $d\in\bZ$ set $\sL_d:= \sO_X(dm(K_{X/Y} + D))$.
  
  Observe that there exists a dense big open subset $U\subseteq Y$ over which all the
  possible unions of the components of $D$ (with the reduced structure) are flat.
  Our goal is to apply \autoref{thm:generalized_ampleness_lemma} for $f_U:X_U\to U$
  (we allow shrinking $U$ after fixing $d$ and $m$, while keeping $U$ a big open
  set).

  Next we will group the components of $D$ according to their coefficients. Recall
  the definition of $D_c$ from \autoref{def:D_c} where $c \in \bQ$ and observe that
  there is an open set $V\subseteq U$ over which
  \begin{enumerate}[label=({\sf\Alph*})]
  \item\label{item:1} $D_c$ is compatible with base-change as in
    \autoref{def:D_c-comp-with-base-change} for all $c\in\bQ$, and
  \item\label{item:2} 
    the scheme theoretic fiber of $D_c$ over $v \in V$ is reduced and therefore is
    equal to its divisorial restriction (see the definition of the latter in
    \autoref{notation:restriction_divisor}).    
  \end{enumerate}

  To simplify notation we will make the following definitions: Let $\{c_1, \dots,
  c_n\}:=\{c\in\bQ \mid D_c\neq\emptyset\}$ be the set of coefficients appearing in
  $D$ and let $D_i:=D_{c_i}$, for $i=1,\dots,n$.

  Next we choose an $m\in\bZ$ satisfying the following conditions for every integer
  $i,j,d>0$:
  \begin{enumerate}[resume,label=({\sf\Alph*})]
  \item\label{item:3} 
    $m(K_{X/Y} + D)$ is Cartier,
  \item\label{item:4} 
    $\sL_d= \sO_X(dm(K_{X/Y} + D))$ is $f$-very ample,
  \item\label{item:5} 
    $R^jf_*  \sL_d =0$, 
  \item\label{item:6} 
    $ \left. \left( R^j\left(f|_{D_i} \right)_* \sL_d|_{D_i} \right) \right|_V=0$, and
  \item\label{item:7} 
    $\left(f_*\sL_1 \right)|_V \rightarrow \left(\left(f|_{D_i}\right)_*\sL_1|_{D_i}
    \right)|_V$ is surjective.
  \end{enumerate}
  These conditions imply that 

  \begin{enumerate}[resume,label=({\sf\Alph*})]
  \item\label{item:8} 
    $\bN\ni N:=h^0(\sL_1|_{X_y})-1$ is independent of $y\in Y$, and in fact
  \item\label{item:9} 
    $f_* \sL_d$ and $\left. \left( \left(f|_{D_i} \right)_* \sL_d|_{D_i} \right)
    \right|_V$ are locally free and compatible with base-change.
  \end{enumerate}
  By possibly increasing $m$ we may also assume that
  \begin{enumerate}[resume,label=({\sf\Alph*})]
  \item the multiplication maps
    \begin{equation*}
      \hspace{2em} \sym^d(f_*\sL_1) \rightarrow
      \left(f_*\sL_d\right) \hspace{1em} \textrm{and} \hspace{1em}
      \sym^d(f_*\sL_1)|_V \rightarrow
      \left(\left(f|_{D_i}\right)_*\sL_d|_{D_i}\right)|_V  
    \end{equation*}
     are surjective.
    \label{item:10}
  \end{enumerate}
  For the surjectivity of the map
  $\sym^d(f_*\sL_1)|_V \rightarrow
  \left(\left(f|_{D_i}\right)_*\sL_d|_{D_i}\right)|_V$ 
  we write it as the composition of the restriction map
  $\sym^d \left(f_*\sL_1 \right)|_V \rightarrow \sym^d
  \left(\left(f|_{D_i}\right)_*\sL_1|_{D_i} \right)|_V$
  and the multiplication map
  $\sym^d \left(\left(f|_{D_i}\right)_*\sL_1|_{D_i} \right)|_V \rightarrow
  \left(\left(f|_{D_i}\right)_*\sL_d|_{D_i}\right)|_V$.
  The former is surjective by the choice of $m$ and condition \ref{item:7} while the
  surjectivity of the latter follows by the finite generation of the relative section
  ring, after an adequate increase of $m$.

  We fix an $m$ satisfying the above requirements for the rest of the section and use 
  the global sections of $\sL_1|_{X_y}$ to embed $X_y$ (and hence $D_i|_{X_y}$ as
  well) into the fixed projective space $\mathbb P^N_{k}$ for every closed point
  $y\in V$. The ideal sheaves corresponding to these embeddings will be denoted by
  $\sI_{X_y}$ and $\sI_{D_i|_{X_y}}$ respectively. As the embedding of $X_y$ is
  well-defined only up to the action of $\GL(N+1,k)$, the corresponding ideal
  sheaf is also well-defined only up to this action. Furthermore, in what follows we
  deal with only such properties of $X_y$, $D_i|_{X_y}$, $\sI_{X_y}$ and
  $\sI_{D_i|_{X_y}}$ that are invariant under the $\GL(N+1,k)$ action.

  So, finally, we choose a $d>0$ such that
    
  \begin{enumerate}[resume,label=({\sf\Alph*})]
  \item 
    for all $y \in V$, $X_y$ as well as $D_i|_{X_y}$ are defined by degree $d$
    equations.
  \end{enumerate}

  From now on we keep $d$ fixed with the above chosen value and we supress it from
  the notation. We make the following definitions:
  \begin{enumerate}[resume,label=({\sf\Alph*})]
  \item\label{item:12} 
    $W := \sym^d(f_* \sL_1)|_U$, and
  \item\label{item:13} 
    $Q_0:= (f_* \sL_d)|_U$.
  \end{enumerate}
  Further note that $\left( f|_{D_i } \right)_* \sL_d|_{D_i}$ is torsion-free, since
  $f|_{D_i }$ is surjective on all components and $D_i$ is reduced. Hence by possibly
  shrinking $U$, but keeping it still a big open set, we may assume that
  \begin{enumerate}[resume,label=({\sf\Alph*})]
  \item\label{item:14}     
    $Q_i:= \left. \left( \left( f|_{D_i } \right)_* \sL_d|_{D_i} \right) \right|_U$
    is locally free for all for $i>0$.    
  \end{enumerate}
  Our setup ensures that we have natural homomorphisms $\alpha_i:W \to Q_i$ which are 
  surjective over $V$ and we may make the following identifications for all closed
  points $y\in V$ up to the above explained $\GL(N+1,k)$ action:
  $$
  \xymatrix@R0em{%
    Q_0\otimes k(y) \ar@{<->}[r] & **[r] H^0 \left(\bP^N,
      \sO_{\bP^N\vphantom{|_{X_y}}}(d)
    \right)  \\
    \ker \bigg[ W \otimes k(y) \to Q_0 \otimes k(y)\bigg] \ar@{<->}[r] & **[r] H^0
    \left(\bP^N, \sI_{X_y\vphantom{|_{X_y}}}(d) \right) \\
    \ker \bigg[ W \otimes k(y) \to Q_i \otimes k(y)\bigg] \ar@{<->}[r] & **[r] H^0
    \left(\bP^N, \sI_{{D_i}|_{X_y}}(d) \right) \text{, for $i>0$.}}
  $$
  We will use this setup and notation for the rest of the present section. 
\end{notation}

\begin{lemma}
  \label{lem:pushforward_nef}
  Let $f : (X, D) \to Y$ be a family of stable log-varieties over a normal proper
  variety $Y$, and let $m>0$ be an integer such that
  \begin{enumerate}
  \item $m(K_{X/Y} + D)$ is Cartier,
  \item $m(K_{X/Y} + D)$ is relatively basepoint-free with respect to $f$, and
  \item $R^if_* \sO_X(m(K_{X/Y} + D))=0$ for all $i>0$.
  \end{enumerate}
  Then $f_* \sO_X(m(K_{X/Y} + D))$ is a nef locally free sheaf.  Further note, that
  the above conditions and hence the statement hold for every divisible enough
  $m$. In particular, it applies for the $m$ chosen in \autoref{not:main-setup}, and
  hence $f_*\sL_d$ is weakly positive for all $d>0$.
\end{lemma}

\begin{proof}
  The assumptions guarantee that $f_* \sO_X(m(K_{X/Y} + D))$ is compatible with
  base-change. As being nef is decided on curves, we may assume that $Y$ is a smooth
  curve. Note that then by the slc version of inversion of adjunction (e.g.,
  \cite[Cor 2.11]{Patakfalvi_Fibered_stable_varieties}) $(X,D)$ itself is slc. Hence,
  \cite[Theorem 1.13]{Fujino_Semi_positivity_theorems_for_moduli_problems} applies
  and yields the statement.
\end{proof}

\begin{proposition}
  \label{prop:big_downstairs} 
  In the situation of \autoref{not:main-setup}, assume that the variation is
  maximal. Then for all $d \gg 0$,
  \begin{equation*}
    \det f_* \sL_d \otimes \left( \otimes_{i=1}^n  \det \left(\left( f|_{D_i}
        \right)_* \sL_d|_{D_i} \right) \right)
  \end{equation*}
  is big.
\end{proposition}

\begin{proof}
  Note that $f_* \sL_1$ is weakly positive by \autoref{lem:pushforward_nef} and hence
  so is $W=\sym^df_* \sL_1$. This will allow us to use
  \autoref{thm:generalized_ampleness_lemma} in the situation of
  \autoref{not:main-setup} by setting $G:=\GL(N+1, k)$ (see \autoref{rem:normal})
  with the natural action on $W$ if we prove that the restriction over $V$ of the
  classifying map of the morphisms $\alpha_i$ for $i=0,\dots,n$ have finite fibers.

  Translating this required finiteness to geometric terms means that fixing a general
  $y \in V(k)$ and the fiber $X_y$, there are only finitely many other general $z \in
  V(k)$, such that for the fiber $X_z$ the degree $d$ forms in the ideals of $X_y$
  and $D_{i}|_{X_y}$ can be taken by a single $\phi \in \GL(N+1,k)$ to the degree
  $d$ forms in the ideals of $X_z$ and $D_{i}|_{X_z}$. However, if such a $\phi$
  exists, then $(X_y, D_y) \simeq (X_z, D_z)$ meaning that $y$ and $z$ lie in the
  same fiber of the associated moduli map $\mu: Y \to {\sM}_{m,v,I}$ (see
  \autoref{sec:particular_functor}). The maximal variation assumption implies that
  $\mu$ is generically finite, so there is an open $Y^0 \subseteq Y$, over which
  $\mu$ has finite fibers, which is exactly what we need.  By shrinking $V$, we may
  assume that $V \subseteq Y^0$ and applying
  \autoref{thm:generalized_ampleness_lemma} yields the statement.
\end{proof}

\begin{lemma}
  \label{lem:flat}
  Let $f : (X,D) \to Y$ be a family of stable log-varieties over a smooth variety.
  Then $D_c|_{T}$ is flat for all $c\in\bQ$, where $T$ is the locus over which $D_c$
  is Cartier. Note that $T|_{X_y}$ is a big open set for every $y \in Y$.
\end{lemma}

\begin{proof}
  As $\sO_{D_c|_T}$ is the cokernel of $\varepsilon: \sO_T(-D_c) \to \sO_T$, it is
  enough to prove that $\varepsilon_y: \sO_T(-D_c) \otimes \sO_{T_y} \to \sO_{T_y}$
  is injective for every $y \in Y$ \cite[Tag 00MD]{stacks-project}. However, as
  $\sO_T(-D_c) \otimes \sO_{T_y}$ is a line bundle on $T_y$, and hence $S_2$, and the
  map $\varepsilon_y$ is an isomorphism, in particular injective, at every generic
  point of $T_y$, it is in fact injective everywhere.
\end{proof}

\begin{lemma}
  \label{lem:equidimensional}
  Let $f : (X,D) \to Y$ be a family of stable log-varieties over a smooth variety.
  Then $D_c \to Y$ is an equidimensional morphism for all $c\in\bQ$.
\end{lemma}

\begin{proof}
  By assumption $D_c$ has codimension $1$ in $X$ and it does not contain any
  irreducible components of any fiber. It follows that the general fiber of $D_c$
  over $Y$ has codimension $1$ in the corresponding fiber of $X$ and that this is the
  maximum dimension any of its fibers may achieve. Since the dimension of the fibers
  is semi-continuous this implies that all fibers of $D_c$ have the same dimension.
\end{proof}

\begin{lemma}
  \label{lem:not_too_ugly}
  Let $f : (X,D) \to Y$ be a family of stable log-varieties over a smooth
  variety. Let $Z$ be the fiber product over $Y$ of some copies of $X$ and of the
  $D_i=D_{c_i}$'s. Then
  \begin{enumerate}
  \item every irreducible component of $Z$ dominates $Y$,\label{item:ugly1}
  \item there is a big open set of $Y$ over which $Z$ is flat and
    reduced,\label{item:ugly2}
  \item $Z$ is equidimensional over $X$,  and\label{item:ugly3}
  \item $X$ is regular at every generic point of $Z$.\label{item:ugly4}
  \end{enumerate}
\end{lemma}

\begin{proof}
  First notice that \autoref{item:ugly3} follows directly from
  \autoref{lem:equidimensional}.

  Next recall that we have already noted in \autoref{not:main-setup} that there
  exists a big open set $U \subseteq Y$, over which $X$ and all the possible unions
  of the components of $D$ are flat, and hence so is $Z$. It follows that all the
  embedded points of $Z$ over $U$ map to the generic point $\eta$ of $Y$. However
  $Z_{\eta}$ is reduced, so $Z$ is not only flat, but also reduced over $U$. This
  proves \autoref{item:ugly2}.

  On the other hand, $Z$ can definitely have multiple irreducible or even connected
  components. Assume that there exists an irreducible component $S$ that does not
  dominate $Y$. Then $S$ is contained in the non-flat locus of $Z$. However,
  according to \autoref{lem:flat}, the non-flat locus of $D_i$ has codimension at
  least one in each fiber of $D_i \to Y$ for all $i$'s. Therefore, the non-flat locus
  of $Z$ also has codimension at least one in each fiber. Hence, the existence of $S$
  would contradict \autoref{item:ugly3} (and ultimately
  \autoref{lem:equidimensional}). This proves \autoref{item:ugly1}.

  By \autoref{item:ugly1} the generic points of $Z$ are dominating the generic points
  of $D_i$.  At these points the corresponding fibers of $X$ are regular and so
  \autoref{item:ugly4} follows.
\end{proof}

\autoref{not_multiproduct} is used in the proof of
\autoref{prop:big_upstairs}.\ref{ketto}, which is presented right after it.

\begin{notation}
  \label{not_multiproduct}
  Assume that we are in the situation of \autoref{not:main-setup}, in particular,
  recall the definition $D_i=D_{c_i}$.  To simplify the notation we also set
  $D_0:=X$.  For a fixed positive natural number $r\in \bN_+$ consider a partition of
  $r$: i.e., a set of natural numbers $r_i\in\bN$ for $i=0,\dots,n$ such that
  $\sum_{i=0}^nr_i=r$. We will denote a partition by $[r_0,r_1,\dots,r_n]$.
  For $[r_0,r_1,\dots,r_n]$ we define the following \emph{mixed product} (we omit $Y$
  from the notation for sanity):
  $$
  D_{\kdot}^{(r_0,r_1,\dots,r_n)} := \left( \mtimes{i=0}nY D_i^{(r_i)} \right)_{\red}
  = \left( D_0^{(r_0)}\times_Y\dots\times_Y D_n^{(r_n)} \right)_{\red}.
  $$
  Observe that $D_\kdot^{(r_0,r_1,\dots,r_n)}$ is naturally a closed subscheme of
  $X^{(r)}_Y$.

  Let us assume now that $r_j>0$ for some $j$. Then
  $[r_0+1,r_1,\dots,r_j-1,\dots,r_n]$ is another partition of the same $r$ and
  $$
  D_{\kdot}^{(r_0,r_1,\dots,r_n)} \subset
  D_{\kdot}^{(r_0+1,r_1,\dots,r_j-1,\dots,r_n)}
  $$
  is a reduced effective Weil divisor no component of which is contained in the
  singular locus of $D_{\kdot}^{(r_0+1,r_1,\dots,r_j-1,\dots,r_n)}$ according to
  \autoref{lem:not_too_ugly}. In particular, for a sequence of partitions,
  \begin{multline*}
    [r_0,r_1,r_2,\dots,r_n], [r_0+1,r_1-1,r_2,\dots,r_n], \dots,
    [r_0+r_1,0,r_2,\dots,r_n],\\ [r_0+r_1+1,0,r_2-1,\dots,r_n],\dots,
    [r_0+r_1+r_2,0,0,\dots,r_n],\dots\\
    [r_0+\dots+r_{n-1},0,\dots,0,r_n], [r_0+\dots+r_{n-1}+1,0,\dots,0,r_n-1], \dots,
    [r,0,\dots,0,0]
  \end{multline*}
  we obtain a filtration of $X^{(r)}$ where each consecutive embedding is a reduced
  effective Weil divisor in the subsequent member of the filtration and furthermore
  no component of the former is contained in the singular locus of the latter:
  \begin{multline*}
    D_{\kdot}^{(r_0,r_1,r_2,\dots,r_n)}\subset
    D_{\kdot}^{(r_0+1,r_1-1,r_2,\dots,r_n)}\subset \dots \subset
    D_{\kdot}^{(r_0+r_1,0,r_2,\dots,r_n)} \subset \\ \subset
    D_{\kdot}^{(r_0+r_1+1,0,r_2-1,\dots,r_n)} \subset \dots \subset
    D_{\kdot}^{(r_0+r_1+r_2,0,0,\dots,r_n)}\subset \dots \\ \subset
    D_{\kdot}^{(r_0+\dots+r_{n-1},0,\dots,0,r_n)}\subset
    D_{\kdot}^{(r_0+\dots+r_{n-1}+1,0,\dots,0,r_n-1)}\subset \dots
    \subset D_{\kdot}^{(r,0,\dots,0,0)} = X^{(r)}.
  \end{multline*}
  In fact, using \autoref{lem:not_too_ugly}, one can see that for every (not
  necessarily subsequent) pair $D' \subseteq D''$ of schemes appearing in the above
  filtration, $D''$ is regular at the generic points of $D'$. Indeed, according to
  \autoref{lem:not_too_ugly} every generic point $\xi$ of $D'$ is over the generic
  point $\eta$ of $Y$. Hence it is enough to see that $D_\eta''$ is regular at
  $\xi$. Observe, that $D_\eta''$ is a product over $\Spec k(\eta)$, and not over a
  positive dimensional scheme as $D''$ is. Hence it is enough to see that all the
  components of $D_\eta''$ are regular at the appropriate projection of $\xi$.
  However, this follows immediately from our definition of stable families
  (\autoref{def:stable_lof_family}), that is, by the assumption that $D_i$ avoid the
  codimension one singular points of the fibers and hence in particular of
  $X_{\eta}$.
\end{notation}


\begin{proof}[Proof of \autoref{prop:big_upstairs}.\ref{ketto}]
  %
  We will use the setup established in \autoref{not:main-setup} and
  \ref{not_multiproduct}.  As before, $f_* \sL_d $ is a nef vector bundle by
  \autoref{lem:pushforward_nef}. Therefore, by the surjective natural map $f^* f_*
  \sL_d \to \sL_d$, $K_{X/Y} + D$ is nef as well. Clearly the same holds for
  $K_{X^{(j)}/Y} + D_{X^{(j)}}$ for any integer $j>0$.


  Now, let $r_0:= \rk f_* \sL_d$ and for $i>0$ let $r_i:= \rk \left( f|_{D_i}
  \right)_* \sL_d|_{D_i} $.  Furthermore, set $r:=\sum_{i=0}^n r_i$, $Z:=
  D^{(r_0,r_1,\dots,r_n)}_{\kdot}$ and $\eta :\wt{Z}\to Z$ the normalization of
  $Z$. Note that $Z$ can be reducible and a priori even non-reduced, but it is a
  closed subscheme of $X^{(r)}$, its irreducible components dominate $Y$ and
  non-reducedness on $Z$ may happen only in large codimension by
  \autoref{lem:not_too_ugly}.

  Consider the natural injection below, which can be defined first over the big open
  set $U \subseteq Y$ of \autoref{not:main-setup}, and then extended reflexively to $Y$,
  \begin{equation}
    \label{eq:det_to_tensor}
    \begin{aligned}
      \xymatrix@C=4em{%
        \displaystyle%
        \iota_d: \sA_d:=\det \left( f_* \sL_d \right) \otimes \left(
          \bigotimes_{i=1}^n \det \left( \left( f|_{D_i} \right)_* \sL_d|_{D_i}
          \right) \right)
        \ar@{^(->}[r] &  \hskip10em
      }
      \\
      \xymatrix@C=4em{%
        \hskip4em
        \ar@{^(->}[r] & \displaystyle%
        \bigotimes^{r_0}_{1} f_* \sL_d \otimes \left[ \bigotimes_{i=1}^n \right]
        \left(
          \bigotimes_{j=1}^{r_i} 
          \left( f|_{D_i} \right)_* \sL_d|_{D_i} \right)
        \simeq
        \underbrace{\left( \left( \left. f^{(r)} \right|_Z \right)_*
            \left. \sL_d^{(r)}\right|_Z \right)^{**}.}_{ \text{iterated use of
            \autoref{lem:pushforward_tensor_product_isomorphism}}}%
        }
   \end{aligned}
    \end{equation}

  By a slight abuse of notation we will denote the composition of restriction from
  $X^{(r)}$ to $Z$ and the pull-back via the normalization morphism $\eta: \wt Z\to
  Z$ by restriction to $\wt Z$. In other words we make the following definition:
  $$
  (\_ )|_{\wt Z} := \eta^* \circ (\_)|_Z
  $$
  So, for instance, $\left( \left. f^{(r)} \right|_{\wt{Z}} \right)^*$ denotes the
  pulling back by the composition $\xymatrix@!C=3ex{\wt{Z} \ar[r]^-\eta & Z\,
    \ar@^{(->}[r] & X^{(r)} \ar[r]^-f & Y}$.

  As in its definition above, if we restrict $\iota_d$ to $U$, then the reflexive
  hulls are unecessary on the right hand side of \autoref{eq:det_to_tensor}. Then by
  adjointness we obtain a non-zero homomorphism
  \begin{equation*}    
    \left.\left( \left. f^{(r)} \right|_{Z} \right)^* \sA_d\right|_{U} \to
    \left. \sL_d^{(r)}\right|_{\left(\left. f^{(r)}\right|_Z\right)^{-1}U }.
  \end{equation*}
Pulling this further back over $\widetilde{Z}$ yields a non-zero homomorphism
  \begin{equation}
    \label{eq:4}
    \left.\left( \left. f^{(r)} \right|_{\wt{Z}} \right)^* \sA_d\right|_{U} \to
    \left. \sL_d^{(r)}\right|_{\left( \left. f^{(r)} \right|_{\wt{Z}}\right)^{-1}U}.  
  \end{equation}
  Since $Z \to Y$ and hence also $\wt{Z} \to Y$ is an equidimensional morphism,
  $\left( \left. f^{(r)} \right|_{\wt{Z}} \right)^{-1}U$ is also a big open set in
  $\wt Z$ and hence \autoref{eq:4} induces a non-zero homomorphism
  \begin{equation}
    \label{eq:non_zero_map_on_X_r}
    \left( \left. f^{(r)} \right|_{\wt{Z}} \right)^* 
    \sA_d \to
    \left. \sL_d^{(r)} \right|_{\wt{Z}}. 
  \end{equation}


  The non-zero map \autoref{eq:non_zero_map_on_X_r} induces another non-zero map
  \begin{equation*}
    \left. \sL_d^{(r)} \right|_{\wt{Z}}     \otimes\left( \left. f^{(r)}
      \right|_{\wt{Z}} \right)^* \sA_d  \to \left. \left( \sL_d^{(r)}
      \right)^{\otimes 2} \right|_{\wt{Z}},  
  \end{equation*}
  where on the left hand side we have a relatively ample and nef line bundle tensored
  with the pullback of a big line bundle. Hence the line bundle on the left hand side
  is big on every component of $\wt{Z}$. Therefore the line bundle on the right hand
  side is big on at least one component.  Let $L^{(r)}$ denote a Cartier divisor
  corresponding to $\sL_d^{(r)}$. Then by the nefness of $L^{(r)}$ it follows that
  \begin{equation*}
    0< \left. L^{(r)} \right|_{\wt{Z}}^{\dim {\wt{Z}}},
  \end{equation*}
  and then also
  \begin{equation}
    \label{eq:positive}
    0< \left. L^{(r)} \right|_{Z}^{\dim {Z}}.
  \end{equation}

  Next we will define a filtration starting with $X^{(r)}$ and ending with $Z$ where
  each consecutive member is a reduced divisor in the previous member.  Recall that
  $r=\sum_{i=0}^nr_i$ and observe that
  for any integer $r_0\leq t< r$ there is a unique $0\leq j< n$ such that
  $$
  \sum_{i=0}^jr_i \leq t < \sum_{i=0}^{j+1}r_i.
  $$
  and hence 
  $$
  0\leq t_{j+1}:= t- \sum_{i=0}^jr_i < r_{j+1}.
  $$
  Now recall \autoref{not_multiproduct} and let us define $Z_r:=X^{(r)}$ and for any
  $t$, $r_0\leq t<r$,
  $$
  Z_t:= D_{\kdot}^{\text{\small $(\sum_{i=0}^jr_i+t_{j+1},
      \overbrace{0,\dots,0}^{\text{$j$ times}}, r_{j+1}-t_{j+1}, r_{j+2},
      \dots,r_n)$}}.
  $$
  Notice that $Z_{r_0}=Z$ and that for all $t$, $r_0\leq t< r$, $Z_t\subset Z_{t+1}$
  is a reduced effective divisor without components contained in the singular locus
  of $Z_{t+1}$ (see \autoref{not_multiproduct} for explanation). Note that set
  theoretically $Z_t$ is the intersection of $Z_{t+1}$ with $p_t^* D_{j+1}$. We claim
  that this is in fact true also divisorially. Indeed, $Z_t$ is reduced and by
  \autoref{lem:not_too_ugly} it is equidimensional. So, it is enough to check that
  $Z_t$ and the divisorial restriction $p_t^* D_{j+1}$ agrees at all codimension one
  points $\xi$ of $Z_{t+1}$.  If $p_t^* D_{j+1}$ contains $\xi$ in its support, then
  $D_{j+1}$ contains $p_t(\xi)$, hence $p_t(\xi)$ has to be a codimension $1$ regular
  point of $X$ lying over the generic point $\eta$ of $Y$.  Note $\mult_{\xi} p_t^*
  D_{j+1} = \mult_{p_t(\xi)} D_{j+1}=1$, and that $Z_{j+1}$ contains exactly the same
  codimension one points of $Z_{t+1}$, which concludes our claim that
  \begin{equation}
    \label{eq:restriction}
    Z_t = p_t^* D_{j+1}|_{Z_{t+1}}. 
  \end{equation}


  Our goal is to show that
  \begin{equation*}
    0< \left(L^{(r)} \right)^{\dim X^{(r)}} \left(= \left(L^{(r)}
      \right)^{\dim Z_r} \right). 
  \end{equation*}
  For any rational number $1 \gg \varepsilon>0$ we have
  \begin{multline*}
    \left(L^{(r)} \right)^{\dim Z_r}
    = \left(L^{(r)} \right)^{\dim Z_r} + \sum_{t=r_0}^{r-1} \varepsilon^{r-j} \left(
      \left. L^{(r)} \right|_{Z_t}^{\dim Z_t} - \left. L^{(r)} \right|_{Z_t}^{\dim
        Z_t} \right) =
    \\ = \left. L^{(r)} \right|_{Z}^{\dim Z} + \sum_{t=r_0}^{r-1} \varepsilon^{r-j-1}
    \left( \left. L^{(r)} \right|_{Z_{t+1}}^{\dim Z_{t+1}} - \varepsilon
      \left. L^{(r)} \right|_{Z_t}^{\dim Z_t} \right).
  \end{multline*}
  Thus, according to \autoref{eq:positive}, it is enough to prove that for each
  integer $r_0 \leq t < r$,
  \begin{equation}
    \label{eq:inequality_goal}
    0 \leq   \left. L^{(r)} \right|_{Z_{t+1}}^{\dim Z_{t+1}} -
    \varepsilon \left. L^{(r)} \right|_{Z_t}^{\dim Z_t}   .
  \end{equation}
  In the rest of the proof we fix an integer $r_0 \leq t < r$, and prove
  \autoref{eq:inequality_goal} for that value of $t$. Let ${\wt Z_{t+1}}$ be the
  normalization of $Z_{t+1}$, and let $S$ be the strict transform of $Z_t$ in ${\wt
    Z_{t+1}}$. Denote by $\rho$ the composition ${\wt Z_{t+1}} \to Z_{t+1} \to
  X^{(r)}$. According to the discussion in \autoref{not_multiproduct}, ${\wt Z_{t+1}}
  \to Z_{t+1}$ is an isomorphism at the generic point of $Z_t$. Hence it is enough to
  prove that
  \begin{equation*}
    0 \leq  \left( \rho^* L^{(r)} \right)^{\dim {\wt Z_{t+1}}} - \varepsilon
    \left( \left. \rho^* L^{(r)} \right|_S \right)^{\dim {\wt Z_{t+1}}-1}  =
    \left( \rho^* L^{(r)} \right)^{\dim {\wt Z_{t+1}}-1} \cdot \left( \rho^*
      L^{(r)} - \varepsilon S \right).
  \end{equation*}
  Note that the right most expression is the intersection of several Cartier divisors
  with a Weil $\bQ$-divisor, and hence it is well-defined.  Furthermore, since
  $\rho^* L^{(r)}$ is nef, to prove the above inequality it is enough to prove that
  the $\bQ$-divisor $\left( \rho^* L^{(r)} - \varepsilon S \right)$ is
  pseudo-effective on every component of $\wt{Z}_{t+1}$. This follows if we apply
  \autoref{lem:long} by setting $Z:=Z_{t+1}$, $\wt{Z}:=\wt{Z}_{t+1}$, $E:=p_t^*
  D_{j+1}$ and by using \autoref{eq:restriction} (and its implication that $S= p_t^*
  D_{j+1}|_{\wt{Z}_{t+1}}$).
\end{proof}

Recall that a $\bQ$-Weil divisor $D$ is called $\bQ$-effective if $mD$ is linearly
equivalent to an effective divisor for some integer $m>0$.

\begin{lemma}\label{lem:long}\ 
  \begin{enumerate}
  \item Let $f : (X,D) \to Y$ be an equidimensional, surjective, projective morphism
    from a semi-log canonical pair onto a smooth projective variety, such that
    $K_{X/Y} + D$ is $f$-ample and all irreducible components of $X$ dominate $Y$.
  \item Let $Z$ be a closed subscheme of $X$, which is equidimensional over $Y$,
    reduced, and all its irreducible components dominate $Y$.
  \item Let $E$ be a reduced effective divisor on $X$ with support in $\Supp D$, in
    particular, no component of $E$ is contained in the singular locus of $X$.
    Assume that $E$ does not contain any component of $Z$ and that both $Z$ and $X$
    are regular at the generic points of $Z$ and at the codimension one points of $Z$
    that are contained in $E$.
  \item Let $\wt{Z} \to Z$ be the normalization.
  \end{enumerate}
  Then $(K_{X/Y} + D - \varepsilon E)|_{\wt Z}$ is pseudo-effective for every
  $\varepsilon\in\bQ$, $0 < \varepsilon \ll 1$, meaning that for any fixed ample
  divisor $A$ on ${\wt Z}$, $(K_{X/Y} + D - \varepsilon E)|_{\wt Z} + \delta A$ is
  $\bQ$-effective on every component of ${\wt Z}$ for every $\delta\in\bQ, 0 < \delta
  \ll 1$.
\end{lemma}

\begin{remark}
  In the above statement $E|_{\wt Z}$ is defined by considering the (big) open locus
  in ${Z}$, where ${E}$ is Cartier, pulling back to $\wt Z$ and taking the closure
  there using that the complement has codimension at least $2$.
\end{remark}

\begin{proof}\addtocounter{theorem}{-1}
  \textbf{Reduction step:} Let $\pi: (\oX, \oD) \to (X,D)$ be the normalization and
  $\oZ$ and $\oE$ the strict transforms (by the regularity assumptions $\pi$ is an
  isomorphism at all generic points of $\oZ$ and $\oE$ so these strict transforms are
  meaningful). Since $\widetilde{Z} \to Z$ factors through $\oZ \to Z$, this setup shows
  that we may assume that $(X, D)$ is log canonical.

  \textbf{Summary of assumptions after the reduction step:}
  \begin{enumerate}
  \item $f : (X,D) \to Y$ is an equidimensional, surjective, projective morphism from
    a log canonical pair onto a smooth projective variety, such that $K_{X/Y} + D$ is
    $f$-ample,
  \item $Z$ is equidimensional over $Y$, reduced, and all its irreducible components
    dominate $Y$,
  \item $\Supp E \subseteq \Supp D$,
  \item no irreducible component of $Z$ is contained in the support of $E$, and 
  \item regularity assumptions: $X$ is regular at the generic points of $Z$ and both
    $E$ and $Z$ are regular at the codimension one points of ${Z}$ that are contained
    in $E$.
  \end{enumerate}

  \textbf{The argument.}  Set $L:=K_{X/Y} + D$, $\sL:= \sO_X(L)$ and $S:= E|_{\wt Z}$
  and let $\rho$ be the composition ${\wt Z} \to Z \to X$.
  Note that to establish that $\rho^* L - \varepsilon S$ is pseudo-effective one may
  use an arbitrary Cartier divisor $A$ on ${\wt Z}$, and show that $\rho^* L -
  \varepsilon S + \delta A$ is $\bQ$-effective on every component for every $ 0 <
  \delta \ll 1$.  Indeed, choosing an ample $A'$, it follows that $tA' -A$ is
  effective on every component for some $t \gg 0$, and hence then
  \begin{equation*}
    \rho^* L - \varepsilon S + \delta tA' =   \rho^* L - \varepsilon S +
    \delta A +  \delta (tA' -A) 
  \end{equation*}
  is also $\bQ$-effective on every component as well.  Here we will choose $A$ to be
  the pullback of an appropriate ample line bundle on $Y$.

  Let us take a $\bQ$-factorial dlt model $\tau: (T, \Theta) \to ( X, D )$ such that
  $K_T+\Theta=\tau^*(K_X+D)$ (cf.\
  \cite[3.1]{Kollar_Kovacs_Log_canonical_singularities_are_Du_Bois}) and define $g:=
  f \circ \tau$. Note that $\tau$ is an isomorphism both at the generic points of $Z$
  and at the codimension one points of $Z$ that are contained in $E$, since $X$ is
  regular at all these points. Set $\Gamma:= \tau^{-1}_* E$.  Consider
  \begin{equation*}
    q \tau^* L -  \Gamma  =  q \left( K_{T/Y} + \Theta -
      \frac{1}{q}\Gamma \right). 
  \end{equation*}
  for a divisible enough integer $q>0$.  There are two important facts about the
  above divisor. On one hand,
  \begin{equation}
    \label{eq:containment}
    \tau_* \sO_T ( q \tau^* L - \Gamma ) \subseteq \sO_X( q L -   E),
  \end{equation} 
  on the other hand, the above divisor is the $q^\text{th}$ multiple of the relative
  log-canonical divisor of a dlt pair. Hence according to \cite[Thm
  1.1]{Fujino_Notes_on_the_weak_positivity_theorems}, for every divisible enough $q$,
  \begin{equation*}
    g_* \sO_T ( q \tau^* L - \Gamma ) 
  \end{equation*}
  is weakly positive.  Therefore after fixing an ample line bundle $H$ on $Y$, for
  each $a>0$, there is a $b>0$, such that
  \begin{equation*}
    \sym^{[ab]}  ( g_* \sO_T ( q  \tau^* L  - \Gamma   ) ) \otimes H^b 
  \end{equation*}
  is generically globally generated.

  Let $U$ be the open set where both $g_* \sO_T ( q \tau^* L - \Gamma )$ and $f_*
  \sO_X( qL- E )$ are locally free. Over $U$ consider the composition of the
  following homomorphisms, where the left most one is the push-forward of the
  embedding in \autoref{eq:containment}:
  \begin{equation}
\label{eq:composition_wp}
    f^* \sym^{ab} ( g_* \sO_T  ( q  \tau^* L  -   \Gamma
    )) \to     f^* \sym^{ab} ( f_* \sO_X( qL-E ))
 \to  f^*  f_* \sO_X( ab (qL-E) )
    \to \sO_X( ab(qL-E)). 
  \end{equation}
  Let us pause for a moment and recall that $qL- E$ is not necessarily Cartier in
  general. However, it is Cartier over a big open set of $f^* U$, so the natural map
  $\sym^{ab} ( f_* \sO_X( qL-E )) \to f_* \sO_X( ab(qL-E) )$, which yields the middle
  arrow above, can still be constructed over that big open set and then extended
  uniquely, since $X$ is normal.

  Setting $h := f \circ \rho$, still over $U$, we obtain the following natural
  morphism by pulling back the composition of \autoref{eq:composition_wp} via $\rho$.
  \begin{equation*}
    h^* \sym^{ab} ( g_* \sO_T ( q  \tau^* L  -   \Gamma )  
    ) \to 
     \sO_{\wt Z} ( ab(q\rho^*L- S)). 
  \end{equation*}
  Again, note that $qL- E$ is not necessarily Cartier over $Z$. However, by our
  regularity assumption it is Cartier over a big open set $U_Z$ of $Z$. So the above
  map is constructed first over $\rho^{-1}(U_Z \cap f^{-1} U)$ and then extended
  uniquely using that $\widetilde{Z}$ is normal.

  So, since ${\wt Z} \to Y$ is equidimensional, $h^{-1} U$ is a big open set of $\wt
  Z$. In particular, we obtain a homomorphism
  \begin{equation}
    \label{eq:difficult_map}
    h^{[*]}  \sym^{[ab]} ( g_* \sO_T ( q  \tau^* L  - \Gamma   ) ) \otimes h^* H^b \to
    \sO_{\wt Z} ( ab (q  \rho^*  L -  S )) \otimes h^* H^b.  
  \end{equation}
  Now choose $q$ divisible enough so that $\tau_* \sO_T(q \tau^* L - \Gamma) \simeq
  \sO_X(qL) \otimes \tau_* \sO_T(- \Gamma)$ is $f$-globally generated (recall that
  $L$ is $f$-ample). Note that the ideal $\tau_* \sO_T(- \Gamma)$ is supported on
  $\Supp E$ and $\Supp E $ does not contain any component of $Z$ by
  assumption. Hence, it follows that the natural map
  \begin{equation*}
    h^* g_* \sO_T ( q  \tau^* L  - \Gamma   ) )  \to \sO_{\widetilde{Z}} (q  \rho^*
    L -  S ) 
  \end{equation*}
  is surjective at all generic points of $\widetilde{Z}$ and then the same holds for
  the map in \autoref{eq:difficult_map}.  Furthermore, the sheaf on the left hand
  side in \autoref{eq:difficult_map} is globally generated at every generic point of
  ${\wt Z}$. This gives us the desired sections of $\sO_{\wt Z} ( ab (q \rho^* L - S
  )) \otimes h^* H^b$ and concludes the proof.
\end{proof}\addtocounter{theorem}{1}

\noindent
We will need the following analog of \autoref{lem:equivalent-props-of-big} for
reducible schemes.

\begin{lemma}
  \label{lem:big_on_one_component}
  If $X$ is a projective scheme of pure dimension $n$ over $k$ and $L$ a nef Cartier
  divisor which is big on at least one component (that is, $L^n>0$), then for every
  Cartier divisor $D$ that does not contain any component of $X$, $L - \varepsilon D$
  is $\bQ$-effective for every rational number $0 < \varepsilon \ll 1$ (however the
  corresponding effective divisor may be zero on every irreducible component but
  one).
\end{lemma}

\begin{proof}
  Let $\sL:= \sO_X(L)$. Consider the exact sequence,
  \begin{equation*}
    \xymatrix{
      0 \ar[r] & \sL^a (-D) \ar[r] & \sL^a \ar[r] & \sL^a|_D \ar[r] & 0
    } 
  \end{equation*}
  Since $L$ is nef, by the asymptotic Riemann-Roch Theorem \cite[Corollary
  1.4.41]{Lazarsfeld_Positivity_in_algebraic_geometry_I}, $h^0(L^a) = \frac{a^n}{n!}
  L^n + O(a^{n-1})$.  Furthermore, $h^0(\sL^a|_D) = O(a^{n-1})$. Hence, for every
  $a\gg 0$ $H^0(\sL^a(-D)) \neq 0$.
\end{proof}

\noindent
\autoref{thm:big_higher_dim_base}.\ref{egy} is an immediate consequence of the
following statement.

\begin{proposition}
  \label{prop:big_higher_dim_base}
  If $f : (X, D) \to Y$ is a family of stable log-varieties of maximal variation over
  a normal proper variety, then there exists an integer $q>0$ and a proper closed
  subvariety $S \subseteq Y$, such that for every integer $a>0$, and closed
  irreducible subvariety $T \subseteq Y$ not contained in $S$, $\det f_* \sO_X(aq
  (K_{X/Y} + \Delta))|_{\wt{T}}$ is big, where $\wt{T}$ is the normalization of $T$.
\end{proposition}


\begin{proof}
  First, note that since $q$ can be chosen to be divisible enough, $f_* \sO_X(aq
  (K_{X/Y} + \Delta))$ commutes with base-change, and hence we may replace $Y$ by any
  of its resolution. That is, we may assume that $Y$ is smooth and projective.  We
  may also replace $\wt{T}$ by a resolution of $T$ in the statement.

  Let $H$ be any ample effective Cartier divisor on $Y$, and let $\sH:=\sO_Y(H)$ be
  the associated line bundle. Let $r>0$ be the integer given by
  \autoref{thm:big_higher_dim_base}.\ref{ketto}.  Since every component of $X^{(r)}$
  dominates $Y$, according to \autoref{lem:big_on_one_component}, $q(K_{X^{(r)}/Y} +
  D_{X^{(r)}}) - \left(f^{(r)} \right)^*H $ is linearly equivalent to an effective
  divisor for some multiple $q$ of $dm$. Equivalently, there is a non-zero map
  \begin{equation}
    \label{eq:embedding_proving_det}
    \big(f^{(r)}\big)^* \sH \to \sO_{X^{(r)}} \left(q \left(K_{X^{(r)}/Y} + D_{X^{(r)}}
      \right) \right). 
  \end{equation}
  Let $S \subseteq Y$ be the (proper) closed set over which
  \autoref{eq:embedding_proving_det} is zero. For any integer $a>0$ consider the
  following non-zero map induced by the $a^\text{th}$ tensor power of
  \autoref{eq:embedding_proving_det}.
  \begin{equation}
    \sH^a \simeq f^{(r)}_* \big(f^{(r)}\big)^* \sH^a \to f^{(r)}_*
    \sO_{X^{(r)}} \left(aq \left(K_{X^{(r)}/Y} + D_{X^{(r)}} \right)
    \right) \simeq \underbrace{\bigotimes^r f_* \sO_X (aq (K_{X/Y} + D ) )
    }_{\textrm{\autoref{lem:pushforward_tensor_product_isomorphism}}}\hskip-.5em
  \end{equation}
  This is necessarily an embedding, because $Y$ is integral. Let $\sigma: \wt{T} \to
  Y$ be the resolution of an irreducible closed subset $T$ of $Y$ that is not
  contained in $S$. Then, the induced map
  \begin{equation*}
    \sigma^*\sH^a 
    \to \bigotimes^r \sigma^* f_* \sO_{X} \left(aq \left(K_{X/Y} +
        D \right) \right)  \simeq \bigotimes^r \left( f_{\wt{T}} \right)_*
    \sO_{X_{\wt{T}}} \left(aq 
      \left(K_{X_{\wt{T}}/{\wt{T}}} +         D_{\wt{T}} \right) \right)   
  \end{equation*}
  is not zero and therefore it is actually an embedding. Let $\sB$ denote the
  saturation of $\sigma^* \sH^a$ in $\bigotimes^r \left( f_{\wt{T}}
  \right)_*\sO_{X_{\wt{T}}/{\wt{T}}} (aq (K_{X_{\wt{T}}}/{\wt{T}} + D_{\wt{T}} )
  )$. Then $\sB$ is big since $\sH$ is ample and it induces another exact sequence
  \begin{equation*}
    \xymatrix{
      0 \ar[r] &  \sB \ar[r] & \bigotimes^r \left( f_{\wt{T}} \right)_*
      \sO_{X_{\wt{T}}} \left(aq         \left(K_{X_{\wt{T}}/{\wt{T}}} + D_{\wt{T}}
        \right) \right) \ar[r] & \sG \ar[r] & 0  , 
    }
  \end{equation*}
  where $\sG$ is locally free in codimension one. Since according to
  \autoref{lem:pushforward_nef}, $\left( f_{\wt{T}} \right)_* \sO_{X_{\wt{T}}} (aq
  (K_{X_{\wt{T}}/{\wt{T}}} + D_{\wt{T}} ) )$ is nef, $\sG$ is weakly-positive
  according to \cite[prop 2.9.e]{Viehweg_Quasi_projective_moduli} and point
  \autoref{itm:gen_surj} of \autoref{lem:wp_big_properties}. Note that we cannot
  infer that $\sG$ is nef, since $\sG$ does not have to be locally free. However, we
  can infer that $\det \sG$ is weakly-positive as well by \autoref{itm:tensors} of
  \autoref{lem:wp_big_properties} and then for some $N>0$,
  \begin{equation*}
    \det \left( \bigotimes^r \left( f_{\wt{T}} \right)_* \sO_{X_{\wt{T}}} \left(aq
        \left(K_{X_{\wt{T}}/{\wt{T}}} +  
          D_{\wt{T}} \right) \right)  \right) \simeq  \left(\det \left( f_{\wt{T}}
      \right)_* \sO_X       \left(aq \left(K_{X_{\wt{T}}/{\wt{T}}} + D_{\wt{T}}
        \right) \right) \right)^N \simeq \sB \otimes \det  \sG  
  \end{equation*}
  is big by \autoref{itm:wp_big} of \autoref{lem:wp_big_properties}. This concludes
  the proof.
\end{proof}

\begin{proof}[Proof of \autoref{cor:projective}]
  Let $M$ be the algebraic space in the statement, and $\sM$ the (pseudo-)functor
  that it coarsely represents. First note that by finiteness of the automorphism
  groups (\autoref{prop:finite_automorphism}), an appropriate power of the functorial
  polarization required in \autoref{def:a_functor} descends to $M$. Since $M$ is
  proper by \autoref{prop:proper}, according to the Nakai-Moishezon criterion we only
  need to show that the highest self-intersection of this polarization on every
  proper irreducible subspace of $M$ is positive. However, by
  \autoref{cor:finite_cover} it is enough to show this, instead of $M$, for a proper,
  normal scheme $Z$, that supports a family $f : (X_Z,D_Z) \to Z$ with the property
  that each fiber of $f$ is isomorphic to only finitely many others.

  Let us state our goal precisely at this point: we are supposed to exhibit an $r>0$
  such that for any closed irreducible subvariety $V \subseteq Z$,
  \begin{equation*}
    c_1 \left( \det \left(f_V \right)_* \sO_{X_V} \left( r\left( K_{X_V/V} + D_V \right)
      \right)  \right)^{\dim V} >0.  
  \end{equation*}
  In fact we, are proving something slightly stronger. We claim that \emph{there
    exist an integer $q>0$, such that for every integer $a>0$ and closed irreducible
    subvariety $V \subseteq Z$},
  \begin{equation*}
    c_1\left( \det \left(f_V \right)_* \sO_{X_V} \left( aq\left( K_{X_V/V} +
            D_V \right) 
      \right) \right)^{\dim V} >0.  
  \end{equation*}
  We prove this statement by induction. For $\dim Z=0$ it is vacuous, so 
  we may assume that $\dim Z>0$.
  By \autoref{prop:big_higher_dim_base} there exist a $q_Z>0$ and a closed subset $S
  \subseteq Z$ that does not contain any component of $Z$, such that for every $a>0$
  and every irreducible closed subset $T \subseteq Z$ not contained in $S$, if we set
  $\sN_{aq}:=\det f_* \sO_{X_Z} \left( aq\left( K_{X_Z/Z} + D_Z \right) \right) $,
  then $c_1\left(\sN_{aq_Z}|_T \right)^{\dim T}>0$.  Let $\wt{S}$ denote the
  normalization of $S$. Then by induction, since $\dim S< \dim Z$, there exists a
  $q_{\wt{S}}>0$, such that for every $a>0$ and every irreducible closed subset $V
  \subseteq \wt{S}$, $c_1\big( \sN_{aq_{\wt{S}}}|_V \big)^{\dim V}>0$. Taking $q :=
  \max\{q_Z, q_{\wt{S}} \}$ concludes the proofs of the claim and of
  \autoref{cor:projective}. 
\end{proof}

\begin{remark}
\label{rem:labeling}
If one allows labeling of the components as well, which was excluded up to this point
from \autoref{def:a_functor} for simplicity, then \autoref{thm:big_higher_dim_base}
still yields projectivity as in \autoref{cor:projective} for the unlabeled case. This
follows from the fact that each stable log-variety admits at most finitely many
labelings. Hence, forgetting the labeling of a labeled family with finite isomorphism
equivalence classes yields a non-labeled family with finite isomorphism equivalence
classes. In particular, the proof of \autoref{cor:projective} implies that the
polarization by $\det f_* \sO_X(dm(K_{X/Y} + D))$ yields an ample line bundle on the
base of the labeled family as well.
\end{remark}


\section{Pushforwards without determinants}
\label{sec:pushforward}

The main goal of this section is to prove the following result.

\begin{theorem}
  \label{thm:pushforward}
  If $f : (X,D) \to Y$ is a family of stable log-varieties of maximal variation over
  a normal, projective variety $Y$ with klt general fiber, then $f_*
  \sO_X(q(K_{X/Y}+D))$ is big for every divisible enough integer $q>0$.
\end{theorem}

\begin{remark}
  One might wonder if this could be true without assuming that the general fiber is
  klt.  We will show below that that assumption is in fact necessary.
\end{remark}

\begin{corollary}
  \label{cor:big_total_space}
  If $f : (X,D) \to Y$ is a family of stable log-varieties of maximal variation over
  a normal, projective variety $Y$ with klt general fiber, then $K_{X/Y}+D$ is big.
\end{corollary}

This corollary follows from \autoref{thm:pushforward} by a rather general argument
which we present in the following lemma.

\begin{lemma}
  Let $f:X\to Y$ be a surjective morphism between normal proper varieties and assume
  also that $Y$ is projective. Let $\sL$ be an $f$-big line bundle on $X$ such that
  $f_*\sL$ is a big vector bundle. Then $\sL$ itself is big.
\end{lemma}

\begin{proof}
  Choose an ample line bundle $\sA$ on $Y$ such that $f^*\sA\otimes \sL$ is big. Then 
  by 
  \autoref{def:wp_big} there is a generically isomorphic inclusion for some integer
  $a>0$:
  \begin{equation*}
    \bigoplus \sA \hookrightarrow \Sym^{a}(f_* \sL)
  \end{equation*}
  This induces the following non-zero composition of homomorphisms, which concludes
  the proof:
  \begin{equation*}
    \bigoplus \underbrace{ f^*\sA \otimes \sL}_{\textrm{big}} \hookrightarrow
    f^* \Sym^{a} ( f_* \sL) \otimes \sL \to f^*  f_* (\sL^a) \otimes \sL \to
    \sL^{a+1}. \qedhere 
  \end{equation*}
\end{proof}

\begin{proof}[Proof of \autoref{cor:big_total_space}]
  Take $\sL=\sO_X(q(K_{X/Y}+D))$ for a divisible enough $q>0$.
\end{proof}

Next we show that the klt assumption in \autoref{thm:pushforward} is necessary. 

\begin{example}
  \label{ex:klt-is-needed}
  Let $f:X\to Y$ be an arbitrary non-isotrivial smooth projective family of curves
  over a smooth projective curve. Assume that it admits a section $\sigma: Y\to X$
  (this can be easily achieved after a base change) and let $D=\im\sigma\subset
  X$. This is one of the simplest examples of a family of stable
  log-varieties. Notice that the fibers are log canonical, but not klt.
  By adjunction $K_D=(K_X+D)|_D$ and as $f|_D:D\to Y$ is an isomorphism, it follows
  that $\sO_X(K_{X/Y}+D)|_D\simeq \sO_D$.  The following claim implies that $f_*
  \sO_X(r(K_{X/Y}+D))$ cannot be big for for any integer $r>0$.

\begin{subclaim}
  Let $f:X\to Y$ be a flat morphism, $\sL$ a torsion-free sheaf on $X$, and $\sE$ a
  locally free sheaf on $Y$. Further let $D\subset X$ be the image of a section
  $\sigma: Y\to X$ and assume that $Y$ is irreducible, that $\sL|_D\subseteq \sO_D$,
  and that there exists a homomorphism $\varrho: f^*\sE\to \sL$ such that $\varrho|_D
  \neq 0$.  Then $\sE$ cannot be big.
\end{subclaim}

\begin{proof}
 Since $f|_D$ is an isomorphism, if $\sE$
  were big, so would be $(f^*\sE)|_D$ and then $\varrho|_D$ would imply that $\sO_D$
  is big. This is a contradiction which proves the statement.
\end{proof}
\end{example}

A variant of \autoref{ex:klt-is-needed} shows that even assuming that $D=0$ would not
be enough to get the statement of \autoref{thm:pushforward} without the klt
assumption:

\begin{example}
  \label{ex:klt-is-needed-two}
  Let $f:X\to Y$ be an arbitrary non-isotrivial smooth projective family of curves
  over a smooth projective curve. Assume that it admits two disjoint sections
  $\sigma_i: Y\to X$ for $i=1,2$ and let $D_i=\im\sigma\subset X$. Next glue $X$ to
  itself by identifying $D_1$ and $D_2$ via the isomorphism
  $\sigma_1\circ\sigma_2^{-1}$ and call the resulting variety $X'$. Then the induced
  $f':X'\to Y$ is a family of stable varieties. The same computation as above shows
  that $f'_* \sO_{X'}(rK_{X'/Y})$ cannot be big for any $r>0$ for this example as
  well. For computing the canonical class of non-normal varieties see
  \cite[5.7]{Kollar_Singularities_of_the_minimal_model_program}.
\end{example}

A variant of the above examples can be found in \cite[Thm.\
3.0]{Keel_Basepoint_freeness_for_nef_and_big_line_bundle_in_positive_characteristics},
for which not only $K_{X/Y} +D$ is numerically trivial on a curve $C$ contained in
$D$ (and hence other ones can be constructed where the same happens over the double
locus), but $K_{X/Y} +D|_C$ is not even semi-ample.

One might complain that in \autoref{ex:klt-is-needed-two} the fibers are not
normal. One can construct a similar example of a family of stable varieties where the
general fiber is log canonical (and hence normal) that shows that the klt assumption
is necessary, but this is a little bit more complicated.

\begin{example}
  \label{ex:klt-is-needed-three}
  Let $Z$ be a projective cone over a genus $1$ curve $C$. Assume that $Z \subseteq
  \bP^3$ is embedded compatibly with this cone structure, that is, via this
  embedding, $Z \cap \bP^2 = C$ for some fixed $\bP^2 \subseteq \bP^3$. Fix also
  coordinates $x_0,\dots, x_3$ such that $x_1, x_2, x_3$ are coordinates for $\bP^2$
  and the cone point is $P:=[1,0,0,0]$. Choose two general polynomials $f(x_1, x_2,
  x_3)$ and $g(x_0,x_1,x_2,x_3)$. Consider the pencil of hypersurfaces in $Z$ defined
  by these two equations. This yields a hypersurface $\sD \subseteq Z \times \bP^1$
  with $\sD_0= V(f) \cap Z$ a general conic hypersurface section of $Z$ and
  $\sD_\infty = V(g) \cap Z$ a general hypersurface section of $Z$. Since $g$ was
  chosen generally, $P \not\in \sD_\infty$. On the other hand, $P \in \sD_0$, and
  hence $P \not\in \sD_t$ for $t \neq 0$. Furthermore, since in codimension $1$
  hypersurface sections of $Z$ disjoint from $P$ acquire only nodes 
  $\sD_t$ is either smooth or has only nodes for $t \neq 0$. Hence, for $d \gg 0$ the
  family $(Z \times \bP^1, \sD) \to \bP^1$ is a family of stable log-varieties
  outside $t=0$. For $t=0$ we run stable reduction. Since the stable limit is unique,
  we may figure out the stable limit without going through the meticulous process by
  hand: it is enough to exhibit one family that is isomorphic in a neighborhood of
  $0$ to the original family after a base-change and which does have a stable limit.
  The pencil $\sD$ around $t=0$ is described by the equation $f(x_1, x_2, x_3) + t
  g(x_0,x_1,x_2,x_3)$. Extract a $d$-th root from $t$ and denote the new family also
  by $(Z \times \Spec k[t], \sD)$ (i.e., we keep the same notation for the
  boundary). Then $\sD$ around $t=0$ is described by the equation
  \begin{equation*}
    F_1(t,x_0,x_1,x_2,x_3):=f(x_1, x_2, x_3) + t^d g(x_0,x_1,x_2,x_3). 
  \end{equation*}
  Now set
  \begin{equation*}
    F_2(t,x_0,x_1,x_2,x_3):=f(x_1, x_2, x_3) + t^d g(x_0/t,x_1,x_2,x_3),
  \end{equation*}
  and let $\sD'$ be the hypersurface of $Z \times \Spec k[t]$ defined by $F_2$. Then
  in a punctured neighborhood of $t=0$, $(Z \times \Spec k[t], \sD) $ is isomorphic
  to $(Z \times \Spec k[t], \sD' )$, via the map
  \begin{equation*}
    x_i \mapsto x_i (i \neq 0) \qquad t \mapsto t \qquad x_0 \mapsto t \cdot x_0 .
  \end{equation*}
  Here the key is that Z, being a cone, is invariant under scaling by $x_0$. Note
  that since $g$ is general, $x_0^d$ has a non-zero coefficient, say $c$. Then it is
  easy to see that $F_2(0, x_1,x_2, x_3) = f(x_1,x_2,x_3) + c x_0^d$. That is,
  $\sD_0'$ is a $d$-th cyclic cover of $V(f) \cap C \subseteq \bP^2$ in $Z$. Since
  $f$ is general, $V(f) \cap C$ is smooth (i.e., a union of reduced points), and
  hence $\sD_0'$ is also smooth. Furthermore, $\sD_0'$ avoids $P$. It follows that
  $(Z, \sD_0')$ is log canonical, whence stable and therefore it has to be the
  central fiber of the stable reduction.

  Summarizing, after the stable reduction, we obtain a family $(\sZ, \sD) \to Y$ of
  stable log-pairs over a smooth projective curve (we denote the divisor by $\sD$
  here as well for simplicity), such that $\sZ_y \simeq Z$ and $\sD_y$ avoids the
  cone point in $\sZ_y$ for each $y \in Y$. Note that $\sZ$ cannot be isomorphic to
  $Y \times Z$ anymore (not even after a proper base-change), since then $\sD_y$
  would give a proper family of moving divisors in $Z$ that does not contain
  $P$. This is impossible, since a proper family covers a proper image, which would
  have to be the entire $Z$.

  In any case, after possibly a finite base-change, we are able to take the cyclic
  cover of $\sZ$ of degree $d$ ramified along $\sD$.  For $d \gg 0$ the obtained
  family $X \to Y$ is stable of maximal variation over the projective curve $Y$. It
  has elliptic singularities along a curve $B$ that covers $d$ times the singularity
  locus of $\sZ \to Y$. Hence, $B \to Y$ is proper and has $d$ preimages over each
  point. In particular it is \'etale (though $B$ might be reducible). If we blow-up
  $B$, and resolve the other singular points as well (which are necessarily
  disjoint from $B$, since they originate from the nodal fibers of $\sD \to Y$), we
  obtain a resolution $\pi:V \to X$. Let $E$ be the (reduced) preimage of $B$. Then
  we have that $K_{V/Y} + E + F \equiv \pi^* K_{X/Y}$, where $F$ is exceptional and
  disjoint from $E$. In particular then
  \begin{equation*}
    K_{E/Y} \equiv (K_{V/Y} + E) |_E \equiv (K_{V/Y} + E + F)|_E \equiv \pi^* K_{X/Y}|_E
    \equiv \left(\pi|_E \right)^* \left( K_{X/Y}|_B \right). 
  \end{equation*}
  Hence it is enough to show that $K_{E/Y} \equiv 0$ (since then we have found a
  horizontal curve over which $K_{X/Y}$ is numerically trivial). Since $B \to Y$ is
  \'etale, it is enough to show that $K_{E/B} \equiv 0$. However $E \to B$ is a
  smooth family of isomorphic genus one curves. In particular, after a finite
  base-change we may also assume that it has a section, in which case we do know that
  its relative canonical sheaf is numerically trivial. However, then it is
  numerically trivial even without the base-change. 
  It follows that $K_{X/Y}|_B$ is numerically trivial and the same argument as above
  shows that then it cannot be big. 
\end{example}

Recall that if $(X,D)$ is a klt pair and $\Gamma$ a $\bQ$-Cartier divisor, then the
log canonical threshold is defined as
\begin{equation*}
  \sup \{ t | (X, D+ t \Gamma) \textrm{ is log canonical } \}.
\end{equation*}

\begin{lemma}
  \label{lem:lct_semicontinuous_divisor}
  The log canonical threshold is lower semi-continuous in projective, flat families
  with $\bQ$-Cartier relative log canonical bundle. That is, if $f : (X, D) \to S$ is
  a projective, flat morphism with $S$ normal and essentially of finite type over $k$
  such that $K_{X/S} +D$ is $\bQ$-Cartier, $(X_s, D_s)$ is klt for all $s \in S$ and
  $\Gamma \geq 0$ is a $\bQ$-Cartier divisor on $X$ not containing any fibers, then
  $\lct(\Gamma_s; X_s, D7_s)$ is lower semi-continuous.

  Furthermore, if $S$ is regular, then for every $s \in S$ there is a neighborhood
  $U$ of $s$, such that 
  \begin{equation*}
    \lct(\Gamma|_{f^{-1} U}; f^{-1} U , D|_{f^{-1} U}) \geq
    \lct(\Gamma_s; X_s, D_s).
  \end{equation*}
\end{lemma}

\begin{proof}
  Let us first show the second statement, which is an application of inversion of
  adjunction. Let $A = f^{-1} H$ for some very ample reduced effective divisor
  $H$. Then
  \begin{multline*}
    (A, D|_A + t\Gamma|_A) \textrm{ is lc } 
    \Rightarrow (X, D + t\Gamma + A) \textrm{ is lc in a neighborhood of } A
    \Rightarrow
    \\ \Rightarrow (X, D+ t \Gamma ) \textrm{ is lc in a neighborhood of } A.
  \end{multline*}
  Applying this inductively gives the second statement, since for regular schemes
  every point can be (locally) displayed as the intersection of hyperplanes.

  Next, let us prove that $s \mapsto \lct(\Gamma_s; X_s, D_s)$ is constant on a dense
  open set $U$ and that $U$ can be chosen such that $\lct(\Gamma|_{f^{-1}U}; f^{-1}
  U, D|_{f^{-1}U} )$ agrees with this constant value. For this we may assume that $S$
  is smooth. Take a resolution $\pi : Y \to X$ of $(X,D+ \Gamma)$. By replacing $X$
  with a dense Zariski open set we may assume that all exceptional divisors of $\pi$
  are horizontal and that $f \circ \pi$ is smooth. However, then the discrepancies of
  $(X_s, D_s+ t \Gamma_s)$ agree for all $s \in S$ and $t \in \bQ$ and furthermore,
  this is the same set as the discrepancies of $(X, D+t \Gamma)$.  This concludes our
  claim.

  The above two claims show that we have semi-continuity over smooth curves, and also
  that the function is constructible. These together show that the function is
  semi-continuous in general.
\end{proof}

\begin{definition}
  We define the \emph{log canonical threshold of a line bundle} $\sL$ on a projective
  pair $(X,\Delta)$ as the minimum of the log canonical thresholds of the effective
  divisors in $\bP\left(H^0(X,\sL)^*\right)$, the complete linear system of $\sL$:
  \begin{equation*}
    \lct(\sL;X, \Delta):= \min \left\{ \lct(\Gamma;X, \Delta) \,\big\vert\, \Gamma \in
      \bP\left(H^0(X,\sL)^*\right) \right\}. 
  \end{equation*}
  By the above lemma this minimum exists.
\end{definition}

\begin{lemma}
  \label{lem:lct_semicontinuous_line_bundle}
  The log canonical threshold of a line bundle is bounded in projective, flat
  families. That is, let $f : (X,D) \to T$ be a projective flat morphism with $T$
  normal and essentially of finite type over $k$ and $\sL$ a line bundle on $X$.
  Assume that $(X_t, D_t)$ is klt for all $t \in T$ and $K_{X/T} + D$ is
  $\bQ$-Cartier.  Then there exists a real number $c$, such that $\lct(\sL_t; X_t,
  D_t)>c$ for all $t \in T$.
\end{lemma}

\begin{proof}
  First assume that $f_* \sL$ commutes with base-change (and it is consequently
  locally free) and let $\bP:= \Proj_T ((f_* \sL)^*)$. Notice that the points of
  $\bP_t$ for $t\in T$ may be identified with elements of the linear systems
  $\bP\left(H^0(X,\sL)^*\right)$.  Further let $\Gamma$ be the universal divisor on
  $X \times_T \bP$ corresponding to $\sL$, that is, $(x, [D]) \in \Gamma$ iff $x \in
  D$. Now, applying \autoref{lem:lct_semicontinuous_divisor} to $X \times_T \bP \to
  \bP$ and $\Gamma$ yields the statement.

  In the general case, we work by induction on the dimension of $T$. We can find a
  dense open set over which $f_* \sL$ commutes with base change. So, there is a lower
  bound as above over this open set, and there is another lower bound on the
  complement. Combining the two gives a lower bound over the entire $T$.
\end{proof}

The essence of the argument of the proof of the following proposition was taken from
\cite[Lemma 5.18]{Viehweg_Quasi_projective_moduli}, though the context is slightly
different.

\begin{proposition}
  \label{prop:tricky}
  Let $f : (X,D) \to Y$ be a flat morphism such that $D$ does not contain any fibers,
  $(X_y,D_y)$ is klt for a fixed $y \in Y$, $K_X+D$ is $\bQ$-Cartier and $Y$ is
  smooth. Let $\Gamma$ be a $\bQ$-Cartier effective divisor on $X$ that contains no
  fibers, let $\Delta$ be a normal crossing divisor (with arbitrary
  $\bQ$-coefficients) on $Y$, let $\tau : Z \to X$ be a log-resolution of
  singularities of $(X, D+ \Gamma + f^* \Delta)$ and finally let $t$ be a real number
  such that $t < \lct(\Gamma|_{X_y}; X_y, D_y)$. Then in a neighborhood of $X_y$
  \begin{equation*}
      \tau_* \sO_Z\left( \left\lceil K_{Z/X} - \tau^*  ( D + t  \Gamma  +  f^*
          \Delta ) \right\rceil \right) \simeq  \sO_X( \lceil -  f^*
      \Delta \rceil ) 
  \end{equation*}
\end{proposition}

\begin{proof}
  Since the statement is local near $y\in Y$, we may replace $Y$ with any arbitrarily
  small neighborhood of $y$, which we will do multiple times as needed without
  explicitly saying so. 

  In order to prove the desired isomorphism it is enough to find a map
  \begin{equation}
    \label{eq:lct-map}
    \xymatrix{%
      \tau_* \sO_Z\left( \left\lceil K_{Z/X} - \tau^*  ( D + t  \Gamma  +  f^* \Delta )
        \right\rceil \right)     \ar[r]^-\varsigma & \sO_X( \lceil -  f^* \Delta \rceil )
    }
  \end{equation}
  which is surjective in a neighborhood of $X_y$. Indeed, $\tau_* \sO_Z\left(
    \left\lceil K_{Z/X} - \tau^* ( D + t \Gamma + f^* \Delta ) \right\rceil \right)$
  is torsion-free of rank $1$, so if $\varsigma$ is surjective, then it is
  generically an isomorphism and hence $\ker\varsigma$ would be a torsion sheaf and
  hence zero. Therefore, it is enough to prove the existence of a map as in
  \autoref{eq:lct-map}.

  Next we will prove that such a map exists.  By
  \autoref{lem:lct_semicontinuous_divisor} and the klt assumption
  \begin{equation*}
    K_{Z} + \tau^{-1}_*( D + t  \Gamma) =
    \tau^*  (K_X+ D + t  \Gamma)  + \sum a_iE_i
  \end{equation*}
  where $a_i>-1$ and $E_i$ are pairwise distinct irreducible $\tau$-exceptional 
  divisors. (In order to get equality we choose canonical divisors on $X$ and $Z$ in
  a coherent manner). 
  Let us write $\tau^*f^*\Delta$ in a similar fashion:
  \begin{equation*}
    \tau^*f^*\Delta =     \tau^{-1}_*f^*\Delta + \sum b_iE_i
  \end{equation*}
  with some appropriate $b_i\in\bQ$. Putting these together we get
  \begin{equation*}
    K_{Z/X} - \tau^*  ( D + t  \Gamma  +  f^*\Delta)
    = - \tau^{-1}_*( D + t  \Gamma   +  f^*\Delta) + \sum (a_i-b_i)E_i,
  \end{equation*}
  which, using the facts that $D$ and $\Gamma$ has no $f$-vertical components and
  that $\lfloor D + t\Gamma\rfloor=0$ by the klt property, yields that
  \begin{equation}
    \label{eq:klt-roundup}
    \begin{aligned}
      \left\lceil K_{Z/X} - \tau^* ( D + t \Gamma + f^*\Delta) \right\rceil =
      \makebox[.5\columnwidth]{} \\ = \left\lceil - \tau^{-1}_*( D + t \Gamma +
        f^*\Delta) + \sum (a_i-b_i)E_i \right\rceil = \makebox[.15\columnwidth]{} \\
      = \left\lceil -f^*\Delta \right\rceil + \sum \left\lceil a_i-b_i\right\rceil
      E_i .
    \end{aligned}
  \end{equation}

  Since the $E_i$ are $\tau$-exceptional, after pushing forward via $\tau$, the
  components with non-negative coefficient $\left\lceil a_i-b_i\right\rceil$
  disappear and hence we obtain a map in a neighborhood of $X_y$ as requested in
  \autoref{eq:lct-map}:
  \begin{equation*}
    \xymatrix{%
      \tau_* \sO_Z\left( \left\lceil K_{Z/X} - \tau^*  ( D + t  \Gamma  +  f^* \Delta
          )         \right\rceil \right)     \ar[r]^-\varsigma & \sO_X( \lceil -  f^*
      \Delta      \rceil ).  
    }
  \end{equation*}

  In the rest of the proof we will show that this map is surjective in a neighborhood
  of $X_y$.
 
  Notice that the integral part of $\Delta$ makes no difference by the projection
  formula, so we may replace $\Delta$ with $\{ \Delta\}$, that is, we may assume that
  $\lfloor \Delta \rfloor=0$.  Furthermore, by the klt assumption on $(X_y,D_y)$ it
  follows that $X_y$ is reduced and hence we may assume that the pre-image of any
  component of $\Delta$ is reduced, which implies that any estimate or rounding of
  the coefficients of $\Delta$ remain true for $f^*\Delta$.  In particular, we may
  assume that $ \lceil - f^* \Delta \rceil=0$

  We will use induction on the number of components of $\Delta$.  If $\Delta=0$, then
  the statement follows from \autoref{eq:klt-roundup}, since in this case
  $b_i=0$. Next, let $E$ be an arbitrary component of $\Delta$. Define $H:=f^* E$,
  and let $\widetilde{H}$ be the strict transform of $H$ in $Z$. Note that
  $\tau|_{\wt H}:\wt H\to H$ is a log resolution of $(H, D|_H+ \Gamma|_H +
  f^*((\Delta-(\coeff_E \Delta) E)|_E))$.

  Observe, that in order to prove that $\varsigma$ of \autoref{eq:lct-map} is
  surjective, using Nakayama's lemma, it is enough to prove that it is surjective
  after composing with the natural surjective map $\sO_X\to\sO_H$. We will denote
  this composition by $\delta$.

  Consider the following commutative diagram. After the diagram we explain why the
  indicated maps exists and why $\alpha$ and $\gamma$ are
  surjective. 
  \begin{equation}\label{eq:the-diagram}\hskip-3em
    \begin{aligned}
      \xymatrix{%
        q\tau_* \sO_Z \left( \left\lceil K_{Z/X} - \tau^* (D + t \Gamma + H + f^*
            (\Delta - (\coeff_E \Delta) E) ) + \widetilde{H} \right\rceil \right)
        \ar@{->>}[dd]_\alpha \ar@{}@<3em>[d]^(.2){}="a1" &
        \ar@{}@<-11.5em>[d]^(.975){}="a2" \ar
        ^(.5)\beta "a1";"a2" \\ &
        **[l]\tau_* \sO_Z \left( \left\lceil K_{Z/X} - \tau^* (D + t \Gamma + f^*
            \Delta ) \right\rceil \right) \hskip-4em \ar@<-3.25em>[dd]^\delta
        \\
        \tau_* \sO_{\widetilde{H}} \left( \left\lceil K_{\widetilde{H}/H} - \tau^*
            \left( D|_H + t \Gamma|_H + f^* ((\Delta - (\coeff_E \Delta) E)|_E)
            \right) \right\rceil \right)
        \ar@{}@<10.75em>[d]^(.2){}="b1" & \ar@{}@<-3.75em>[d]^(.975){}="b2"
        \ar@{->>}_-\gamma "b1";"b2" \\ &
        \hskip-6em \sO_H }
    \end{aligned}
  \end{equation}
  Using adjunction on $X$ and $Z$ respectively we have that $K_H=(K_X+H)|_H$ and
  $K_{\wt H}=(K_Z+\wt H)|_{\wt H}$ and hence $K_{\wt H/H}=(K_{Z/X}+\wt
  H-\tau^*H)|_{\wt H}$, so $\alpha$ is simply the $\tau_*$ of the restriction map
  from $Z$ to $\wt H$.  By \cite[Theorem
  9.4.15]{Lazarsfeld_Positivity_in_algebraic_geometry_II},
  \begin{equation*}
    R^1 \tau_* \sO_Z \left( \left\lceil K_{Z/X}  - \tau^* ( D +  t \Gamma   +  H  +
        f^* (\Delta - (\coeff_E \Delta) E)) \right\rceil \right) = 0,
  \end{equation*}
  which implies that $\alpha$ is surjective.
  The map $\gamma$ is the equivalent of $\varsigma$ for $H$ and hence it is
  surjective by the inductive hypothesis. 
  To construct  $\beta$ we will show that 
  \begin{equation}
    \label{eq:two_divisors_ceiled}
    \left\lceil K_{Z/X} - \tau^* (D  + t \Gamma  + H + f^* (\Delta - (\coeff_E
      \Delta) E) )      + \widetilde{H}      \right\rceil    \leq \left \lceil
      K_{Z/X} - \tau^* (D +   t    \Gamma + f^* \Delta) \vphantom{\widetilde{H}
      }\right\rceil.\hspace{-.75em} 
  \end{equation}
  This also proves that $\beta$ is injective, but we do not need that fact. 

  To show \autoref{eq:two_divisors_ceiled} first note that the coefficient of
  $\widetilde{H}$ in the divisor on the left is zero (even before the round-up) and
  it is zero after the round-up on the right side, because of our assumption that
  $\lfloor \Delta \rfloor=0$.  To compare the coefficients of the other prime
  divisors, note that the difference in \autoref{eq:two_divisors_ceiled} between the
  divisor on the left and on the right side (before the round-up) is
  \begin{equation}
    \label{eq:F}
    \begin{aligned}
      F:=\widetilde{H} - \tau^* H + \tau^* f^* ( (\coeff_E \Delta) E) =
      \makebox[.4\columnwidth]{}  
      \\ =
      \widetilde{H} - \tau^* (H - f^* ( (\coeff_E \Delta) E) ) = 
      \makebox[.2\columnwidth]{}  \\ = \widetilde{H} -
      \tau^* (1 - \coeff_E \Delta) H .
    \end{aligned}
  \end{equation}
  Since $0 \leq \coeff_E \Delta < 1$ by our initial simplification, \autoref{eq:F}
  implies that $\coeff_G F \leq 0$ for every prime divisor $\wt{H} \neq G \subseteq
  \Supp \tau^* H$. This shows that \autoref{eq:two_divisors_ceiled} is satisfied over
  each divisor, except possibly over $\wt{H}$. To see what happens over $\wt{H}$ let
  us compute the coefficients over it on the two sides of
  \autoref{eq:two_divisors_ceiled} (before the round-up): On the left hand side the
  terms containing $H$ are $ \wt{H} - \tau^* (H + f^*(\Delta-(\coeff_E \Delta) E))$,
  however the coefficients of these over $\wt{H}$ cancel out. On the right hand side
  the coefficient of $\wt{H}$ is $- \coeff_{\wt{H}} \tau^* f^* \Delta = - \coeff_E
  \Delta$, which is at most $0$ but greater than $-1$. Hence, after rounding up the
  coefficients of $\wt{H}$ on both sides end up being $0$. This proves the existence
  (and injectivity) of $\beta$ and then \autoref{eq:the-diagram} shows that $\delta$
  is surjective and we had already observed that this implies the surjectivity of
  $\varsigma$ by Nakayama's lemma completing the proof.
\end{proof}

\begin{proposition}
  Let $(V, D_V)$ and $(Y, D_Y)$ be two klt pairs and $\sL$ and $\sN$ two line bundles
  on $V$ and $Y$ respectively. Then
  \begin{equation*}
    \lct \left(p_V^* \sL \otimes p_Y^* \sN ; V \times Y, p_V^* D_V  + p_Y^* D_Y \right) =
    \min \{ \lct(\sL;V, D_V), \lct(\sN;Y, D_Y) \} 
  \end{equation*}
\end{proposition}

\begin{proof}
  It is obvious that 
  \begin{equation*}
    \lct(p_V^* \sL \otimes p_Y^* \sN ; V \times Y, p_V^* D_V +
    p_Y^* D_Y) \leq \min \{ \lct(\sL;V, D_V), \lct(\sN;Y, D_Y) \}.
  \end{equation*}
  We have to prove the opposite inequality.  To do that, choose $\Gamma \in |p_V^*
  \sL \otimes p_Y^* \sN |$, $x \in V \times Y$ and $t< \min \{ \lct(\sL;V, D_V),
  \lct(\sN;Y, D_Y) \}$. Let $\rho: \widetilde{Y}, \to (Y,D_Y)$ be a log-resolution,
  $D_{\widetilde{Y}}:= \tau^* D_Y$, $\pi : X:= V \times \widetilde{Y} \to V \times Y$
  the natural morphism, $\widetilde{f} : X \to \widetilde{Y}$ the projection and
  $\tau : Z \to X$ a log resolution of $\left(X, \pi^* \Gamma+ p_V^* D_V +
    \widetilde{f}^* D_{\widetilde{Y}} \right)$. Note that, according to \cite[Claim
  5.20]{Viehweg_Quasi_projective_moduli}, $\widetilde{Y}$ can be chosen such that
  $\pi^* \Gamma= \Gamma' + \widetilde{f}^* \Delta$ where $\Delta$ is simple normal
  crossing on $\widetilde{Y}$ and $\Gamma'$ contains no fibers.

  According to \autoref{prop:tricky}, there is an isomorphism
  \begin{multline*}
    \pi_* \tau_* \sO_Z\left( \left\lceil K_{Z/V \times Y} - \tau^* \pi^* (p_V^* D_V
        + p_Y^* D_Y+ t  \Gamma   )\right\rceil \right)  \\
    \simeq \pi_* \tau_* \sO_Z\left( \left\lceil K_{Z/X} - \tau^* \left( p_V^* D_V + t
        \Gamma' - \widetilde{f}^* \left(K_{\widetilde{Y}/Y} + D_{\widetilde{Y}} - t
        \Delta \right) \right)\right\rceil \right)
    \\ \simeq
    \pi_* \sO_X \left( \left\lceil \widetilde{f}^* (K_{\widetilde{Y}/Y} +
        D_{\widetilde{Y}} - t \Delta) \right\rceil \right) .
  \end{multline*}
  Note that by the choice of $t$, $\left\lceil \widetilde{f}^*
    (K_{\widetilde{Y}/Y} + D_{\widetilde{Y}} - t \Delta) \right\rceil \geq 0$ and
  then
  \begin{equation*}
    \pi_* \sO_X \left( \left\lceil  \widetilde{f}^* (K_{\widetilde{Y}/Y}  +
        D_{\widetilde{Y}}- t \Delta) \right\rceil \right)  \simeq \sO_{V \times Y} . 
  \end{equation*}
  This finishes the proof. 
\end{proof}

For the next statement recall \autoref{notation:product}.

\begin{corollary}
  \label{cor:lct_product}
  If $(X,D)$ is a projective klt pair, $\sL$ a line bundle on $X$, then for all
  integers $m>0$,
  \begin{equation*}
    \lct \left(  \sL^{(m)};X^{(m)}, D_{X^{(m)}}
    \right)  = \lct( L; X,D).
  \end{equation*}
\end{corollary}

In the next statement multiplier ideals are used. Recall that the \emph{multiplier
  ideal} of a pair $(X, D)$ of a normal variety and an effective
$\bQ$-divisor is $\sJ(X,D):= \tau_* \sO_Z( \lceil K_{X/X} - \tau^* D \rceil )
\subseteq \sO_X$.

\begin{proposition}
  \label{prop:nef}
  Let $f : X \to Y$ be a surjective morphism between projective, normal varieties
  with equidimensional, reduced $S_2$ fibers, $L$ a Cartier divisor and $\Delta \geq
  0$ an effective divisor on $X$ such that $\Delta$ containins no general fibers,
  $(X_y, \Delta_y)$ is klt for general $y \in Y$ and $L - K_{X/Y} - \Delta$ is a nef
  and f-ample $\bQ$-Cartier divisor. Assume further that $K_Y$ is Cartier. Then $f_*
  \sO_X(L)$ is weakly-positive (in the weak sense).
\end{proposition}

\begin{proof}
  Set $\sL:= \sO_X(L)$.  Let $A$ be a general very ample effective divisor on $Y$ and
  $m>0$ an integer.  In this proof a subscript of $A$ will denote a base change to
  $A$.  
  \begin{subclaim}
    \label{claim:H0-is-surjective}
    For any nef Cartier divisor $N$ on $Y$ the natural restriction map,
    
    \begin{multline*}
      H^0 \left( X^{(m)}, \sJ \left(X^{(m)}, \Delta_{X^{(m)}} \right) \otimes
        \sL^{(m)} \left( \left(f^{(m)} \right)^* ( K_Y + 2 A + N ) \right) \right)
      \longrightarrow %
      \\ \longrightarrow H^0 \left( X^{(m)}_A, \sJ \left(X^{(m)}_A,
          \Delta_{X^{(m)}_A} \right) \otimes \left.\sL^{(m)} \left( \left(f^{(m)}
            \right)^* ( K_Y + 2 A + N )\right)\right|_{X^{(m)}_A} \right)
    \end{multline*}
    is surjective.
  \end{subclaim}
  \begin{proof}
    Note that in the statement we are already using the fact that $\sJ
    \left(X^{(m)}_A, \Delta_{X^{(m)}_A} \right) \simeq \sO_{X^{(m)}_A} \otimes \sJ
    \left(X^{(m)}, \Delta_{X^{(m)}} \right)$, which follows from the general choice
    of $A$.  For the above homomorphism to be surjective, it is enough to prove that
  \begin{equation}
    \label{eq:H1-vanish}
    H^1  \left( X^{(m)}, \sJ \left(X^{(m)}, \Delta_{X^{(m)}} \right) \otimes
      \sL^{(m)} \left( \left(f^{(m)} \right)^* 
        ( K_Y + A + N )\right) \right) =0
  \end{equation}
  However, 
  \begin{equation*}
      L_{X^{(m)}}  + \left( f^{(m)} \right)^*    ( K_Y + A + N ) 
    - \left(K_{X^{(m)}} + \Delta_{X^{(m)}}\right) 
    =  \underbrace{L - K_{X/Y} - \Delta}_{\txt{\tiny \txt{nef and\\ relatively ample}}} +
    \left( f^{(m)} \right)^*( \underbrace{A +N}_{\txt{\tiny ample}}) 
  \end{equation*}
  is ample, hence \autoref{eq:H1-vanish} holds by Nadel-vanishing. This proves the
  claim.
  \end{proof}

  Note that the assumptions of the proposition remain valid for $f|_{X_A} : X_A \to
  A$ and $\Delta|_{X_A}$. Hence we may use \autoref{claim:H0-is-surjective}
  iteratively. By the klt assumption on the general fiber, we may further leave out
  the multiplier ideal in the last term. Thus we obtain a surjective homomorphism
  \begin{equation*}
    H^0 \left( X^{(m)}, \sJ \left(X^{(m)}, \Delta_{X^{(m)}} \right) \otimes 
      \sL^{(m)} \left( \left(f^{(m)} \right)^{*}  \left( K_Y + A + \sum_{i=1}^n A_i
        \right)\right) \right)   
    \longrightarrow
    %
    H^0 \left( X^{(m)}_y, \sL^{(m)}_y \right)
  \end{equation*}
  where $A_1, \dots, A_n \in |A|$ are general, $ y \in \bigcap_{i=1}^n A_i$ is
  arbitrary and $n:=\dim Y$. Since the left hand side of this homomorphism is a
  subspace of
  \begin{equation*}
    H^0\left( Y, \sO_Y\left( K_Y + A + \sum_{i=1}^n A_i  \right) \otimes f^{(m)}_*
      \sL^{(m)} \right)  
  \end{equation*}
  and the right hand side can be identified with $f^{(m)}_*\sL^{(m)}\otimes k(y)$
  (recall $y \in Y$ is general), we obtain that
  \begin{equation*}
    H^0\left( Y, \sO_Y\left( K_Y + A + \sum_{i=1}^n A_i  \right) \otimes f^{(m)}_*
      \sL^{(m)} \right)  %
    \to  f^{(m)}_* \sL^{(m)} \otimes k(y) 
  \end{equation*}
  is surjective. Therefore, 
  \vskip-2.5em
  \begin{equation*}
    \xymatrix@C=-1.75em@R=.5em{%
    & \ar[d]\textrm{\tiny \autoref{lem:pushforward_tensor_product_isomorphism}}\\
    \sO_Y\left( K_Y + A + \sum_{i=1}^n A_i  \right) \otimes f^{(m)}_* \sL^{(m)}
    & {\simeq}  &
    \sO_Y\left( K_Y + A + \sum_{i=1}^n A_i  \right) \otimes
    \left[ \bigotimes_{j=1}^m \right] f_*
    \sL \\
    }
  \end{equation*}
  is generically globally generated for all $m>0$. However, then so is $\sO_Y\left(
    K_Y + A + \sum_{i=1}^n A_i \right) \otimes \Sym^{[a]} (f_* \sL)$, since there is
  a generically surjective homomorphism from the former to the latter.  This yields
  weak positivity (in the weak sense).
\end{proof}

\begin{proof}[Proof of \autoref{thm:pushforward}]
  By resolving $Y$ and then pulling back $X$ to the resolution we may assume that $Y$
  is smooth.  According to \autoref{thm:big_higher_dim_base}, for all divisible
  enough $q>0$, $\det f_* \sO_X(q(K_{X/Y}+D))$ is big. Fix such a $q$.  According to
  \autoref{lem:lct_semicontinuous_line_bundle} there is a real number $c>0$ such that
  \begin{equation*}
    c< \lct \left(\sO_{X_y} \left( q \left(K_{X_y}+D_y \right) \right); X_y, D_y \right) 
  \end{equation*}
  for every $y \in U$, where $U$ is the open locus over which the fibers $(X_y, D_y)$
  are klt. Fix also such a $c$ and let $l:= \left\lceil \frac{1}{c} \right\rceil$.
  By replacing $Y$ with a finite cover, we may assume that $\det f_*
  \sO_X(q(K_{X/Y}+D))=\sO_Y( l A)$ for some Cartier divisor $A$. Define $m:=\rk f_*
  \sO_X(q(K_{X/Y}+D)) $ and consider the natural homomorphism, \vskip-2em
  \begin{equation}
    \label{eq:pushforward:embedding}
    \begin{aligned}
      \xymatrix@R=-1em{%
        && \\
        \sO_Y(lA)=\det f_* \sO_X(q(K_{X/Y}+D)) 
      \ar@{}[r]^(0.75){}="b1"
      \ar@{}[r]^(1.1){}="b2"
      \ar@{^(->} "b1" ; "b2" 
      & & \protect{
        \phantom{\overbrace{a}^{\textrm{\autoref{lem:pushforward_tensor_product_isomorphism}}}}}
      \\ 
      \ar@{}[r]^(-0.3){}="a1"
      \ar@{}[r]^(0.1){}="a2"
      \ar@{^(->} "a1" ; "a2" 
      & & **[l]\bigotimes_{i=1}^m f_* \sO_X(q(K_{X/Y} + D)) \simeq
      \underbrace{f^{(m)}_* \sO_X \left( q \left( K_{X^{(m)}/Y} + D_{X^{(m)}} \right)
        \right)}_{\textrm{\autoref{lem:pushforward_tensor_product_isomorphism}}} }\\
    \end{aligned}
  \end{equation}
  \vskip-1em\noindent which implies that
  \begin{equation}
    \label{eq:Gamma}
    \left(f^{(m)}\right)^*     lA  + \Gamma \sim  q \left( K_{X^{(m)}/Y} +
      D_{X^{(m)}} \right) 
  \end{equation}
  for some appropriate effective divisor $\Gamma$ on $X^{(m)}$. Note that since
  \autoref{eq:pushforward:embedding} has a local splitting, $\Gamma_y \neq 0$ for any
  $y \in Y$. In particular, $\Gamma$ does not contain any $X^{(m)}_y$ for any $y \in
  U$, since fibers over $U$ are irreducible.

  By \autoref{eq:Gamma} we obtain that
  \begin{equation}
    \label{eq:Q_linearly_equiv}
    \frac{1}{l}\Gamma + \frac{q(2l-1)}{l}\left( K_{X^{(m)}/Y} + D_{X^{(m)}} \right)
    \sim_{\bQ}  2q\left( K_{X^{(m)}/Y} + D_{X^{(m)}} \right) - \left(f^{(m)}\right)^*  A . 
  \end{equation}
  Note that for each $y \in U$,
  \begin{multline*}
    \lct \left( \frac{1}{l}\Gamma_y; X^{(m)}_y, \left(D_{X^{(m)}}\right)_y
    \right) \leq \\ \leq 
    l \cdot \lct \left( \sO_{X_y^{(m)}}\left(q\left( K_{X_y^{(m)}} +
          \left(D_{X^{(m)}}\right)_y \right) \right); X^{(m)}_y,
      \left(D_{X^{(m)}}\right)_y \right)
    = \\ =
    \underbrace{l \cdot \lct \left( \sO_{X_y} \left( q \left(K_{X_y} + D_y \right) 
        \right) ; X_y, D_y \right) }_{\textrm{\autoref{cor:lct_product}}}
    > \left\lceil \frac{1}{c} \right\rceil c \geq 1
  \end{multline*}
  Therefore, $\left(X^{(m)}_y, \frac{1}{l}\Gamma_y + \left(D_{X^{(m)}}\right)_y
  \right)$ is klt for all $y \in U$. Then by \autoref{eq:Q_linearly_equiv} and
  \autoref{lem:pushforward_nef} we may apply \autoref{prop:nef} to show that
  \begin{multline*}
    f^{(m)}_* \sO_{X^{(m)}} \left( 2q\left(K_{X^{(m)}/Y} + D_{X^{(m)}} \right) -
      \left(f^{(m)}\right)^*      A  \right)
    \simeq \\ \simeq f^{(m)}_* \sO_{X^{(m)}} \left( 2q \left(K_{X^{(m)}/Y} +
        D_{X^{(m)}} \right) \right) \otimes \sO_Y(- A )
    \simeq \\ \simeq \underbrace{\sO_Y(- A ) \otimes \bigotimes_{i=1}^m f_* \sO_{X}
      \left( 2q \left(K_{X/Y} + D \right)
      \right)}_{\textrm{\autoref{lem:pushforward_tensor_product_isomorphism}}}
  \end{multline*}
  is weakly-positive. Therefore there exists an integer $b>0$ such that
  \begin{multline*}
    \sO_Y(bA) \otimes \Sym^{2b} \left( \sO_Y(- A ) \otimes \bigotimes_{i=1}^m f_*
      \sO_{X} \left( 2q \left(K_{X/Y} + D \right) \right) \right)
    \simeq \\ \simeq
    \sO_Y(-bA) \otimes \Sym^{2b} \left( \bigotimes_{i=1}^m f_* \sO_{X} \left( 2q
        \left(K_{X/Y} + D \right) \right) \right)
    \twoheadrightarrow \\ \twoheadrightarrow \sO_Y(-bA) \otimes \Sym^{2bm} \left( \
      f_* \sO_{X} \left( 2q \left(K_{X/Y} + D \right) \right) \right)
  \end{multline*}
  is generically globally generated. Hence $f_* \sO_{X} \left( 2q \left(K_{X/Y} +
      D\right) \right)$ is big.
\end{proof}


\section{Subadditivity of log-Kodaira dimension}
\label{sec:subadditivity}

\newcommand{\lfs}{log canonical fiber space\xspace}
\newcommand{\lfss}{log canonical fiber spaces\xspace}

In this section we are considering the question of subadditivity of log-Kodaira
dimension. Since, at this point, there are multiple non-equivalent statements of this
conjecture in the literature, we state a couple of them. All of these are
straightforward consequences of \autoref{prop:subadditivity}.

\begin{definition}
  \label{def:fiber_space}
  A \emph{\lfs} is a surjective morphism $f : (X, D) \to Y $ such that 
  \begin{enumerate}
  \item both $X$ and $Y$ are irreducible, normal and projective,
  \item $K_X +D $ is $\bQ$-Cartier and
  \item $(X_\eta, D_\eta)$ has log canonical singularities.
  \end{enumerate}
\end{definition}



Next we define the notion of variation for \lfss. 
Unfortunately, at this time we have to put a restriction on the \lfss on which the
definition works.  The main issue is that in \autoref{def:variation}, variation is
defined only for families of stable log-varieties. For general \lfss as in
\autoref{def:fiber_space} the reasonable expectation is that we would define
variation as the variation of the relative log canonical model of $(X, D)$
(restricted to the open locus where it is a stable family). However, for log
canonical singularities, the existence of a log canonical model is not known even in
the log-general type case. Hence, in order to make this definition, we assume that a
relative log canonical model exists. This is known for example if the general fiber
is klt.

\begin{definition}
  \label{def:variation_fiber_space}
  Let $f : (X, D) \to Y $ be a \lfs such that 
  $K_{X_\eta} + D_{\eta}$ is big and $(X_\eta,D_\eta)$ admits a log canonical model,
  where $\eta$ is the generic point of $Y$. Then let $\var f$ to be the variation of
  the log canonical model of $(X_\eta, D_\eta)$ as defined in
  \autoref{def:variation}.
\end{definition}

\begin{remark}
  If $(X_\eta,D_\eta)$ is klt and $K_{X_\eta} + D_{\eta}$ is big, then
  $(X_\eta,D_\eta)$ admits a log canonical model by \cite[Thm
  1.2]{Birkar_Cascini_Hacon_McKernan_Existence_of_minimal_models}) and hence in this
  case $\var f$ is defined. 
\end{remark}


\begin{theorem}
  \label{thm:subadditivity}
  If $f : (X, D) \to Y $ is a \lfs with $K_{X_\eta} + D_{\eta}$ big, where $\eta$ is
  the generic point of $Y$, then subadditivity of log-Kodaira dimension holds. That
  is,
  \begin{equation*}
    \kappa(K_X + D) \geq  \kappa(Y)  +
    \kappa(K_{X_\eta} + D_{\eta} ). 
  \end{equation*}
  Furthermore, if $(X_\eta,D_\eta)$ is klt, then
  \begin{equation*}
    \kappa(K_X + D) \geq  \max\{\kappa(Y), \var f \}  +
    \kappa(K_{X_\eta} + D_{\eta} ). 
  \end{equation*}
\end{theorem}

\begin{theorem}
  \label{thm:to_subadditivity_for_q_proj_varieties}
  If $f : (X, D) \to (Y,E) $ is a surjective map of log-smooth and log canonical
  projective pairs, such that $D \geq f^*E$ and $K_{X_\eta} + D_{\eta}$ is big, where
  $\eta$ is the generic point of $Y$, then
  \begin{equation*}
    \kappa(K_X + D) \geq \kappa\left( K_{X_{\eta}} + D_{\eta} \right) + \kappa (K_Y + E).
  \end{equation*}
\end{theorem}

In the next corollary we use the notion of Kodaira dimension of an arbitrary
algebraic variety $X$. It is defined via finding a resolution $X_0'$ of $X$ with a
projective compactification $X'$ such that $D':=(X' \setminus X_0')_{\red}$ is simple
normal crossing, and then setting
\begin{equation*}
  \kappa(X) := \kappa(K_{X'} + D').
\end{equation*}
The following statement follows immediately from
\autoref{thm:to_subadditivity_for_q_proj_varieties}.

\begin{corollary}
\label{cor:subadditivty_q_proj}
  Let $f : X \to Y $ be a dominant map of (not necessarily proper) algebraic
  varieties such that the generic fiber has maximal Kodaira dimension. Then
  \begin{equation*}
    \kappa(X) \geq \kappa\left(X_\eta \right)  + \kappa(Y).
  \end{equation*}
\end{corollary}

\begin{proposition}
  \label{prop:subadditivity}
  Let $f : (X, D) \to Y $ be a \lfs such that
  \begin{enumerate}
  \item $K_{X_\eta} + D_{\eta}$ is big, where $\eta$ is the generic point of $Y$,
  \item $(X,D)$ is log-smooth, and
  \item $Y$ is smooth.
  \end{enumerate}
  Further let $M$ be a $\bQ$-Cartier divisor on $Y$ with $\kappa(M) \geq 0$. Then
  \begin{equation*}
    \kappa(K_{X/Y} + D + f^*M) \geq  \kappa(M)  +
    \kappa(K_{X_\eta} + D_{\eta} ). 
  \end{equation*}
  Furthermore, if $(X_\eta,D_\eta)$ is klt, then
  \begin{equation*}
    \kappa(K_{X/Y} + D + f^*M) \geq  \max\{\kappa(M), \var f \}  +
    \kappa(K_{X_\eta} + D_{\eta} ). 
  \end{equation*}
\end{proposition}

\begin{proof}[Proof of \autoref{thm:subadditivity} using
  \autoref{prop:subadditivity}. ]
  Let $\tau: Y' \to Y$ be a resolution of $Y$, and let $X'$ be a resolution of the
  component of $X \times_Y Y'$ that dominates $X$ such that $\pi^* D$ is a simple
  normal crossing divisor, where $\pi : X' \to X$, is the induced map.  Choose
  canonical divisors $K_K$ and $K_{X'}$ such that they agree on the locus where $\pi$
  is an isomorphism. Then define $D'$ and $E$ via
  \begin{equation*}
    K_{X'} + D' = \pi^* (K_X + D ) + E
  \end{equation*}
  such that $E, D' \geq 0$ and have no common components. Then we have $\kappa(K_X +
  D) = \kappa(K_{X'} + D')$, $\kappa\left(K_{X_\eta} + D_{\eta} \right) = \kappa
  \left(K_{X_\eta'} + D_\eta' \right)$ and if $(X_\eta , D_\eta)$ is klt then also
  $\var f = \var f'$, where $f' : (X', D') \to Y'$ is the induced
  morphism. \autoref{prop:subadditivity} applied to $f'$ with $M= K_{Y'}$ completes
  the proof.
\end{proof}

\begin{proof}[Proof of \autoref{thm:to_subadditivity_for_q_proj_varieties} using
  \autoref{prop:subadditivity}. ]
  Let $D':= D- f^*E$, set $M:= K_Y +E$, and apply \autoref{prop:subadditivity} to $f
  : (X, D') \to Y$ and $M$. Notice that we may assume that $\kappa(K_Y+E)\geq 0$,
  since otherwise the statement is trivial.  This yields
  \begin{equation*}\hskip-.1em
    \kappa(K_X +D)=      \kappa(K_{X/Y} + D' + f^*(K_Y + E)) \geq  \kappa(K_Y + E)  + 
    \kappa(K_{X_\eta} + D_{\eta}' ) = \kappa(K_Y + E)  + \kappa(K_{X_\eta} +
    D_{\eta}).\hskip-.7em\qedhere
  \end{equation*}
\end{proof}

The rest of the section concerns proving \autoref{prop:subadditivity}.

\begin{lemma}
  \label{lem:dualizing_embedding}
  Consider the following commutative diagram of normal varieties, where $f$ is flat
  and Gorenstein, $\tau$ is surjective, $\oX:= X \times_Y Y'$ and $X_n$ is the
  normalization of the component of $X \times_Y Y'$ dominating $Y'$.
\begin{equation*}
  \xymatrix@C=6em{
    X \ar[d]_f & \oX \ar[d]_{\of} \ar[l]_{\alpha} & X_n \ar[ld]^{f_n} \ar[l]_{\beta}
    \\ 
    Y & Y' \ar[l]^{\tau}
  }
\end{equation*}
Then, there is a natural embedding $\omega_{X_n/Y'} \hookrightarrow \beta^* \alpha^*
\omega_{X/Y}$.
\end{lemma}

\begin{proof}
  Since $f$ is flat and Gorenstein, $\omega_{\oX/Y'} \simeq \alpha^* \omega_{X/Y}$
  according to \cite[Thm 3.6.1]{Conrad_Grothendieck_duality_and_base_change}. In
  particular, $\omega_{\oX/Y'}$ is a line bundle. Consider then the Gorthendieck
  trace of $\beta$, $\beta_* \omega_{X_n/Y'} \to \omega_{\oX/Y'}$. Pulling this back
  we obtain $\phi : \beta^* \beta_* \omega_{X_n/Y'} \to \beta^*
  \omega_{\oX/Y'}$. \emph{We claim that $\phi$ factors through the natural map $\xi :
    \beta^* \beta_* \omega_{X_n/Y'} \to \omega_{X_n/Y'}$.} For this note first that
  since $\beta^* \omega_{\oX/Y'}$ is a line bundle, it is torsion-free. Hence if
  $\sT$ is the torsion part of $\beta^* \beta_* \omega_{X_n/Y'}$, $\phi$ factors
  through the natural map $\beta^* \beta_* \omega_{X_n/Y'} \to \beta^* \beta_*
  \omega_{X_n/Y'}/\sT$. Therefore, it is enough to show that the latter map is
  isomorphic to $\xi$, that is, that $\ker \xi = \sT$ and that $\xi$ is
  surjective. The surjectivity follows immediately, since $\beta$ is affine and for
  any ring map $A \to B$ and $B$-module $M$, the natural map $M \otimes_A B \to M$ is
  surjective. To show that $\ker \xi \subseteq \sT$ we just note that $\beta$ is
  generically an isomorphism, and hence $\xi$ is generically an isomorphism. The
  opposite containment, that is, that $\ker \xi \supseteq \sT$, follows from the fact
  that $\omega_{X_n/Y'}$ is torsion-free. This concludes our claim. Hence we obtain
  an embedding $\omega_{X_n/Y'} \hookrightarrow \beta^* \omega_{\oX/Y'} \simeq \beta^*
  \alpha^* \omega_{X/Y} $.
\end{proof}

\addtocounter{theorem}{-1}
\addtocounter{equation}{3}
\begin{proof}[Proof of \autoref{prop:subadditivity}]\ 

  \smallskip
  \noindent
  \emph{{\sc Step 0:} Assuming klt.}  If $(X_\eta,D_\eta)$ is not klt, then by
  decreasing the coefficients of $D$ a little all our assumptions remain true, and so
  we may assume that $(X_\eta, D_\eta)$ is klt.

  \smallskip
  \noindent
  \emph{{\sc Step 1:} Allowing an extra divisor avoiding a big open set of the base.}
  According to \cite[Lemma 7.3]{Viehweg_Weak_positivity}, there is a birational
  morphism $\wt Y \to Y$ from a smooth projective variety, and another one from $\wt
  X$ onto the component of $X \times_Y \wt Y$ dominating $\wt Y$, such that for the
  induced map $\wt f : \wt X \to \wt Y$ and for every prime divisor $E \subseteq \wt
  X$, if $\codim_{\wt Y} \wt f(E) \geq 2$, then $E$ is $\wt X \to X$
  exceptional. Furthermore, it follows from the proof of \cite[Lemma
  7.3]{Viehweg_Weak_positivity} that we may choose $\wt X \to X$ and $\wt Y \to Y$ to
  be isomorphisms over the smooth locus of $f$ on $Y$. Let $\rho: \wt X \to X$ and
  $\tau : \wt Y \to Y$ be the induced maps and set $\wt D:= \rho^* D$ and $\wt M:=
  \tau^* M$.
  \begin{subclaim}\label{claim}
    It is enough to prove that for some divisor $0 \leq B$ on $\wt X$, for which
    $\codim_{\wt Y} \wt f(B) \geq 2$ the following holds:
    \begin{equation*}
      \kappa(K_{\wt X/\wt Y} + \wt D + \wt f^{*}\wt M +  B) \geq  \max\{\kappa(\wt
      M), \var       \wt f \}  +       \kappa(K_{\wt{X}_\eta} + \wt{D}_{\eta} ). 
    \end{equation*}
  \end{subclaim}
  \begin{proof}[Proof of \autoref{claim}]
    We have that $X_\eta= \wt{X}_\eta$, $\kappa(K_{X_\eta} + D_{\eta}
    )=\kappa(K_{\wt{X}_\eta} + \wt{D}_{\eta} )$, $\var f= \var \wt f$ and $\kappa(M)=
    \kappa(\wt M)$.  Furthermore, note that since both $Y$ and $\wt{Y}$ are smooth,
    there is an effective divisor $E$ on $\wt{Y}$ such that $K_{\wt{Y}}= \tau^* K_Y +
    E$. In particular, the following holds.
    \begin{equation}
      \label{eq:relativ_canonical_for_modification} 
      K_{\wt{X}/\wt{Y}} = K_{\wt{X}} - \wt{f}^* K_{\wt{Y}}= K_{\wt{X}} - \wt{f}^*
      (\tau^* K_Y + E)= K_{\wt{X}} - \rho^* f^* K_Y - \wt{f}^* E 
    \end{equation}
    Since $B$ is $\rho$ exceptional and $\rho$ is birational, we obtain using
    \autoref{eq:relativ_canonical_for_modification} that for every $m>0$ integer
    \begin{equation*}
      \rho_* \sO_{\wt X}(m(K_{\wt X/\wt{Y}}+B)) \hookrightarrow \rho_* \sO_{\wt
        X}(m(K_{\wt X} - \rho^* f^* K_Y+B))\simeq  \sO_{X}(m(K_{X} -  f^* K_Y))
      \simeq       \sO_X(mK_{X/Y}) 
    \end{equation*}
    Furthermore, by construction, $\wt f^{*}\wt M + \wt D=\rho^*(f^{*}M + D)$ and
    hence for every divisible enough $m>0$ integer there is an injection
    \begin{equation*}
      \rho_* \sO_{\wt X}(m(K_{\wt X/\wt Y} + \wt f^{*}\wt M + \wt D + B)) \hookrightarrow 
      \sO_X(m(K_{X/Y} + D + f^* M)).  
    \end{equation*}
    This shows that $\kappa(K_{\wt X/\wt Y} + \wt D + \wt f^{*}\wt M + B) \leq
    \kappa(K_{X/Y} + D + f^* M)$, which implies the claim.
  \end{proof}

  From now on our goal is to prove that for some $0 \leq B$ for which $\codim_Y f(B)
  \geq 2$,
  \begin{equation}
    \label{eq:subadditivity_goal_1}
    \kappa(K_{X/Y} + D + f^*M +  B) \geq  \max\{\kappa(M), \var f \}  +
    \kappa(K_{X_\eta} + D_{\eta} ). 
  \end{equation}

  \smallskip
  \noindent
  \emph{{\sc Step 2:} Disallowing vertical components of $D$.} If $D$ contains a
  vertical component, after deleting that component from $D$ our assumptions are
  still satisfied. In other words, we may assume that $D$ contains no vertical
  components.

  \smallskip
  \noindent
  \emph{{\sc Step 3:} Replacing $\var f$ by $\var f_{\can}$.}  Let $f_{\can} :
  (X_{\can}, D_{\can}) \to Y$ be the log canonical model of $(X,D)$ over some dense
  open set $U \subseteq Y$ over which $(X,D)$ is klt.  By shrinking $U$ we may assume
  that $f_{\can}$ is a stable family. Note that if $(X,D)$ was klt to start with,
  then $\var f = \var f_{\can}$ (where the latter is taken as the variation as a
  stable family of log-varieties). Hence, in order to obtain
  \autoref{eq:subadditivity_goal_1} it is enough to prove the following inequality:
  \begin{equation}
    \label{eq:subadditivity_goal_2}    
    \kappa(K_{X/Y} + D + f^*M + B) \geq
    \max\{\kappa(M), \var f_{\can} \}  + 
    \kappa(K_{X_\eta} + D_{\eta} ). 
  \end{equation}

  \smallskip
  \noindent
  \emph{{\sc Step 4:} An auxilliary base change.} 
  Set $n:= \dim X - \dim Y$, $v:= \vol \left( K_{X_\eta} + D_{\eta} \right)$, where
  $\eta$ is the generic point of $Y$.  Let $I \subseteq [0,1]$ be a finite
  coefficient set closed under addition (\autoref{def:close_under_addition}) that
  contains the coefficients of $D$.  Let $\mu : U \to \sM_{n,v,I}$ be the moduli map
  associated to $\left(V_{\can}, \left. D_{\can}\right|_{V_{\can}} \right)$, where
  $V_{\can}:=f_{\can}^{-1}(U)$ and let $S \to \sM_{n,v,I}$ be the finite cover
  granted by \autoref{cor:finite_cover}. Define then $Y^{\mathrm{aux}}$ to be the
  resolution of a compactification of a component of $U \times_{\sM_{n,v,I}} S$ that
  dominates $U$. We may further assume that the birational map $\delta :
  Y^{\mathrm{aux}} \dashrightarrow Y$ is a morphism. Let $Y''$ be the normalization
  of the image of $Y^{\mathrm{aux}}$ in $S$ and $f'' : (X'', D'') \to Y''$ the family
  over $Y''$ induced by ${\sf f}\in\sM(S)$ given in
  \autoref{cor:finite_cover}. Then the pullback of this family over $\delta^{-1}(U)$
  is isomorphic to the pullback of $(X_{\can}, D_{\can})$ and hence $\dim Y''= \var
  f_{\can}$. 
  
  \smallskip
  \noindent
  \emph{{\sc Step 5:} Local stable reduction.} That is, we construct a generically
  finite map $Y' \to Y^{\mathrm{aux}}$ and a normal pair $(X', D')$, such that
  \begin{itemize}
  \item $Y'$ is smooth
  \item $X'$ maps birationally onto the component of $X \times_Y Y'$ dominating $Y'$,
  \item $X' \to X \times_Y Y'$ is an isomorphism over the generic point of $Y'$,
  \item if $\tau : Y' \to Y$ and $\rho : X' \to X$ are the induced maps, then $D'
    \leq \rho^* D$, and
  \item $(X',D')$ is a \emph{locally stable family} at every $y \in Y'$ for which 
    $\codim_Y \tau(y)=1$, that is, at every such $y \in Y'$ the following two
    equivalent conditions hold:
    \begin{itemize}
    \item $(X', \overline{X_y'} + D')$ is lc around $X_y'$, where $\overline{X_y'}$
      is the closure of $X_y'$, or equivalently
    \item $(X_y', D_y')$ is slc and $K_{X'} + D'$ is $\bQ$-Cartier around $X_y'$.
    \end{itemize}
  \end{itemize}
  To obtain the above, we apply the process described in \cite[first 6 paragraphs in
  the proof of Thm 12.11]{Kollar_Second_moduli_book} to the fibers of $X \times_Y
  Y^{\mathrm{aux}} \to Y^{\mathrm{aux}}$ over $y \in Y^{\mathrm{aux}}$ mapping to
  codimension one points of $Y$. That is, first we resolve the main component of $X
  \times_Y Y^{\mathrm{aux}}$ to obtain $X^{\mathrm{nc}}$ such that $X^{\mathrm{nc}}$
  is smooth, and if $D^{\mathrm{nc}}$ is the horizontal part of the pullback of $D$
  to $X^{\mathrm{nc}}$, then $(X^{\mathrm{nc}},D^{\mathrm{nc}} + X^{\mathrm{nc}}_y)$
  is a normal crossing pair around $X^{\mathrm{nc}}_y$ for each $y \in
  Y^{\mathrm{aux}}$ mapping to a codimension one point of $Y$. Then a generically
  finite cover of the base with prescribed ramifications at finitely many codimension
  one points (allowing further ramifications at unprescribed points) yields the
  required $Y'$ as in \cite[proof of Thm 12.11]{Kollar_Second_moduli_book}. $X'$ is
  defined to be the normalization of the main component of $X^{\mathrm{nc}}
  \times_{Y^{\mathrm{aux}}} Y'$.

  \smallskip
  \noindent
  \emph{{\sc Step 6:} Choosing nice big open sets.} Let $f' : X' \to Y'$ be the
  natural morphism and let $Y_0 \subseteq Y$ be the big open set over which
  \begin{enumerate}
  \item $f'$ is flat over $\tau^{-1}Y_0$,
  \item $\left(X_0',D|_{X_0'}\right)$ is klt and forms a flat locally stable family of
    log-varieties, where $X_0':= f'^{-1} \tau^{-1} Y_0$.
  \end{enumerate}
  Let $Y_0':=\tau^{-1} Y_0$ and let $f_{\can}' : (X_{\can}', D_{\can}') \to Y_0'$ be
  the log canonical model of $\left(X_0',D|_{X_0'} \right)$ over $Y_0'$. By shrinking
  $Y_0$ (and $Y_0'$ and $X_0'$ compatibly, keeping $Y_0$ big), we may further assume
  that
  \begin{enumerate}[resume]
  \item $f_{\can}'$ is flat (and hence it is a family of stable log-varieties).
  \end{enumerate}
  Let $\eta'$ be the generic point of $Y'$. Then we have $\left(X_{\eta'}',
    D_{\eta'}' \right) \simeq \left(X_\eta, D_\eta \right)_{\eta'}$, since over the
  locus (in $Y$) over which $f$ is smooth and $(X,D)$ is a relative normal crossing
  divisor, $(X',D')$ is isomorphic to $(X,D) \times_Y Y'$. Therefore
  \begin{equation*}
    \left( X_{\can}', D_{\can}' \right)_{\eta'} \simeq \left( X_{\can}, D_{\can}
    \right)_{\eta'} \simeq  \left( X'', D'' \right)_{\eta'}. 
  \end{equation*}
  That is, $(X'', D'') \times_{Y''} Y_0'$ is isomorphic to $\left(X_{\can}',
    D_{\can}'\right)$ over $\eta'$. Equivalently, their $\Isom$ scheme has a rational
  point over $\eta'$. The closure of this rational point yields a section of the
  $\Isom$ scheme over a big open set of $Y_0'$. Hence, by further restricting $Y_0$,
  we may assume that
  \begin{enumerate}[resume]
  \item $(X'', D'') \times_{Y''} Y_0'$ is isomorphic to $(X_{\can}', D_{\can}')$.
  \end{enumerate}

  \smallskip
  \noindent
  \emph{{\sc Step 7:} Bounding $\kappa\left(K_{X_{\can}'/Y_0'} + D_{\can}' +
      f_{\can}'^* \tau^* M \right)$.} According to \autoref{cor:big_total_space},
  $K_{X''/Y''} + D''$ is big. In particular, there is an ample divisor $H$ and an
  effective divisor $E$ on $X''$, such that $H +E \sim q(K_{X''/Y''} + D'')$ for some
  $q>0$ divisible enough. Let $\pi : X_{\can}' \to X''$ be the induced map and let $V
  \subseteq |\pi^*H|$ be a linear system inducing $\pi$. Further let $W \subseteq |q
  f_{\can}'^* \tau^* M|$ be the linear system that identifies with $|qM|$ via the
  natural embedding $|qM| \hookrightarrow |q f_{\can}'^* \tau^* M|$.
  \begin{equation*}
    \xymatrix@C=5em{
      X_{\can}' \ar@{-->}[rrrd]_{\phi_{V + W}} \ar[d]_{\pi= \phi_V} \ar[r]^f
      \ar@{-->}@/^2pc/[rrr]^{\phi_W} & Y' \ar[r]^{\tau} & Y
      \ar@{-->}[r]_{\phi_{|qM|}} & Z \\ 
      X'' & & & X'' \times Z \ar[u] \ar[lll]
    }
  \end{equation*}
  We compute the dimension of a general fiber of $\phi_{V + W}$. For that choose an
  open set $U' \subseteq X_{\can}'$, such that $\phi_{V+W}$ is a morphism over $U'$
  and $\phi_{|qM|}$ is a morphism over $\tau \left( f_{\can}' \left( U' \right)
  \right)$. In the next few sentences, when computing fibers of $\phi_{V + W}$,
  $\phi_W$ and $\phi_{|qM|}$, we take $U'$ and $\tau_* \left(f_{\can}'\right)_* U'$
  as the domain. So, choose $z \in Z$ and $x \in X''$ general. We have $\phi_{V +
    W}^{-1}((x,z))= \phi_W^{-1}(z) \cap \phi_V^{-1}(x)$. Furthermore,
  $\phi_W^{-1}(z)$ is of the form $f_{\can}'^{-1}(Z')$ for a variety $Z'$ of
  dimension $\dim Y - \kappa(M)$. On the other hand, $\phi_V^{-1}(x)$ intersects each
  fiber of $f_{\can}'$ in at most one point and has dimension $\dim Y - \var
  f_{\can}$. Therefore,
  \begin{equation}
    \label{eq:general_fiber}
    \dim \phi_{V + W}^{-1}((x,z)) \leq \min \{\dim Y - \var f_{\can}, \dim Y -
    \kappa(M) \}. 
  \end{equation}
  Hence,     \vskip-2em
  \begin{multline}
    \label{eq:Kodaira_dimension_estimate_1}
    \kappa \left( K_{X_{\can}'/Y_0'} + D_{\can}' + f_{\can}'^* \tau^* M \right) \geq
    \hskip-1.75em \overbrace{\kappa \left( \pi^* H + q f_{\can}'^* \tau^* M
      \right)}^{\txt{\scriptsize $\pi^* H + \pi^* E \sim
        q\left(K_{X_{\can}'/Y_0'}+D_{\can}' \right),$ \\ \tiny { and } $E \geq 0$}}
    \hskip-1.75em
    \geq \dim \im \phi_{V+W} \geq 
    \\ \geq \underbrace{ n + \dim Y - \min \{\dim Y - \var f_{\can}, \dim Y -
      \kappa(M) \}}_{\autoref{eq:general_fiber}}
    = n + \max \{ \var f_{\can}, \kappa(M) \}.
\end{multline}

  \smallskip
  \noindent
  \emph{{\sc Step 8:} Conclusion.} We use here the notations introduced in
  \autoref{lem:dualizing_embedding}. First, note that that
  \begin{equation*}
    H^0 \left(X_0', q \left(  K_{X_0'/Y_0'} + \left(D' + f_*' \tau^* M\right)|_{X_0'}
      \right) \right) = \kappa\left(X_{\can}', q\left( K_{X_{\can}'/Y_0'} + D_{\can}'
        + f_{\can}'^* \tau^* M \right) \right) , 
  \end{equation*}
  Hence, by \autoref{eq:Kodaira_dimension_estimate_1} there is an effective divisior
  $B'$ on $X'$ supported in $X' \setminus X_0'$, such that
  \begin{equation}
    \label{eq:Kodaira_dimension_on_X_prime}
    \kappa( K_{X'/Y'} + D' + f_*' \tau^* M + B' ) \geq n  + \max \{  \var f_{\can},
    \kappa(M) \}. q
  \end{equation}
  Note that $ K_{X'/Y'} + D' + f_*' \tau^* M + B'$ might not be a $\bQ$-Cartier
  divisor, so what we mean by the above statement is that for some $q>0$ divisible
  enough $q$-times multiple of this divisor defines a rational map with image of
  dimension at least $n + \max \{ \var f_{\can}, \kappa(M) \}$.
 
  Let $\gamma : X' \to X_n$ be the induced map and let $\xymatrix{X_n \ar[r]^{\xi} &
    X_n^s \ar[r]^{\zeta} & X}$ be the Stein-factorization of $\alpha \circ
  \beta$. Then
  \begin{equation}
    \label{eq:subadditivity_penultimate}
    \begin{aligned}
      \kappa \left( \zeta^* (K_{X/Y} + D + f^* M) + \xi_* \gamma_* B' \right)
      \geq \underbrace{\kappa \left( \beta^* \alpha^* (K_{X/Y} + D + f^* M) +
          \gamma_* B' \right) }_{\textrm{$\xi$ is birational}}
      \geq 
      \makebox[.05\columnwidth]{}  
      \\ \geq \underbrace{\kappa \left( K_{X_n/Y'} + \beta^* \alpha^* D + \beta^*
          \alpha^* f^* M + \gamma_* B' \right)}_{
        \textrm{\autoref{lem:dualizing_embedding}}}
      =
      \makebox[.3\columnwidth]{}  \\
      = \underbrace{\kappa \left( \gamma_* \left( K_{X'/Y'} + D' + f_*' \tau^* M + B'
          \right)\right)}_{\textrm{$D' \leq \rho^* D$, and $\gamma$ is birational}}
      \geq \underbrace{n + \max \{ \var f_{\can}, \kappa(M) \}}_{\textrm{
          \autoref{eq:Kodaira_dimension_on_X_prime}}}.
    \end{aligned}
  \end{equation}
  Choose now an effective divisor $B$ with support in $X \setminus f^{-1} Y_0$, such
  that $\zeta^* B \geq \xi_* \gamma_* B'$ (recall, $\xi_* \gamma_* B'$ is disjoint
  from $\zeta^{-1} f^{-1} Y_0$). By \autoref{eq:subadditivity_penultimate}, for this
  choice of $B$
  \begin{equation*}
    \kappa(  \zeta^* (K_{X/Y} +  D + f^* M  + B) ) \geq n  + \max \{  \var f_{\can},
    \kappa(M) \} =  
    \kappa(K_{X_\eta} + D_{\eta} ) + \max \{  \var f_{\can},  \kappa(M) \}. 
  \end{equation*}
  Hence, since Kodaira-dimension of a line bundle is invariant under finite pullback,
  \autoref{eq:subadditivity_goal_2} holds.
\end{proof}
\addtocounter{theorem}{1}


\section{
  Almost proper bases}
\label{sec:almost_results}

\begin{lemma}
  \label{lem:deducing_almost_projective}
  Consider the following commutative diagram of normal, irreducible varieties, where
  \begin{enumerate}
  \item $\oY$ and $\oY'$ are projective over $k$
  \item $\tau$ is generically finite,
  \item $Y'= \tau^{-1} Y$,
  \item $Y$ is a big open set of $\oY$,
  \item there are vector bundles $\sF$ and $\sG$ given on $Y$ and $\oY'$
    respectively, such that $\sG$ is big and $\tau^* \sF \simeq \sG|_{Y'}$.
  \end{enumerate}
  \begin{equation*}
    \xymatrix{%
      Y' \ar@{^(->}[r] \ar[d]^{\tau} & \oY' \ar[d]^{\overline{\tau}} \\
      Y \ar@{^(->}[r] &  \oY
    }
  \end{equation*}
  Then $\sF$ is big as well.
\end{lemma}

\begin{proof}
  Choose ample line bundles $\sH$ and $\sA$ on $\oY$ and $\oY'$, respectively. Let
  $b>0$ be an integers such that there is an injection $\overline{\tau}^* \sH
  \hookrightarrow \sA^b$. Since $\sG$ is big, there is an integer $a>0$ such that
  $\Sym^{a}(\sG) \otimes \sA^{-1}$ is generically globally generated. Hence, so is
  $\Sym^{ab}(\sG) \otimes \sA^{-b}$. So, by the embedding $\Sym^{ab}(\sG) \otimes
  \sA^{-b} \hookrightarrow \Sym^{ab}(\sG) \otimes \overline{\tau}^* \sH^{-1}$, the
  latter sheaf is generically globally as well. In particular, so is
  \begin{equation*}
    \left. \Sym^{ab}(\sG) \otimes \overline{\tau}^* \sH^{-1} \right|_{Y'} \simeq
    \Sym^{ab}(\tau^*\sF) \otimes \tau^* \sH|_Y^{-1} 
  \end{equation*}
  Let
  \begin{equation*}
    \xymatrix{%
      Y' \ar@/^2pc/[rr]^\tau \ar[r]^\nu &   \ar[r]^\rho Z & Y
    } 
  \end{equation*}
  be the Stein factorization of $\tau$. Then since $\nu$ is birational,
  \begin{equation*}
    \nu_* \left( \Sym^{ab}(\tau^*\sF) \otimes \tau^* \sH|_Y^{-1} \right) \simeq
    \Sym^{ab}(\rho^*\sF) \otimes \rho^* \sH|_Y^{-1} 
  \end{equation*}
  is also generically globally generated. Then \cite[Lem
  1.3]{Viehweg_Zuo_Base_spaces_of_non_isotrivial_families_of_smooth_minimal_models}
  shows that $\Sym^{ab}(\sF) \otimes \sH^{-1}|_Y$ is generically globally generated,
  and hence $\sF$ is big indeed.
\end{proof}

Using \autoref{lem:deducing_almost_projective} and
\autoref{cor:extending_stable_log_families} immediately follow versions of point
\autoref{egy} of \autoref{thm:big_higher_dim_base} and of \autoref{thm:pushforward}
for the almost projective base case.

\begin{corollary}
  If $f : (X, D) \to Y$ is a family of stable log-varieties of maximal variation over
  a normal almost projective variety, then
  \begin{enumerate}
  \item for every divisible enough $q>0$, $\det f_* \sO_X(q (K_X + \Delta))$ is big.
  \item $f_* \sO_X(q(K_{X/Y}+D))$ is big for every divisible enough integer $q>0$,
    provided that $(X,D)$ has klt general fibers over $Y$.
  \end{enumerate}
\end{corollary}


\bibliographystyle{skalpha}
\bibliography{includeNice}

\end{document}